\documentclass[reqno,a4paper,11pt]{amsart}

\usepackage{mystyle}
\usepackage{etoolbox}

\usepackage[dvipsnames]{xcolor}
\definecolor{highlightArrow}{rgb}{0.33, 0.25, 0.85} 
\definecolor{globjects}{rgb}{0.0,0.6,0.2}
\definecolor{functorL}{rgb}{0.7,0.0,0.0}
\definecolor{functorM}{rgb}{0.0,0.2,0.4}
\definecolor{functorF1}{rgb}{0.0,0.5,0.2}
\definecolor{functorF2}{rgb}{0.5,0.0,0.6}
\definecolor{functorF3}{rgb}{0.75,0.6,0.0}

\usepackage[headings]{fullpage}
\usepackage{parskip}

\usepackage[utf8]{inputenc}
\usepackage[T1]{fontenc}
\usepackage[british]{babel}
\usepackage[pdfencoding=auto]{hyperref}
\hypersetup{colorlinks=true, linkcolor=blue!30!black, citecolor=green!30!black, urlcolor=red!45!black}

\usepackage{microtype}
\usepackage{booktabs}

\usepackage{amsfonts,amssymb,bbm,mathrsfs,stmaryrd}
\usepackage{amsmath}

\usepackage{mathtools}
\usepackage[noabbrev,capitalize]{cleveref}
\usepackage{centernot}
\usepackage[shortlabels]{enumitem}
\usepackage{url}
\usepackage{tikz}
\usepackage{pict2e}
\usepackage{framed}

\usetikzlibrary{matrix,positioning,calc}
\usetikzlibrary{arrows}
\tikzset{> =stealth}

\newcommand{\addQEDstyle}[2]{\AtBeginEnvironment{#1}{\pushQED{\qed}\renewcommand{\qedsymbol}{#2}}\AtEndEnvironment{#1}{\popQED}}

\theoremstyle{plain}
\newtheorem{theorem}{Theorem}[section]
\newtheorem*{theorem*}{Theorem}
\newtheorem{lemma}[theorem]{Lemma}
\newtheorem{proposition}[theorem]{Proposition}

\theoremstyle{definition}
\newtheorem{definition}[theorem]{Definition}

\addQEDstyle{example}{$\triangle$}

\theoremstyle{remark}
\newtheorem{remark}[theorem]{Remark}

\renewcommand{\epsilon}{\varepsilon}
\renewcommand{\phi}{\varphi}

\renewcommand{\C}{\mathcal{C}}
\newcommand{\B}{\mathcal{B}}
\newcommand{\D}{\mathcal{D}}
\newcommand{\Laxop}{\mathrm{Lax}_\mathrm{op}}

\newcommand{\op}{^\mathrm{op}}

\newcommand{\inv}{^{-1}}

\newcommand{\id}{\mathrm{id}}
\newcommand{\Id}{\mathrm{Id}}
\newcommand{\Dist}{\mathrm{Dist}}

\newcommand{\Cat}{\mathrm{Cat}}

\newcommand{\Hom}{\mathrm{Hom}}

\newcommand{\Ext}{\mathrm{Ext}}

\usetikzlibrary{backgrounds}
\makeatletter
\tikzset{
on layer/.style={
    execute at begin scope={%
      \pgfonlayer{#1}%
      \tikz@options
      },
    execute at end scope={\endpgfonlayer}
  }
}%
\makeatother
\tikzset{every picture/.append style={semithick},
         every scope/.append style={semithick}}

\pgfdeclarelayer{bg}
\pgfdeclarelayer{over}
\pgfsetlayers{bg,main,over}

\tikzset{dot/.style={circle,draw=black,fill=black,minimum size=1mm,inner sep=0mm}}
\tikzset{cross line/.style={preaction={draw=white, -,line width=6pt}}}

\title{2-dimensional bifunctor theorems and distributive laws}

\author[P. F. Faul]{Peter F. Faul}
\address{Department of Pure Mathematics and Mathematical Statistics\\ University of Cambridge}
\email{peter@faul.io}

\author[G. Manuell]{Graham Manuell\thanks{The second author acknowledges partial financial support from the Centre for Mathematics of the University of Coimbra (UIDB/00324/2020, funded by the Portuguese Government through FCT/MCTES)}}
\address{Centre for Mathematics \\ University of Coimbra}
\email{graham@manuell.me}

\author[J. Siqueira]{Jos\'e Siqueira\thanks{This study was financed in part by the Coordenação de Aperfeiçoamento de Pessoal de Nível Superior - Brasil (CAPES), which supported the CAPES scholar Jos\'e Siqueira (process n$^\circ$ 88881.128278/2016-01). }}
\address{Department of Pure Mathematics and Mathematical Statistics\\ University of Cambridge}
\email{jose.siqueira@cantab.net}

\date{\today}


\begin{document}

\maketitle

\begin{abstract}
    In this paper we consider the conditions that need to be satisfied by two families of pseudofunctors with a common codomain for them to be collated into a bifunctor. We observe similarities between these conditions and distributive laws of monads before providing a unified framework from which both of these results may be inferred. We do this by proving a version of the bifunctor theorem for lax functors.
    We then show that these generalised distributive laws may be arranged into a 2-category $\Dist(\B,\C,\D)$, which is equivalent to $\Laxop(\B,\Laxop(\C,\D))$. The collation of a distributive law into its associated bifunctor extends to a 2-functor into $\Laxop(\B \times \C, \D)$, which corresponds to uncurrying via the aforementioned equivalence.
    We also describe subcategories on which collation itself restricts to an equivalence. Finally, we exhibit a number of natural categorical constructions as special cases of our result.
\end{abstract}

\section{Introduction}

Fix categories $\B$, $\C$ and $\D$ and consider a bifunctor $P\colon \B \times \C \to \D$.
By holding each of the inputs of $P$ constant in turn we obtain functors $M_B = P(B,-)$ and $L_C = P(-,C)$ for each object $B$ in $\B$ and each object $C$ in $\C$.
It is a natural to consider when two such families of functors may be rearranged back into a bifunctor. This is the bifunctor theorem and it appears as the first proposition in \cite{maclane}.

\begin{theorem}\label{thm:Mac}
Let $\B$, $\C$ and $\D$ be categories and for each object $B$ in $\B$ and each object $C$ in $\C$ let $L_C \colon \B \to \D$ and $M_B \colon \C \to \D$ be functors such that $L_C(B) = M_B(C)$. Then there exists a bifunctor $P \colon \B \times \C \to \D$ with $L_C = P(-,C)$ and $M_B = P(B,-)$ if and only if for all morphisms $f\colon B_1 \to B_2$ in $\B$ and $g\colon C_1 \to C_2$ in $\C$ we have $L_{C_2}(f)M_{B_1}(g) = M_{B_2}(g)L_{C_1}(f)$. In this case, $P(B,C) = L_C(B) = M_B(C)$ and $P(f,g) = M_{B_2}(g)L_{C_1}(f)$.
\end{theorem}

While the result is somewhat trivial, it can be remarkably useful. For instance, one convenient use case is in defining the $\Ext$ bifunctor of abelian groups.

It is natural to ask what an analogue of this result would be in the 2-categorical setting. For families of pseudofunctors $L$ and $M$, the idea is that the equality $L_{C_2}(f)M_{B_1}(g) = M_{B_2}(g)L_{C_1}(f)$ above should be replaced with an invertible 2-morphism $\sigma$. This invertible 2-morphism is now required to satisfy a number of coherence conditions which were invisible in the classical case. If we depict the $M$ families with blue wires, the $L$ families with red wires and $\sigma$ with a braiding of a red wire over a blue, then the following string diagrams illustrate some of these conditions.

  \begin{equation*}
   \begin{tikzpicture}[scale=0.8,baseline={([yshift=-0.5ex]current bounding box.center)}]
    \begin{scope}[on layer=over]
    \path coordinate[dot, functorM] (mu)
    +(0,1.5) coordinate (d)
    +(-1,-1) coordinate (mbl)
    +(1,-1) coordinate (mbr);
    \path (d -| mbl) ++(-1.25,0) coordinate (tl) ++(0,-1.25) coordinate (tl2);
    \path (mbr) ++(0,-1.5) coordinate (br) ++(1.5,0) coordinate (brr) ++(0,0.25) coordinate (brr2);
    \path (mbl) ++(0,-1.5) coordinate (bl);
    \end{scope}
    \draw[functorM] (bl) -- (mbl) to[out=90, in=180] (mu.center) to[out=0, in=90] (mbr) -- (br)
                    (mu) -- (d);
    \draw[functorL, cross line] (tl) -- (tl2) to[out=270, in=90] (brr);
    \coordinate (cornerNW) at ($(tl) + (-0.5,0)$);
    \coordinate (cornerSE) at ($(brr) + (0.5,0)$);
    \draw (cornerNW) rectangle (cornerSE);
   \end{tikzpicture}
   \,=\,
   \begin{tikzpicture}[scale=0.8,baseline={([yshift=-0.5ex]current bounding box.center)}]
    \begin{scope}[on layer=over]
    \path coordinate[dot, functorM] (mu)
    +(0,1.75) coordinate (d)
    +(-1,-1) coordinate (mbl)
    +(1,-1) coordinate (mbr);
    \path (d -| mbl) ++(-1.25,0) coordinate (tl) ++(0,-0.25) coordinate (tl2);
    \path (mbr) ++(0,-1.25) coordinate (br) ++(1.5,0) coordinate (brr) ++(0,1.4) coordinate (brr2);
    \path (mbl) ++(0,-1.25) coordinate (bl);
    \end{scope}
    \draw[functorM] (bl) -- (mbl) to[out=90, in=180] (mu.center) to[out=0, in=90] (mbr) -- (br)
                    (mu) -- (d);
    \draw[functorL, cross line] (tl) to[out=270, in=90] (brr2) -- (brr);
    \coordinate (cornerNW) at ($(tl) + (-0.5,0)$);
    \coordinate (cornerSE) at ($(brr) + (0.5,0)$);
    \draw (cornerNW) rectangle (cornerSE);
   \end{tikzpicture}
  \end{equation*}
  
  \vspace{-3pt plus 1pt}
  \begin{equation*}
   \begin{tikzpicture}[scale=0.8,baseline={([yshift=-0.5ex]current bounding box.center)}]
    \begin{scope}[on layer=over]
    \path coordinate[dot, functorL] (mu)
    +(0,1.5) coordinate (d)
    +(1,-1) coordinate (mbl)
    +(-1,-1) coordinate (mbr);
    \path (d -| mbl) ++(1.25,0) coordinate (tl) ++(0,-1.25) coordinate (tl2);
    \path (mbr) ++(0,-1.5) coordinate (br) ++(-1.5,0) coordinate (brr) ++(0,0.25) coordinate (brr2);
    \path (mbl) ++(0,-1.5) coordinate (bl);
    \end{scope}
    \draw[functorM] (tl) -- (tl2) to[out=270, in=90] (brr);
    \draw[functorL, cross line] (bl) -- (mbl) to[out=90, in=0] (mu.center) to[out=180, in=90] (mbr) -- (br)
                                (mu) -- (d);
    \coordinate (cornerNW) at ($(tl) + (0.5,0)$);
    \coordinate (cornerSE) at ($(brr) + (-0.5,0)$);
    \draw (cornerNW) rectangle (cornerSE);
   \end{tikzpicture}
   \,=\,
   \begin{tikzpicture}[scale=0.8,baseline={([yshift=-0.5ex]current bounding box.center)}]
    \begin{scope}[on layer=over]
    \path coordinate[dot, functorL] (mu)
    +(0,1.75) coordinate (d)
    +(1,-1) coordinate (mbl)
    +(-1,-1) coordinate (mbr);
    \path (d -| mbl) ++(1.25,0) coordinate (tl) ++(0,-0.25) coordinate (tl2);
    \path (mbr) ++(0,-1.25) coordinate (br) ++(-1.5,0) coordinate (brr) ++(0,1.4) coordinate (brr2);
    \path (mbl) ++(0,-1.25) coordinate (bl);
    \end{scope}
    \draw[functorM] (tl) to[out=270, in=90] (brr2) -- (brr);
    \draw[functorL, cross line] (bl) -- (mbl) to[out=90, in=0] (mu.center) to[out=180, in=90] (mbr) -- (br)
                                (mu) -- (d);
    \coordinate (cornerNW) at ($(tl) + (0.5,0)$);
    \coordinate (cornerSE) at ($(brr) + (-0.5,0)$);
    \draw (cornerNW) rectangle (cornerSE);
   \end{tikzpicture}
  \end{equation*}

These diagrams may remind the reader of a distributive laws of monads or perhaps of oplax and lax transformations. We will prove a \emph{lax} bifunctor theorem which recovers both the pseudo-bifunctor theorem and distributive laws of monads as special cases. For this reason we call this data a distributive law of lax functors.
Other consequences of the result will be discussed at the end of the paper.

The link to (op)lax transformations is explained by a correspondence between this data and objects of the 2-category $\Laxop(\B,\Laxop(\C,\D))$. This suggests notions of 1-morphism and 2-morphism of distributive laws giving a category $\Dist(\B,\C,\D)$ which is equivalent to $\Laxop(\B,\Laxop(\C,\D))$. The process of `collating' a distributive law into a lax bifunctor extends to a 2-functor from $\Dist(\B,\C,\D)$ to $\Laxop(\B \times \C,\D)$. Moreover, this collation 2-functor can be seen to correspond to uncurrying via the equivalence between $\Dist(\B,\C,\D)$ and $\Laxop(\B,\Laxop(\C,\D))$.

\section{Background}

\subsection*{String diagrams}

We will make extensive use of string diagrams in this paper. For an introduction to string diagrams see \cite{marsden2014category}. Our string diagrams will be read from bottom to top (for vertical composition) and left to right (for horizontal composition). We will not colour the regions of the diagrams, which should be clear from context. However, as explained later we will colour wires when we want to indicate that a lax functor has been applied. Our approach to representing lax functors can be compared to the string diagram calculus for pseudofunctors described in \cite{verdon2020unitary}.

\subsection*{2-categories and lax functors}

By a \emph{(strict) 2-category} we mean any category enriched over the cartesian monoidal category $\Cat$. We may think of a 2-category as consisting of \emph{objects}, \emph{1-morphisms} and \emph{2-morphisms}. Here 1-morphisms compose like functors, and we denote this by juxtaposition, and 2-morphisms compose like natural transformations --- that is to say, they may be composed vertically, which we denote with $\circ$ (or juxtaposition) or horizontally, which we denote with $\ast$.

There are a number of ways to generalise functors to this setting and we shall concern ourselves with two of them: lax functors and pseudofunctors.

\begin{definition}
  A \emph{lax functor} $\mathcal{F}$ between 2-categories $\C$ and $\D$ consists of a function $\mathcal{F}$ sending objects $C$ in $\C$ to objects $\mathcal{F}(C)$ in $\D$ and for each pair of objects $C_1$ and $C_2$ in $\C$ a functor $\mathcal{F}_{C_1,C_2} \colon \Hom(C_1,C_2) \to \Hom(\mathcal{F}(C_1),\mathcal{F}(C_2))$, for which we use the same name. 
  Additionally, for each pair of composable 1-morphisms $(f,g)$ we have a 2-morphism $\gamma_{g,f} \colon \mathcal{F}(g) \circ \mathcal{F}(f) \to \mathcal{F}(g \circ f)$ called the \emph{compositor}, and for each object $A$ in $\C$ we have a 2-morphism $\iota_A \colon \id_{\mathcal{F}(A)} \to \mathcal{F}(\id_A)$ called the \emph{unitor}. This data satisfies the following constraints.
  
  \begin{enumerate}
    \item Let $\alpha\colon f_1 \to f_2$ and $\beta\colon g_1 \to g_2$ be horizontally composable 2-morphisms. Then the compositors must satisfy the following naturality condition.
    
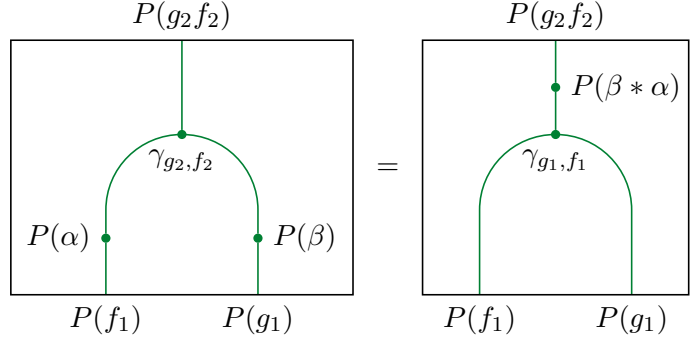
\begin{equation*}
   \begin{tikzpicture}[scale=1.0,baseline={([yshift=-0.5ex]current bounding box.center)}]
    \begin{scope}[on layer=over]
    \path coordinate[dot, functorF1, label=below:$\gamma_{g_2,f_2}$] (mu)
    +(-1,-1) coordinate (mbl)
    +(1,-1) coordinate (mbr)
    ++(0,0.625) coordinate (du2) ++(0,0.625) coordinate[label=above:$P(
    g_2f_2)$] (du);
    \path (mbl) ++(0,-0.375) coordinate[dot, functorF1, label=left:$P(\alpha)$] (bl2) ++(0,-0.75) coordinate[label=below:$P(f_1)$] (bl);
    \path (mbr) ++(0,-0.375) coordinate[dot, functorF1, label=right:$P(\beta)$] (br2) ++(0,-0.75) coordinate[label=below:$P(g_1)$] (br);
    \end{scope}
    \draw[functorF1] (bl) -- (bl2) -- (mbl) to[out=90, in=180] (mu.center) to[out=0, in=90] (mbr) -- (br2) -- (br)
                    (mu) -- (du2) -- (du);
    \coordinate (tl) at (mbl |- du);
    \coordinate (cornerNW) at ($(tl) + (-1.25,0)$);
    \coordinate (cornerSE) at ($(br) + (1.25,0)$);
    \draw (cornerNW) rectangle (cornerSE);
   \end{tikzpicture}
   \enspace=\enspace
   \begin{tikzpicture}[scale=1.0,baseline={([yshift=-0.5ex]current bounding box.center)}]
    \begin{scope}[on layer=over]
    \path coordinate[dot, functorF1, label=below:$\gamma_{g_1,f_1}$] (mu)
    +(-1,-1) coordinate (mbl)
    +(1,-1) coordinate (mbr)
    ++(0,0.625) coordinate[dot, functorF1, label=right:$P(\beta \ast \alpha)$] (du2) ++(0,0.625) coordinate[label=above:$P(g_2f_2)$] (du);
    \path (mbl) ++(0,-0.375) coordinate (bl2) ++(0,-0.75) coordinate[label=below:$P(f_1)$] (bl);
    \path (mbr) ++(0,-0.375) coordinate (br2) ++(0,-0.75) coordinate[label=below:$P(
    g_1)$] (br);
    \end{scope}
    \draw[functorF1] (bl) -- (bl2) -- (mbl) to[out=90, in=180] (mu.center) to[out=0, in=90] (mbr) -- (br2) -- (br)
                    (mu) -- (du2) -- (du);
    \coordinate (tl) at (mbl |- du);
    \coordinate (cornerNW) at ($(tl) + (-0.75,0)$);
    \coordinate (cornerSE) at ($(br) + (0.75,0)$);
    \draw (cornerNW) rectangle (cornerSE);
   \end{tikzpicture}
   \end{equation*}
    
    \item If $f \colon X \to Y$, $g\colon Y \to Z$ and $h\colon Z \to W$ are 1-morphisms, then the following associativity axiom must be satisfied.
      
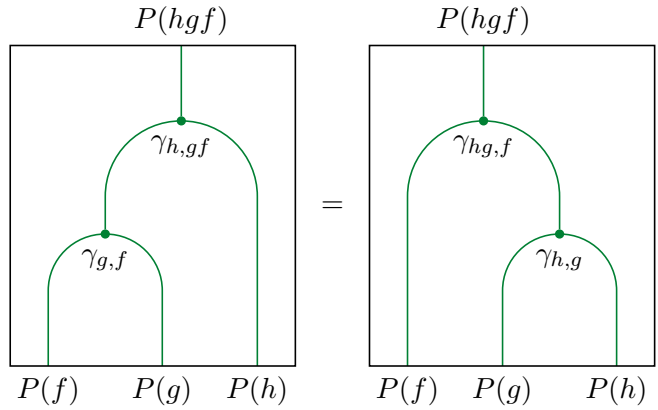
\begin{equation*}
   \begin{tikzpicture}[scale=1.0,baseline={([yshift=-0.5ex]current bounding box.center)}]
    \begin{scope}[on layer=over]
    \path coordinate[dot, functorF1, label=below:$\gamma_{h,gf}$] (muU)
    +(0,1.0) coordinate[label=above:$P(hgf)$] (dU)
    +(-1,-1) coordinate (mblU)
    +(1,-1) coordinate (mbrU);
    \path (mblU) ++(0,-0.5) coordinate[dot, functorF1, label=below:$\gamma_{g,f}$] (mu)
    +(-0.75,-0.75) coordinate (mbl)
    +(0.75,-0.75) coordinate (mbr);
    \path (mbr) ++(0,-1.0) coordinate[label=below:$P(g)$] (br);
    \path (mbl) ++(0,-1.0) coordinate[label=below:$P(f)$] (bl);
    \coordinate[label=below:$P(h)$] (brr) at (mbrU |- br);
    \end{scope}
    \draw[functorF1] (bl) -- (mbl) to[out=90, in=180] (mu.center) to[out=0, in=90] (mbr) -- (br)
                    (mu) -- (mblU) to[out=90, in=180] (muU.center) to[out=0, in=90] (mbrU) -- (brr)
                    (muU) -- (dU);
    \coordinate (tl) at (mbl |- dU);
    \coordinate (cornerNW) at ($(tl) + (-0.5,0)$);
    \coordinate (cornerSE) at ($(brr) + (0.5,0)$);
    \draw (cornerNW) rectangle (cornerSE);
   \end{tikzpicture}
   \enspace=\enspace
   \begin{tikzpicture}[scale=1.0,baseline={([yshift=-0.5ex]current bounding box.center)}]
    \begin{scope}[on layer=over]
    \path coordinate[dot, functorF1, label=below:$\gamma_{hg,f}$] (muU)
    +(0,1.0) coordinate[label=above:$P(hgf)$] (dU)
    +(-1,-1) coordinate (mblU)
    +(1,-1) coordinate (mbrU);
    \path (mbrU) ++(0,-0.5) coordinate[dot, functorF1, label=below:$\gamma_{h,g}$] (mu)
    +(-0.75,-0.75) coordinate (mbl)
    +(0.75,-0.75) coordinate (mbr);
    \path (mbr) ++(0,-1.0) coordinate[label=below:$P(h)$] (br);
    \path (mbl) ++(0,-1.0) coordinate[label=below:$P(g)$] (bl);
    \coordinate[label=below:$P(f)$] (bll) at (mblU |- bl);
    \end{scope}
    \draw[functorF1] (bl) -- (mbl) to[out=90, in=180] (mu.center) to[out=0, in=90] (mbr) -- (br)
                    (mu) -- (mbrU) to[out=90, in=0] (muU.center) to[out=180, in=90] (mblU) -- (bll)
                    (muU) -- (dU);
    \coordinate (tr) at (mbr |- dU);
    \coordinate (cornerNE) at ($(tr) + (0.5,0)$);
    \coordinate (cornerSW) at ($(bll) + (-0.5,0)$);
    \draw (cornerNE) rectangle (cornerSW);
   \end{tikzpicture}
  \end{equation*}
  
    \item If $f \colon X \to Y$ is a 1-morphism then the unit axioms
    must be satisfied.
    
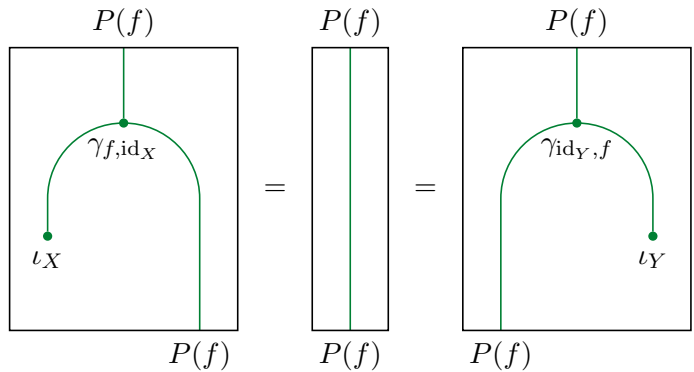
\begin{equation*} 
   \begin{tikzpicture}[scale=1.0,baseline={([yshift=-0.5ex]current bounding box.center)}]
    \begin{scope}[on layer=over]
    \path coordinate[dot, functorF1, label=below:$\gamma_{f,\id_X}$] (muU)
    +(0,1.0) coordinate[label=above:$P(f)$] (dU)
    +(-1,-1) coordinate (mblU)
    +(1,-1) coordinate (mbrU);
    \path (mblU) ++(0,-0.5) coordinate[dot, functorF1, label=below:$\iota_X$] (mu);
    \path (mbrU) ++ (0,-1.75) coordinate[label=below:$P(f)$] (brr);
    \end{scope}
    \draw[functorF1] (mu) -- (mblU) to[out=90, in=180] (muU.center) to[out=0, in=90] (mbrU) -- (brr)
                    (muU) -- (dU);
    \coordinate (tl) at (mblU |- dU);
    \coordinate (cornerNW) at ($(tl) + (-0.5,0)$);
    \coordinate (cornerSE) at ($(brr) + (0.5,0)$);
    \draw (cornerNW) rectangle (cornerSE);
   \end{tikzpicture}
   \enspace=\enspace
   \begin{tikzpicture}[scale=1.0,baseline={([yshift=-0.5ex]current bounding box.center)}]
    \begin{scope}[on layer=over]
    \path coordinate (muU)
    +(0,1.0) coordinate[label=above:$P(f)$] (dU)
    +(0,-2.75) coordinate[label=below:$P(f)$] (b);
    \end{scope}
    \draw[functorF1] (dU) -- (b);
    \coordinate (cornerNW) at ($(dU) + (-0.5,0)$);
    \coordinate (cornerSE) at ($(b) + (0.5,0)$);
    \draw (cornerNW) rectangle (cornerSE);
   \end{tikzpicture}
   \enspace=\enspace
   \begin{tikzpicture}[scale=1.0,baseline={([yshift=-0.5ex]current bounding box.center)}]
    \begin{scope}[on layer=over]
    \path coordinate[dot, functorF1, label=below:$\gamma_{\id_Y,f}$] (muU)
    +(0,1.0) coordinate[label=above:$P(f)$] (dU)
    +(1,-1) coordinate (mbrU)
    +(-1,-1) coordinate (mblU);
    \path (mbrU) ++(0,-0.5) coordinate[dot, functorF1, label=below:$\iota_Y$] (mu);
    \path (mblU) ++ (0,-1.75) coordinate[label=below:$P(f)$] (bll);
    \end{scope}
    \draw[functorF1] (mu) -- (mbrU) to[out=90, in=0] (muU.center) to[out=180, in=90] (mblU) -- (bll)
                    (muU) -- (dU);
    \coordinate (tr) at (mbrU |- dU);
    \coordinate (cornerNE) at ($(tr) + (0.5,0)$);
    \coordinate (cornerSW) at ($(bll) + (-0.5,0)$);
    \draw (cornerNE) rectangle (cornerSW);
   \end{tikzpicture}
  \end{equation*}
  \end{enumerate}
  
  When all $\gamma$ and $\iota$ maps are isomorphisms, then we call $F$ a \emph{pseudofunctor}. When just the $\iota$ maps are isomorphisms we call $F$ a \emph{unitary lax functor}.
\end{definition}

There is a notion of an oplax transformation between two lax functors defined as follows. 

\begin{definition}
Let $(\mathcal{F}_1, \gamma_1 ,\iota_1), (\mathcal{F}_2, \gamma_2, \iota_2)\colon \mathcal{X} \to \mathcal{Y}$ be two lax functors. An oplax transformation $\rho \colon \mathcal{F}_1 \to \mathcal{F}_2$ is given by two families:

\begin{itemize}
  \item For each object $X \in \mathcal{X}$, a 1-morphism $\rho_X \colon \mathcal{F}_1(X) \to \mathcal{F}_2(X)$.
  \item For each 1-morphism $g\colon X \to Y \in \mathcal{X}$, a 2-morphism $\rho_g \colon \rho_Y \mathcal{F}_1(g) \to \mathcal{F}_2(g) \rho_X$.
\end{itemize}

They must satisfy a number of coherence conditions, which we express below in terms of string diagrams.
For ease of understanding we use green to represent morphisms in the image of $\mathcal{F}_1$ (and associated data) and purple for morphisms in the image of $\mathcal{F}_2$. The 2-morphisms $\rho_g$ will be represented by a crossing of a wire representing $\rho_Y$ and $\rho_X$ over one representing $\mathcal{F}_1(g)$ and $\mathcal{F}_2(g)$.
\begin{enumerate}
\item If $g\colon X \to Y$ and $f\colon Y \to Z$ are 1-morphisms in $\mathcal{X}$, then $\rho$ must respect composition as shown in the following string diagram.
\\ 
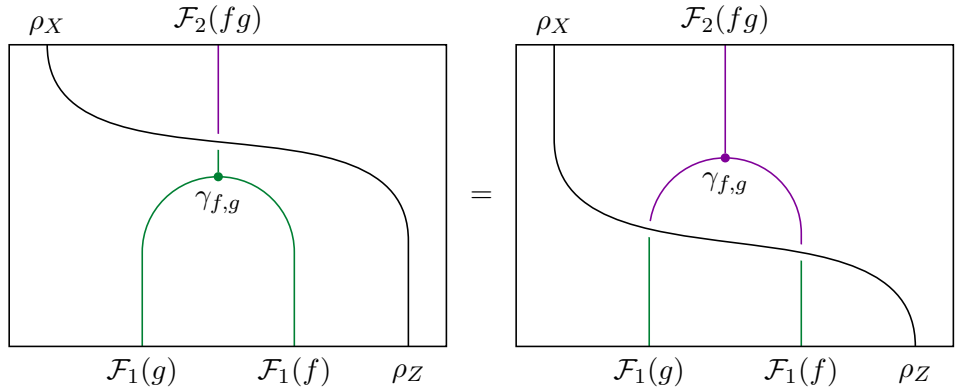
\vspace{-3pt plus 1pt}
\begin{equation*}
 \begin{tikzpicture}[scale=1.0,baseline={([yshift=-0.5ex]current bounding box.center)}]
  \begin{scope}[on layer=over]
  \path coordinate[dot, functorF1, label=below:$\gamma_{f,g}$] (mu)
  +(0,1.75) coordinate[label=above:$\mathcal{F}_2(fg)$] (d)
  +(-1,-1) coordinate (mbl)
  +(1,-1) coordinate (mbr);
  \path (d -| mbl) ++(-1.25,0) coordinate[label=above:$\rho_X$] (tl) ++(0,-0.25) coordinate (tl2);
  \path (mbr) ++(0,-1.25) coordinate[label=below:$\mathcal{F}_1(f)$] (br) ++(1.5,0) coordinate[label=below:$\vphantom{\mathcal{F}}\rho_Z$] (brr) ++(0,1.4) coordinate (brr2);
  \path (mbl) ++(0,-1.25) coordinate[label=below:$\mathcal{F}_1(g)$] (bl);
  \end{scope}
  \newcommand{\crossing}{(tl) to[out=270, in=90] (brr2) -- (brr)}
  \newcommand{\fork}{
   (bl) -- (mbl) to[out=90, in=180] (mu.center) to[out=0, in=90] (mbr) -- (br)
   (mu) -- (d)
  }
  \coordinate (cornerNW) at ($(tl) + (-0.5,0)$);
  \coordinate (cornerSE) at ($(brr) + (0.5,0)$);
  \begin{scope}
  \clip \crossing -- (cornerSE) -- (cornerNW -| cornerSE) -- cycle;
  \draw[functorF2] \fork;
  \end{scope}
  \begin{scope}
  \clip (cornerNW) -- \crossing -- (cornerNW |- cornerSE) -- cycle;
  \draw[functorF1] \fork;
  \end{scope}
  \draw[cross line] \crossing;
  \draw (cornerNW) rectangle (cornerSE);
 \end{tikzpicture}
 \enspace =\enspace
 \begin{tikzpicture}[scale=1.0,baseline={([yshift=-0.5ex]current bounding box.center)}]
  \begin{scope}[on layer=over]
  \path coordinate[dot, functorF2, label=below:$\gamma_{f,g}$] (mu)
  +(0,1.5) coordinate[label=above:$\mathcal{F}_2(fg)$] (d)
  +(-1,-1) coordinate (mbl)
  +(1,-1) coordinate (mbr);
  \path (d -| mbl) ++(-1.25,0) coordinate[label=above:$\rho_X$] (tl) ++(0,-1.25) coordinate (tl2);
  \path (mbr) ++(0,-1.5) coordinate[label=below:$\mathcal{F}_1(f)$] (br) ++(1.5,0) coordinate[label=below:$\vphantom{\mathcal{F}}\rho_Z$] (brr) ++(0,0.25) coordinate (brr2);
  \path (mbl) ++(0,-1.5) coordinate[label=below:$\mathcal{F}_1(g)$] (bl);
  \end{scope}
  \newcommand{\crossing}{(tl) -- (tl2) to[out=270, in=90] (brr)}
  \newcommand{\fork}{
   (bl) -- (mbl) to[out=90, in=180] (mu.center) to[out=0, in=90] (mbr) -- (br)
   (mu) -- (d)
  }
  \coordinate (cornerNW) at ($(tl) + (-0.5,0)$);
  \coordinate (cornerSE) at ($(brr) + (0.5,0)$);
  \begin{scope}
  \clip \crossing -- (cornerSE) -- (cornerNW -| cornerSE) -- cycle;
  \draw[functorF2] \fork;
  \end{scope}
  \begin{scope}
  \clip (cornerNW) -- \crossing -- (cornerNW |- cornerSE) -- cycle;
  \draw[functorF1] \fork;
  \end{scope}
  \draw[cross line] \crossing;
  \draw (cornerNW) rectangle (cornerSE);
 \end{tikzpicture}
\end{equation*}
  
\item For each object $X \in \mathcal{X}$, $\rho$ must respect the identity $\id_X$. We express this with the following string diagram.
\\ 
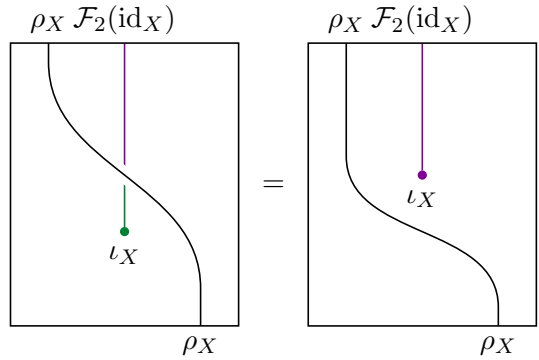
\vspace{-3pt plus 1pt}
\begin{equation*}
 \begin{tikzpicture}[scale=1.0,baseline={([yshift=-0.5ex]current bounding box.center)}]
  \begin{scope}[on layer=over]
  \path coordinate (mu)
  +(0,1.5) coordinate[label=above:$\vphantom{\rho_X}\smash{\mathcal{F}_2(\id_X)}$] (d)
  ++(0,-1.0) coordinate[dot, functorF1, label=below:$\iota_X$] (mb) ++(0,-1.25) coordinate (b);
  \path (d) ++(-1.0,0) coordinate[label=above:$\rho_X$] (tl) ++(0,-0.25) coordinate (tl2);
  \path (b) ++ (1.0,0) coordinate[label=below:$\rho_X$] (brr) ++(0,0.5) coordinate (rr);
  \end{scope}
  \newcommand{\crossing}{(tl) -- (tl2) to[out=270, in=90] (rr) -- (brr)}
  \newcommand{\fork}{
   (d) -- (mu) -- (mb)
  }
  \coordinate (cornerNW) at ($(tl) + (-0.5,0)$);
  \coordinate (cornerSE) at ($(brr) + (0.5,0)$);
  \begin{scope}
  \clip \crossing -- (cornerSE) -- (cornerNW -| cornerSE) -- cycle;
  \draw[functorF2] \fork;
  \end{scope}
  \begin{scope}
  \clip (cornerNW) -- \crossing -- (cornerNW |- cornerSE) -- cycle;
  \draw[functorF1] \fork;
  \end{scope}
  \draw[cross line] \crossing;
  \draw (cornerNW) rectangle (cornerSE);
 \end{tikzpicture}
 \enspace =\enspace
 \begin{tikzpicture}[scale=1.0,baseline={([yshift=-0.5ex]current bounding box.center)}]
  \begin{scope}[on layer=over]
  \path coordinate (mu)
  +(0,0.75) coordinate[label=above:$\vphantom{\rho_X}\smash{\mathcal{F}_2(\id_X)}$] (d)
  ++(0,-1.0) coordinate[dot, functorF2, label=below:$\iota_X$] (mb) ++(0,-2.0) coordinate (b);
  \path (d) ++(-1.0,0) coordinate[label=above:$\rho_X$] (tl) ++(0,-1.5) coordinate (tl2);
  \path (b) ++ (1.0,0) coordinate[label=below:$\rho_X$] (brr) ++(0,0.25) coordinate (rr);
  \end{scope}
  \newcommand{\crossing}{(tl) -- (tl2) to[out=270, in=90] (rr) -- (brr)}
  \newcommand{\fork}{
   (d) -- (mu) -- (mb)
  }
  \coordinate (cornerNW) at ($(tl) + (-0.5,0)$);
  \coordinate (cornerSE) at ($(brr) + (0.5,0)$);
  \draw[functorF2] \fork;
  \draw[cross line] \crossing;
  \draw (cornerNW) rectangle (cornerSE);
 \end{tikzpicture}
\end{equation*}

\item The family of maps $\rho_g$ is natural in $g$ in the sense that for each 1-morphisms $g,g' \colon X \to Y$ and 2-morphism $\alpha\colon g \to g'$ the following equality of string diagrams holds.
\\ 
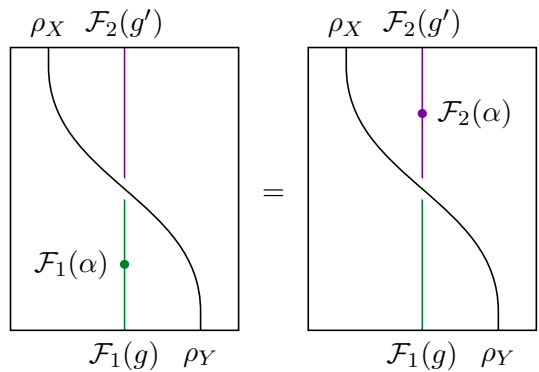
\vspace{-3pt plus 1pt}
\begin{equation*}
 \begin{tikzpicture}[scale=1.0,baseline={([yshift=-0.5ex]current bounding box.center)}]
  \begin{scope}[on layer=over]
  \path coordinate (mu)
  +(0,0.875) coordinate[label=above:$\vphantom{\rho_X}\smash{\mathcal{F}_2(g')}$] (d)
  ++(0,-2.0) coordinate[dot, functorF1, label=left:$\mathcal{F}_1(\alpha)$] (mb) ++(0,-0.875) coordinate[label=below:$\mathcal{F}_1(g)$] (b);
  \path (d) ++(-1.0,0) coordinate[label=above:$\rho_X$] (tl) ++(0,-0.25) coordinate (tl2);
  \path (b) ++ (1.0,0) coordinate[label=below:$\vphantom{\mathcal{F}}\rho_Y$] (brr) ++(0,0.25) coordinate (rr);
  \end{scope}
  \newcommand{\crossing}{(tl) -- (tl2) to[out=270, in=90] (rr) -- (brr)}
  \newcommand{\fork}{
   (d) -- (mu) -- (mb) -- (b)
  }
  \coordinate (cornerNW) at ($(tl) + (-0.5,0)$);
  \coordinate (cornerSE) at ($(brr) + (0.5,0)$);
  \begin{scope}
  \clip \crossing -- (cornerSE) -- (cornerNW -| cornerSE) -- cycle;
  \draw[functorF2] \fork;
  \end{scope}
  \begin{scope}
  \clip (cornerNW) -- \crossing -- (cornerNW |- cornerSE) -- cycle;
  \draw[functorF1] \fork;
  \end{scope}
  \draw[cross line] \crossing;
  \draw (cornerNW) rectangle (cornerSE);
 \end{tikzpicture}
 \enspace =\enspace
 \begin{tikzpicture}[scale=1.0,baseline={([yshift=-0.5ex]current bounding box.center)}]
  \begin{scope}[on layer=over]
  \path coordinate[dot, functorF2, label=right:$\mathcal{F}_2(\alpha)$] (mu)
  +(0,0.875) coordinate[label=above:$\vphantom{\rho_X}\smash{\mathcal{F}_2(g')}$] (d)
  ++(0,-2.0) coordinate (mb) ++(0,-0.875) coordinate[label=below:$\mathcal{F}_1(g)$] (b);
  \path (d) ++(-1.0,0) coordinate[label=above:$\rho_X$] (tl) ++(0,-0.25) coordinate (tl2);
  \path (b) ++ (1.0,0) coordinate[label=below:$\vphantom{\mathcal{F}}\rho_Y$] (brr) ++(0,0.25) coordinate (rr);
  \end{scope}
  \newcommand{\crossing}{(tl) -- (tl2) to[out=270, in=90] (rr) -- (brr)}
  \newcommand{\fork}{
   (d) -- (mu) -- (mb) -- (b)
  }
  \coordinate (cornerNW) at ($(tl) + (-0.5,0)$);
  \coordinate (cornerSE) at ($(brr) + (0.5,0)$);
  \begin{scope}
  \clip \crossing -- (cornerSE) -- (cornerNW -| cornerSE) -- cycle;
  \draw[functorF2] \fork;
  \end{scope}
  \begin{scope}
  \clip (cornerNW) -- \crossing -- (cornerNW |- cornerSE) -- cycle;
  \draw[functorF1] \fork;
  \end{scope}
  \draw[cross line] \crossing;
  \draw (cornerNW) rectangle (cornerSE);
 \end{tikzpicture}
\end{equation*}
\end{enumerate}

If each 2-morphism $\rho_f$ is an isomorphism, we call $\rho$ a pseudonatural transformation.
A \emph{lax} transformation is similar to an oplax transformation, but has $\rho_f$ going in the opposite direction.

The composition of oplax transformations $\rho^1\colon \mathcal{F}_1 \to \mathcal{F}_2$ and $\rho^2\colon \mathcal{F}_2 \to \mathcal{F}_3$ has $(\rho^2\rho^1)_X = \rho^2_X \rho^1_X$ and $(\rho^2\rho^1)_f = \rho^2_f \rho^1_X \circ \rho^2_Y \rho^1_f$.
This is shown in the diagram below.
\\ 
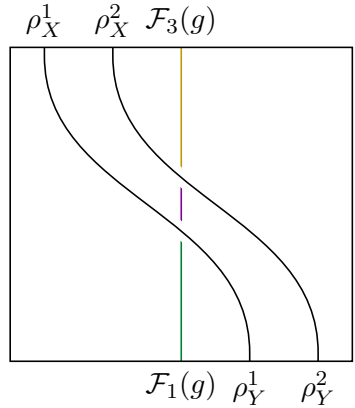
\vspace{-3pt plus 1pt}
\begin{center}
 \begin{tikzpicture}[scale=0.9,baseline={([yshift=-0.5ex]current bounding box.center)}]
  \begin{scope}[on layer=over]
  \path coordinate (mu)
  +(0,0.875) coordinate[label=above:$\vphantom{\rho^1_X}\mathcal{F}_3(g)$] (d)
  ++(0,-2.875) coordinate (mb) ++(0,-0.875) coordinate[label=below:$\mathcal{F}_1(g)$] (b);
  \path (d) ++(-1.0,0) coordinate[label=above:$\rho^2_X$] (tl) ++(0,-0.125) coordinate (tl2);
  \path (tl) ++(-1.0,0) coordinate[label=above:$\rho^1_X$] (tll) ++(0,-0.125) coordinate (tll2);
  \path (b) ++ (1.0,0) coordinate[label=below:$\vphantom{\mathcal{F}}\rho^1_Y$] (br) ++(0,0.125) coordinate (r);
  \path (br) ++ (1.0,0) coordinate[label=below:$\vphantom{\mathcal{F}}\rho^2_Y$] (brr) ++(0,0.125) coordinate (rr);
  \end{scope}
  \newcommand{\crossingA}{(tll) -- (tll2) to[out=270, in=90] (r) -- (br)}
  \newcommand{\crossingB}{(brr) -- (rr) to[out=90, in=270] (tl2) -- (tl)}
  \newcommand{\fork}{
   (d) -- (mu) -- (mb) -- (b)
  }
  \coordinate (cornerNW) at ($(tll) + (-0.5,0)$);
  \coordinate (cornerSE) at ($(brr) + (0.5,0)$);
  \begin{scope}
  \clip (cornerNW -| cornerSE) -- (cornerSE) -- \crossingB -- cycle;
  \draw[functorF3] \fork;
  \end{scope}
  \begin{scope}
  \clip \crossingA -- \crossingB -- cycle;
  \draw[functorF2] \fork;
  \end{scope}
  \begin{scope}
  \clip (cornerNW) -- \crossingA -- (cornerNW |- cornerSE) -- cycle;
  \draw[functorF1] \fork;
  \end{scope}
  \draw[cross line] \crossingA;
  \draw[cross line] \crossingB;
  \draw (cornerNW) rectangle (cornerSE);
 \end{tikzpicture}
\end{center}
\end{definition}

\begin{definition}
Let $\rho^1, \rho^2\colon \mathcal{F}_1 \to \mathcal{F}_2$ be oplax transformations. A modification $\beth\colon \rho^1 \to \rho^2$ is given by a family of 2-morphisms $\beth_X\colon \rho^1_X \to \rho^2_X$ satisfying the following axiom.
\\ 
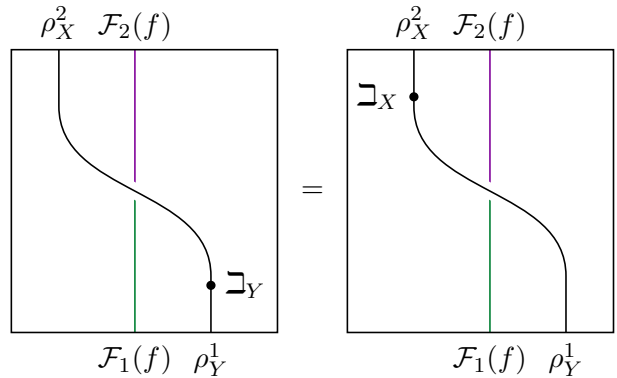
\vspace{-3pt plus 1pt}
\begin{equation*} 
 \begin{tikzpicture}[scale=1.0,baseline={([yshift=-0.5ex]current bounding box.center)}]
  \begin{scope}[on layer=over]
  \path coordinate (mu)
  +(0,0.875) coordinate[label=above:$\vphantom{\rho_X}\smash{\mathcal{F}_2(f)}$] (d)
  ++(0,-2.0) coordinate (mb) ++(0,-0.875) coordinate[label=below:$\vphantom{\rho_Y^1}\mathcal{F}_1(f)$] (b);
  \path (d) ++(-1.0,0) coordinate[label=above:$\rho^2_X$] (tl) ++(0,-0.625) coordinate (tl2) ++(0,-0.125) coordinate (tl3);
  \path (b) ++ (1.0,0) coordinate[label=below:$\rho^1_Y$] (brr) ++(0,0.625) coordinate[dot, label=right:$\beth_Y$] (rr) ++(0,0.125) coordinate (rr2);
  \end{scope}
  \newcommand{\crossing}{(tl) -- (tl2.center) -- (tl3) to[out=270, in=90] (rr2) -- (rr.center) -- (brr)}
  \newcommand{\fork}{
   (d) -- (mu) -- (mb) -- (b)
  }
  \coordinate (cornerNW) at ($(tl) + (-0.625,0)$);
  \coordinate (cornerSE) at ($(brr) + (0.875,0)$);
  \begin{scope}
  \clip \crossing -- (cornerSE) -- (cornerNW -| cornerSE) -- cycle;
  \draw[functorF2] \fork;
  \end{scope}
  \begin{scope}
  \clip (cornerNW) -- \crossing -- (cornerNW |- cornerSE) -- cycle;
  \draw[functorF1] \fork;
  \end{scope}
  \draw[cross line] \crossing;
  \draw (cornerNW) rectangle (cornerSE);
 \end{tikzpicture}
 \enspace =\enspace
 \begin{tikzpicture}[scale=1.0,baseline={([yshift=-0.5ex]current bounding box.center)}]
  \begin{scope}[on layer=over]
  \path coordinate (mu)
  +(0,0.875) coordinate[label=above:$\vphantom{\rho_X}\smash{\mathcal{F}_2(f)}$] (d)
  ++(0,-2.0) coordinate (mb) ++(0,-0.875) coordinate[label=below:$\vphantom{\rho_Y^1}\mathcal{F}_1(f)$] (b);
  \path (d) ++(-1.0,0) coordinate[label=above:$\rho^2_X$] (tl) ++(0,-0.625) coordinate[dot, label=left:$\beth_X$] (tl2) ++(0,-0.125) coordinate (tl3);
  \path (b) ++ (1.0,0) coordinate[label=below:$\rho^1_Y$] (brr) ++(0,0.625) coordinate (rr) ++(0,0.125) coordinate (rr2);
  \end{scope}
  \newcommand{\crossing}{(tl) -- (tl2.center) -- (tl3) to[out=270, in=90] (rr2) -- (rr.center) -- (brr)}
  \newcommand{\fork}{
   (d) -- (mu) -- (mb) -- (b)
  }
  \coordinate (cornerNW) at ($(tl) + (-0.875,0)$);
  \coordinate (cornerSE) at ($(brr) + (0.625,0)$);
  \begin{scope}
  \clip \crossing -- (cornerSE) -- (cornerNW -| cornerSE) -- cycle;
  \draw[functorF2] \fork;
  \end{scope}
  \begin{scope}
  \clip (cornerNW) -- \crossing -- (cornerNW |- cornerSE) -- cycle;
  \draw[functorF1] \fork;
  \end{scope}
  \draw[cross line] \crossing;
  \draw (cornerNW) rectangle (cornerSE);
 \end{tikzpicture}
\end{equation*}

Vertical and horizontal composition of modifications are defined componentwise in the obvious way.
\end{definition}

Let $\Laxop(\B,\C)$ denote the 2-category of lax functors between $\B$ and $\C$, oplax transformations and modifications.

\subsection*{Monads and distributive Laws}

Monads in a 2-category $\C$ may be thought of as lax functors out of the terminal 2-category into $\C$. If $\mathcal{F}$ is such a lax functor, then the underlying functor of the monad is given by the endomorphism $\mathcal{F}(\id)$ on $\mathcal{F}(*)$ with the multiplication and unit derived from the compositor and unitor respectively.
The conditions for $\mathcal{F}$ to be a lax functor then yield associativity and unit axioms for the monad.

Given two monads $(S,\gamma^S,\iota^S)$ and $(T,\gamma^T,\iota^T)$ on an object $X$, it is natural to ask if $TS$ has a natural monad structure. The answer in general is `no', but the question has an affirmative answer in the presence of a distributive law $\sigma\colon ST \to TS$.

\begin{definition}
  Let $(S,\gamma^S,\iota)$ and $(T,\gamma^T,\iota^T)$ be monads. A distributive law between them is a 2-morphism $\sigma\colon ST \to TS$ satisfying the following string diagrams. Here $\sigma$ is represented by a braiding of the red wire $S$ over the blue wire $T$.
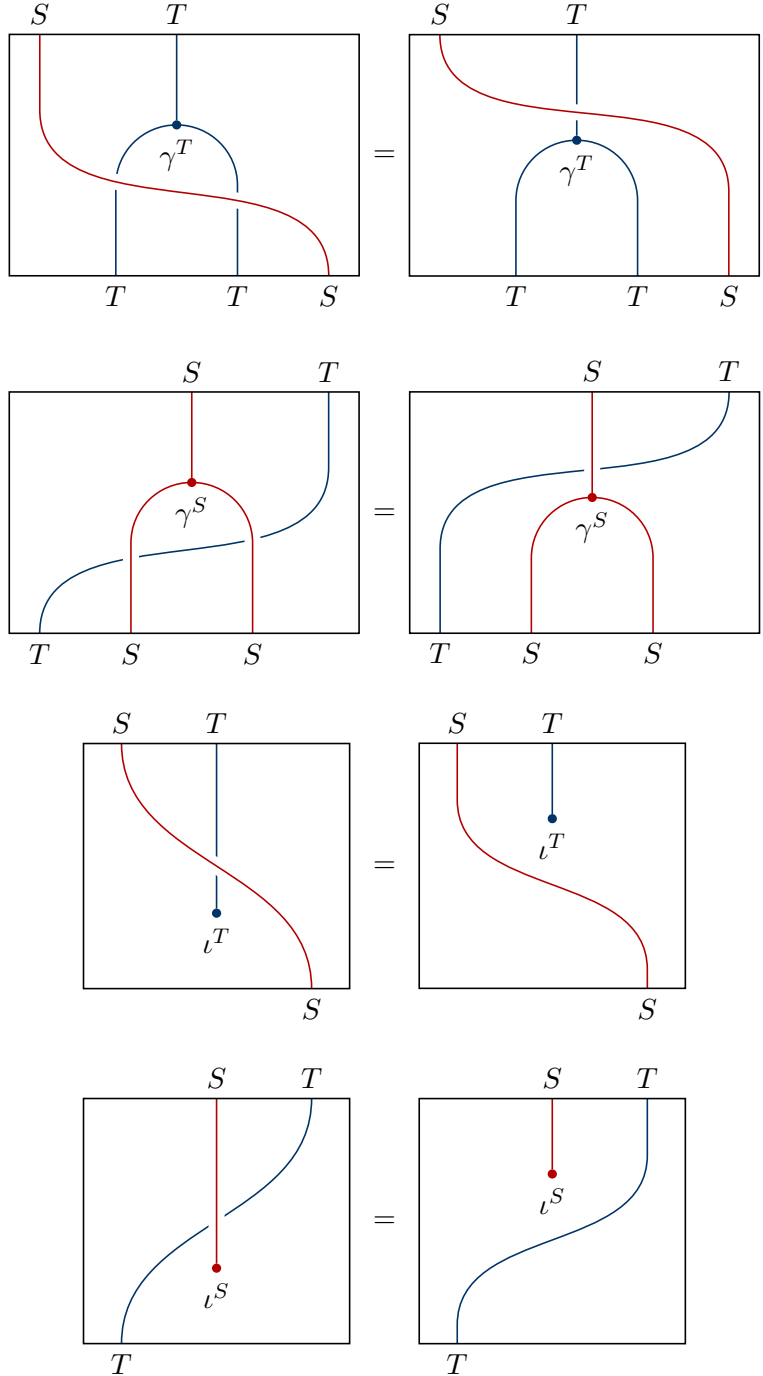
\begin{equation*}
   \begin{tikzpicture}[scale=0.8,baseline={([yshift=-0.5ex]current bounding box.center)}]
    \begin{scope}[on layer=over]
    \path coordinate[dot, functorM, label=below:$\gamma^T$] (mu)
    +(0,1.5) coordinate[label=above:$T$] (d)
    +(-1,-1) coordinate (mbl)
    +(1,-1) coordinate (mbr);
    \path (d -| mbl) ++(-1.25,0) coordinate[label=above:$S$] (tl) ++(0,-1.25) coordinate (tl2);
    \path (mbr) ++(0,-1.5) coordinate[label=below:$T$] (br) ++(1.5,0) coordinate[label=below:$S$] (brr) ++(0,0.25) coordinate (brr2);
    \path (mbl) ++(0,-1.5) coordinate[label=below:$T$] (bl);
    \end{scope}
    \draw[functorM] (bl) -- (mbl) to[out=90, in=180] (mu.center) to[out=0, in=90] (mbr) -- (br)
                    (mu) -- (d);
    \draw[functorL, cross line] (tl) -- (tl2) to[out=270, in=90] (brr);
    \coordinate (cornerNW) at ($(tl) + (-0.5,0)$);
    \coordinate (cornerSE) at ($(brr) + (0.5,0)$);
    \draw (cornerNW) rectangle (cornerSE);
   \end{tikzpicture}
   \,=\,
   \begin{tikzpicture}[scale=0.8,baseline={([yshift=-0.5ex]current bounding box.center)}]
    \begin{scope}[on layer=over]
    \path coordinate[dot, functorM, label=below:$\gamma^T$] (mu)
    +(0,1.75) coordinate[label=above:$T$] (d)
    +(-1,-1) coordinate (mbl)
    +(1,-1) coordinate (mbr);
    \path (d -| mbl) ++(-1.25,0) coordinate[label=above:$S$] (tl) ++(0,-0.25) coordinate (tl2);
    \path (mbr) ++(0,-1.25) coordinate[label=below:$T$] (br) ++(1.5,0) coordinate[label=below:$S$] (brr) ++(0,1.4) coordinate (brr2);
    \path (mbl) ++(0,-1.25) coordinate[label=below:$T$] (bl);
    \end{scope}
    \draw[functorM] (bl) -- (mbl) to[out=90, in=180] (mu.center) to[out=0, in=90] (mbr) -- (br)
                    (mu) -- (d);
    \draw[functorL, cross line] (tl) to[out=270, in=90] (brr2) -- (brr);
    \coordinate (cornerNW) at ($(tl) + (-0.5,0)$);
    \coordinate (cornerSE) at ($(brr) + (0.5,0)$);
    \draw (cornerNW) rectangle (cornerSE);
   \end{tikzpicture}
  \end{equation*}
  
  \vspace{-3pt plus 1pt}
  \begin{equation*}
   \begin{tikzpicture}[scale=0.8,baseline={([yshift=-0.5ex]current bounding box.center)}]
    \begin{scope}[on layer=over]
    \path coordinate[dot, functorL, label=below:$\gamma^S$] (mu)
    +(0,1.5) coordinate[label=above:$S$] (d)
    +(1,-1) coordinate (mbl)
    +(-1,-1) coordinate (mbr);
    \path (d -| mbl) ++(1.25,0) coordinate[label=above:$T$] (tl) ++(0,-1.25) coordinate (tl2);
    \path (mbr) ++(0,-1.5) coordinate[label=below:$S$] (br) ++(-1.5,0) coordinate[label=below:$T$] (brr) ++(0,0.25) coordinate (brr2);
    \path (mbl) ++(0,-1.5) coordinate[label=below:$S$] (bl);
    \end{scope}
    \draw[functorM] (tl) -- (tl2) to[out=270, in=90] (brr);
    \draw[functorL, cross line] (bl) -- (mbl) to[out=90, in=0] (mu.center) to[out=180, in=90] (mbr) -- (br)
                                (mu) -- (d);
    \coordinate (cornerNW) at ($(tl) + (0.5,0)$);
    \coordinate (cornerSE) at ($(brr) + (-0.5,0)$);
    \draw (cornerNW) rectangle (cornerSE);
   \end{tikzpicture}
   \,=\,
   \begin{tikzpicture}[scale=0.8,baseline={([yshift=-0.5ex]current bounding box.center)}]
    \begin{scope}[on layer=over]
    \path coordinate[dot, functorL, label=below:$\gamma^S$] (mu)
    +(0,1.75) coordinate[label=above:$S$] (d)
    +(1,-1) coordinate (mbl)
    +(-1,-1) coordinate (mbr);
    \path (d -| mbl) ++(1.25,0) coordinate[label=above:$T$] (tl) ++(0,-0.25) coordinate (tl2);
    \path (mbr) ++(0,-1.25) coordinate[label=below:$S$] (br) ++(-1.5,0) coordinate[label=below:$T$] (brr) ++(0,1.4) coordinate (brr2);
    \path (mbl) ++(0,-1.25) coordinate[label=below:$S$] (bl);
    \end{scope}
    \draw[functorM] (tl) to[out=270, in=90] (brr2) -- (brr);
    \draw[functorL, cross line] (bl) -- (mbl) to[out=90, in=0] (mu.center) to[out=180, in=90] (mbr) -- (br)
                                (mu) -- (d);
    \coordinate (cornerNW) at ($(tl) + (0.5,0)$);
    \coordinate (cornerSE) at ($(brr) + (-0.5,0)$);
    \draw (cornerNW) rectangle (cornerSE);
   \end{tikzpicture}
  \end{equation*}

\vspace{-3pt plus 1pt}
\begin{equation*}
 \begin{tikzpicture}[scale=1,baseline={([yshift=-0.5ex]current bounding box.center)}]
  \begin{scope}[on layer=over]
  \path coordinate (mu)
  +(0,1.25) coordinate[label=above:$T$] (d)
  ++(0,-1.0) coordinate[dot, functorM, label=below:$\iota^T$] (mb) ++(0,-1.0) coordinate (b);
  \path (d) ++(-1.25,0) coordinate[label=above:$S$] (tl) ++(0,-0.0) coordinate (tl2);
  \path (b) ++ (1.25,0) coordinate[label=below:$S$] (brr) ++(0,0.0) coordinate (rr);
  \end{scope}
  \coordinate (cornerNW) at ($(tl) + (-0.5,0)$);
  \coordinate (cornerSE) at ($(brr) + (0.5,0)$);
  \draw[functorM] (d) -- (mu) -- (mb);
  \draw[functorL, cross line] (tl) -- (tl2) to[out=270, in=90] (rr) -- (brr);
  \draw (cornerNW) rectangle (cornerSE);
 \end{tikzpicture}
 \enspace=\enspace
 \begin{tikzpicture}[scale=1,baseline={([yshift=-0.5ex]current bounding box.center)}]
  \begin{scope}[on layer=over]
  \path coordinate (mu)
  +(0,0.25) coordinate[label=above:$T$] (d)
  ++(0,-0.75) coordinate[dot, functorM, label=below:{$\iota^T$}] (mb) ++(0,-2.25) coordinate (b);
  \path (d) ++(-1.25,0) coordinate[label=above:$S$] (tl) ++(0,-0.75) coordinate (tl2);
  \path (b) ++ (1.25,0) coordinate[label=below:$S$] (brr) ++(0,0.25) coordinate (rr);
  \end{scope}
  \coordinate (cornerNW) at ($(tl) + (-0.5,0)$);
  \coordinate (cornerSE) at ($(brr) + (0.5,0)$);
  \draw[functorM] (d) -- (mu) -- (mb);
  \draw[functorL, cross line] (tl) -- (tl2) to[out=270, in=90] (rr) -- (brr);
  \draw (cornerNW) rectangle (cornerSE);
 \end{tikzpicture}
\end{equation*}

\vspace{-3pt plus 1pt}
\begin{equation*}
 \begin{tikzpicture}[scale=1,baseline={([yshift=-0.5ex]current bounding box.center)}]
  \begin{scope}[on layer=over]
  \path coordinate (mu)
  +(0,1.25) coordinate[label=above:$S$] (d)
  ++(0,-1.0) coordinate[dot, functorL, label=below:$\iota^S$] (mb) ++(0,-1.0) coordinate (b);
  \path (d) ++(1.25,0) coordinate[label=above:$T$] (tl) ++(0,-0.0) coordinate (tl2);
  \path (b) ++ (-1.25,0) coordinate[label=below:$T$] (brr) ++(0,0.0) coordinate (rr);
  \end{scope}
  \coordinate (cornerNW) at ($(tl) + (0.5,0)$);
  \coordinate (cornerSE) at ($(brr) + (-0.5,0)$);
  \draw[functorM] (tl) -- (tl2) to[out=270, in=90] (rr) -- (brr);
  \draw[functorL, cross line] (d) -- (mu) -- (mb);
  \draw (cornerNW) rectangle (cornerSE);
 \end{tikzpicture}
 \enspace=\enspace
 \begin{tikzpicture}[scale=1,baseline={([yshift=-0.5ex]current bounding box.center)}]
  \begin{scope}[on layer=over]
  \path coordinate (mu)
  +(0,0.25) coordinate[label=above:$S$] (d)
  ++(0,-0.75) coordinate[dot, functorL, label=below:{$\iota^S$}] (mb) ++(0,-2.25) coordinate (b);
  \path (d) ++(1.25,0) coordinate[label=above:$T$] (tl) ++(0,-0.75) coordinate (tl2);
  \path (b) ++ (-1.25,0) coordinate[label=below:$T$] (brr) ++(0,0.25) coordinate (rr);
  \end{scope}
  \coordinate (cornerNW) at ($(tl) + (0.5,0)$);
  \coordinate (cornerSE) at ($(brr) + (-0.5,0)$);
  \draw[functorM] (tl) -- (tl2) to[out=270, in=90] (rr) -- (brr);
  \draw[functorL, cross line] (d) -- (mu) -- (mb);
  \draw (cornerNW) rectangle (cornerSE);
 \end{tikzpicture}
\end{equation*}

\end{definition}

Given a distributive law $\sigma\colon ST \to TS$, we have that $(TS, (\gamma^T \ast \gamma^S) (T \sigma S), \iota^T \ast \iota^S)$ is a monad.

\section{The lax bifunctor theorem}

In order to generalise the classical bifunctor theorem to the setting of lax functors, the equality $L_{C_2}(f)M_{B_1}(g) = M_{B_2}(g)L_{C_1}(f)$ in \Cref{thm:Mac} should be replaced by some suitable 2-morphism $\sigma_{f,g}\colon L_{C_2}(f)M_{B_1}(g) \to M_{B_2}(g)L_{C_1}(f)$. While it often makes sense to replace equalities with isomorphisms, in the lax setting it is more natural not to ask for the morphism to be invertible. The construction will only succeed when each $\sigma$ satisfies certain coherence conditions, which we describe below.
Our terminology is chosen by analogy to the distributive laws of monads, which have similar axioms.

\begin{definition}
Given a family of lax functors $L_C\colon \B \to \D$ indexed by objects of $\C$ and a family of lax functors $M_B\colon \C \to \D$ indexed by objects of $\B$ such that $L_C(B) = M_B(C)$, we say a family of 2-morphisms $\sigma_{f,g}\colon L_{C_2}(f) M_{B_1}(g) \to M_{B_2}(g)L_{C_1}(f)$ indexed by two 1-morphisms $f\colon B_1 \to B_2$ in $\B$ and $g\colon C_1 \to C_2$ in $\C$ is called a \emph{distributive law of lax functors} if it satisfies the following conditions represented in string diagrams below, where morphisms to which an `$L$' lax functor has been applied are written in red, morphisms to which an `$M$' lax functor has been applied are written in blue and $\sigma$ is denoted by a red line on the right crossing over a blue line on the left.
  
  \vspace{-3pt plus 1pt}
  \begin{equation}\tag{D1}\label{eq:distcomp1}
   \begin{tikzpicture}[scale=1.0,baseline={([yshift=-0.5ex]current bounding box.center)}]
    \begin{scope}[on layer=over]
    \path coordinate[dot, functorM, label=below:$\gamma^{B_2}_{g_2,g_1}$] (mu)
    +(0,1.5) coordinate[label=above:$M_{B_2}(g_2g_1)$] (d)
    +(-1,-1) coordinate (mbl)
    +(1,-1) coordinate (mbr);
    \path (d -| mbl) ++(-1.25,0) coordinate[label=above:$L_{C_1}(f_1)$] (tl) ++(0,-1.25) coordinate (tl2);
    \path (mbr) ++(0,-1.5) coordinate[label=below:$M_{B_1}(g_2)$] (br) ++(1.5,0) coordinate[label=below:$L_{C_3}(f_1)$] (brr) ++(0,0.25) coordinate (brr2);
    \path (mbl) ++(0,-1.5) coordinate[label=below:$M_{B_1}(g_1)$] (bl);
    \end{scope}
    \draw[functorM] (bl) -- (mbl) to[out=90, in=180] (mu.center) to[out=0, in=90] (mbr) -- (br)
                    (mu) -- (d);
    \draw[functorL, cross line] (tl) -- (tl2) to[out=270, in=90] (brr);
    \coordinate (cornerNW) at ($(tl) + (-0.5,0)$);
    \coordinate (cornerSE) at ($(brr) + (0.5,0)$);
    \draw (cornerNW) rectangle (cornerSE);
   \end{tikzpicture}
   \,=\,
   \begin{tikzpicture}[scale=1.0,baseline={([yshift=-0.5ex]current bounding box.center)}]
    \begin{scope}[on layer=over]
    \path coordinate[dot, functorM, label=below:$\gamma^{B_1}_{g_2,g_1}$] (mu)
    +(0,1.75) coordinate[label=above:$M_{B_2}(g_2g_1)$] (d)
    +(-1,-1) coordinate (mbl)
    +(1,-1) coordinate (mbr);
    \path (d -| mbl) ++(-1.25,0) coordinate[label=above:$L_{C_1}(f_1)$] (tl) ++(0,-0.25) coordinate (tl2);
    \path (mbr) ++(0,-1.25) coordinate[label=below:$M_{B_1}(g_2)$] (br) ++(1.5,0) coordinate[label=below:$L_{C_3}(f_1)$] (brr) ++(0,1.4) coordinate (brr2);
    \path (mbl) ++(0,-1.25) coordinate[label=below:$M_{B_1}(g_1)$] (bl);
    \end{scope}
    \draw[functorM] (bl) -- (mbl) to[out=90, in=180] (mu.center) to[out=0, in=90] (mbr) -- (br)
                    (mu) -- (d);
    \draw[functorL, cross line] (tl) to[out=270, in=90] (brr2) -- (brr);
    \coordinate (cornerNW) at ($(tl) + (-0.5,0)$);
    \coordinate (cornerSE) at ($(brr) + (0.5,0)$);
    \draw (cornerNW) rectangle (cornerSE);
   \end{tikzpicture}
  \end{equation}
  
  \vspace{-3pt plus 1pt}
  \begin{equation}\tag{D2}\label{eq:distcomp2}
   \begin{tikzpicture}[scale=1.0,baseline={([yshift=-0.5ex]current bounding box.center)}]
    \begin{scope}[on layer=over]
    \path coordinate[dot, functorL, label=below:$\gamma^{C_1}_{f_2,f_1}$] (mu)
    +(0,1.5) coordinate[label=above:$L_{C_1}(f_2f_1)$] (d)
    +(1,-1) coordinate (mbl)
    +(-1,-1) coordinate (mbr);
    \path (d -| mbl) ++(1.25,0) coordinate[label=above:$M_{B_3}(g_1)$] (tl) ++(0,-1.25) coordinate (tl2);
    \path (mbr) ++(0,-1.5) coordinate[label=below:$L_{C_2}(f_1)$] (br) ++(-1.5,0) coordinate[label=below:$M_{B_1}(g_1)$] (brr) ++(0,0.25) coordinate (brr2);
    \path (mbl) ++(0,-1.5) coordinate[label=below:$L_{C_2}(f_2)$] (bl);
    \end{scope}
    \draw[functorM] (tl) -- (tl2) to[out=270, in=90] (brr);
    \draw[functorL, cross line] (bl) -- (mbl) to[out=90, in=0] (mu.center) to[out=180, in=90] (mbr) -- (br)
                                (mu) -- (d);
    \coordinate (cornerNW) at ($(tl) + (0.5,0)$);
    \coordinate (cornerSE) at ($(brr) + (-0.5,0)$);
    \draw (cornerNW) rectangle (cornerSE);
   \end{tikzpicture}
   \,=\,
   \begin{tikzpicture}[scale=1.0,baseline={([yshift=-0.5ex]current bounding box.center)}]
    \begin{scope}[on layer=over]
    \path coordinate[dot, functorL, label=below:$\gamma^{C_2}_{f_2,f_1}$] (mu)
    +(0,1.75) coordinate[label=above:$L_{C_1}(f_2f_1)$] (d)
    +(1,-1) coordinate (mbl)
    +(-1,-1) coordinate (mbr);
    \path (d -| mbl) ++(1.25,0) coordinate[label=above:$M_{B_3}(g_1)$] (tl) ++(0,-0.25) coordinate (tl2);
    \path (mbr) ++(0,-1.25) coordinate[label=below:$L_{C_2}(f_1)$] (br) ++(-1.5,0) coordinate[label=below:$M_{B_1}(g_1)$] (brr) ++(0,1.4) coordinate (brr2);
    \path (mbl) ++(0,-1.25) coordinate[label=below:$L_{C_2}(f_2)$] (bl);
    \end{scope}
    \draw[functorM] (tl) to[out=270, in=90] (brr2) -- (brr);
    \draw[functorL, cross line] (bl) -- (mbl) to[out=90, in=0] (mu.center) to[out=180, in=90] (mbr) -- (br)
                                (mu) -- (d);
    \coordinate (cornerNW) at ($(tl) + (0.5,0)$);
    \coordinate (cornerSE) at ($(brr) + (-0.5,0)$);
    \draw (cornerNW) rectangle (cornerSE);
   \end{tikzpicture}
  \end{equation}

\vspace{-3pt plus 1pt}
\begin{equation}\tag{D3}\label{eq:distunit1}
 \begin{tikzpicture}[scale=1.25,baseline={([yshift=-0.5ex]current bounding box.center)}]
  \begin{scope}[on layer=over]
  \path coordinate (mu)
  +(0,1.25) coordinate[label=above:$M_{B_2}(\id_C)$] (d)
  ++(0,-1.0) coordinate[dot, functorM, label=below:$\iota^{B_1}_C$] (mb) ++(0,-1.0) coordinate (b);
  \path (d) ++(-1.25,0) coordinate[label=above:$L_{C}(f)$] (tl) ++(0,-0.0) coordinate (tl2);
  \path (b) ++ (1.25,0) coordinate[label=below:$L_{C}(f)$] (brr) ++(0,0.0) coordinate (rr);
  \end{scope}
  \coordinate (cornerNW) at ($(tl) + (-0.5,0)$);
  \coordinate (cornerSE) at ($(brr) + (0.5,0)$);
  \draw[functorM] (d) -- (mu) -- (mb);
  \draw[functorL, cross line] (tl) -- (tl2) to[out=270, in=90] (rr) -- (brr);
  \draw (cornerNW) rectangle (cornerSE);
 \end{tikzpicture}
 \enspace=\enspace
 \begin{tikzpicture}[scale=1.25,baseline={([yshift=-0.5ex]current bounding box.center)}]
  \begin{scope}[on layer=over]
  \path coordinate (mu)
  +(0,0.25) coordinate[label=above:$M_{B_2}(\id_C)$] (d)
  ++(0,-0.75) coordinate[dot, functorM, label=below:{$\iota^{B_2}_C$}] (mb) ++(0,-2.25) coordinate (b);
  \path (d) ++(-1.25,0) coordinate[label=above:$L_{C}(f)$] (tl) ++(0,-0.75) coordinate (tl2);
  \path (b) ++ (1.25,0) coordinate[label=below:$L_{C}(f)$] (brr) ++(0,0.25) coordinate (rr);
  \end{scope}
  \coordinate (cornerNW) at ($(tl) + (-0.5,0)$);
  \coordinate (cornerSE) at ($(brr) + (0.5,0)$);
  \draw[functorM] (d) -- (mu) -- (mb);
  \draw[functorL, cross line] (tl) -- (tl2) to[out=270, in=90] (rr) -- (brr);
  \draw (cornerNW) rectangle (cornerSE);
 \end{tikzpicture}
\end{equation}

\vspace{-3pt plus 1pt}
\begin{equation}\tag{D4}\label{eq:distunit2}
 \begin{tikzpicture}[scale=1.25,baseline={([yshift=-0.5ex]current bounding box.center)}]
  \begin{scope}[on layer=over]
  \path coordinate (mu)
  +(0,1.25) coordinate[label=above:$L_{C_1}(\id_B)$] (d)
  ++(0,-1.0) coordinate[dot, functorL, label=below:$\iota^{C_2}_B$] (mb) ++(0,-1.0) coordinate (b);
  \path (d) ++(1.25,0) coordinate[label=above:$M_{B}(g)$] (tl) ++(0,-0.0) coordinate (tl2);
  \path (b) ++ (-1.25,0) coordinate[label=below:$M_{B}(g)$] (brr) ++(0,0.0) coordinate (rr);
  \end{scope}
  \coordinate (cornerNW) at ($(tl) + (0.5,0)$);
  \coordinate (cornerSE) at ($(brr) + (-0.5,0)$);
  \draw[functorM] (tl) -- (tl2) to[out=270, in=90] (rr) -- (brr);
  \draw[functorL, cross line] (d) -- (mu) -- (mb);
  \draw (cornerNW) rectangle (cornerSE);
 \end{tikzpicture}
 \enspace=\enspace
 \begin{tikzpicture}[scale=1.25,baseline={([yshift=-0.5ex]current bounding box.center)}]
  \begin{scope}[on layer=over]
  \path coordinate (mu)
  +(0,0.25) coordinate[label=above:$L_{C_1}(\id_B)$] (d)
  ++(0,-0.75) coordinate[dot, functorL, label=below:{$\iota^{C_1}_B$}] (mb) ++(0,-2.25) coordinate (b);
  \path (d) ++(1.25,0) coordinate[label=above:$M_{B}(g)$] (tl) ++(0,-0.75) coordinate (tl2);
  \path (b) ++ (-1.25,0) coordinate[label=below:$M_{B}(g)$] (brr) ++(0,0.25) coordinate (rr);
  \end{scope}
  \coordinate (cornerNW) at ($(tl) + (0.5,0)$);
  \coordinate (cornerSE) at ($(brr) + (-0.5,0)$);
  \draw[functorM] (tl) -- (tl2) to[out=270, in=90] (rr) -- (brr);
  \draw[functorL, cross line] (d) -- (mu) -- (mb);
  \draw (cornerNW) rectangle (cornerSE);
 \end{tikzpicture}
\end{equation}

\vspace{-3pt plus 1pt}
\begin{equation}\tag{D5}\label{eq:distnat1}
 \begin{tikzpicture}[scale=1.25,baseline={([yshift=-0.5ex]current bounding box.center)}]
  \begin{scope}[on layer=over]
  \path coordinate (mu)
  +(0,0.75) coordinate[label=above:$M_{B_2}(g')$] (d)
  ++(0,-1.75) coordinate[dot, functorM, label=left:$M_{B_1}(\beta)$] (mb) ++(0,-0.75) coordinate[label=below:$M_{B_1}(g)$] (b);
  \path (d) ++(-1.25,0) coordinate[label=above:$L_{C_1}(f)$] (tl) ++(0,-0.0) coordinate (tl2);
  \path (b) ++ (1.25,0) coordinate[label=below:$L_{C_2}(f)$] (brr) ++(0,0.0) coordinate (rr);
  \end{scope}
  \coordinate (cornerNW) at ($(tl) + (-0.5,0)$);
  \coordinate (cornerSE) at ($(brr) + (0.5,0)$);
  \draw[functorM] (d) -- (mu) -- (mb) -- (b);
  \draw[functorL, cross line] (tl) -- (tl2) to[out=270, in=90] (rr) -- (brr);
  \draw (cornerNW) rectangle (cornerSE);
 \end{tikzpicture}
 \enspace=\enspace
 \begin{tikzpicture}[scale=1.25,baseline={([yshift=-0.5ex]current bounding box.center)}]
  \begin{scope}[on layer=over]
  \path coordinate[dot, functorM, label=right:$M_{B_2}(\beta)$] (mu)
  +(0,0.75) coordinate[label=above:$M_{B_2}(g')$] (d)
  ++(0,-1.75) coordinate (mb) ++(0,-0.75) coordinate[label=below:$M_{B_1}(g)$] (b);
  \path (d) ++(-1.25,0) coordinate[label=above:$L_{C_1}(f)$] (tl) ++(0,-0.0) coordinate (tl2);
  \path (b) ++ (1.25,0) coordinate[label=below:$L_{C_2}(f)$] (brr) ++(0,0.0) coordinate (rr);
  \end{scope}
  \coordinate (cornerNW) at ($(tl) + (-0.5,0)$);
  \coordinate (cornerSE) at ($(brr) + (0.5,0)$);
  \draw[functorM] (d) -- (mu) -- (mb) -- (b);
  \draw[functorL, cross line] (tl) -- (tl2) to[out=270, in=90] (rr) -- (brr);
  \draw (cornerNW) rectangle (cornerSE);
 \end{tikzpicture}
\end{equation}

\vspace{-3pt plus 1pt}
\begin{equation}\tag{D6}\label{eq:distnat2}
 \begin{tikzpicture}[scale=1.25,baseline={([yshift=-0.5ex]current bounding box.center)}]
  \begin{scope}[on layer=over]
  \path coordinate (mu)
  +(0,0.75) coordinate[label=above:$L_{C_1}(f')$] (d)
  ++(0,-1.75) coordinate[dot, functorL, label=right:$L_{C_2}(\alpha)$] (mb) ++(0,-0.75) coordinate[label=below:$L_{C_2}(f)$] (b);
  \path (d) ++(1.25,0) coordinate[label=above:$M_{B_2}(g)$] (tl) ++(0,-0.0) coordinate (tl2);
  \path (b) ++ (-1.25,0) coordinate[label=below:$M_{B_1}(g)$] (brr) ++(0,0.0) coordinate (rr);
  \end{scope}
  \coordinate (cornerNW) at ($(tl) + (0.5,0)$);
  \coordinate (cornerSE) at ($(brr) + (-0.5,0)$);
  \draw[functorM] (tl) -- (tl2) to[out=270, in=90] (rr) -- (brr);
  \draw[functorL, cross line] (d) -- (mu) -- (mb) -- (b);
  \draw (cornerNW) rectangle (cornerSE);
 \end{tikzpicture}
 \enspace=\enspace
 \begin{tikzpicture}[scale=1.25,baseline={([yshift=-0.5ex]current bounding box.center)}]
  \begin{scope}[on layer=over]
  \path coordinate[dot, functorL, label=left:$L_{C_1}(\alpha)$] (mu)
  +(0,0.75) coordinate[label=above:$L_{C_1}(f')$] (d)
  ++(0,-1.75) coordinate (mb) ++(0,-0.75) coordinate[label=below:$L_{C_2}(f)$] (b);
  \path (d) ++(1.25,0) coordinate[label=above:$M_{B_2}(g)$] (tl) ++(0,-0.0) coordinate (tl2);
  \path (b) ++ (-1.25,0) coordinate[label=below:$M_{B_1}(g)$] (brr) ++(0,0.0) coordinate (rr);
  \end{scope}
  \coordinate (cornerNW) at ($(tl) + (0.5,0)$);
  \coordinate (cornerSE) at ($(brr) + (-0.5,0)$);
  \draw[functorM] (tl) -- (tl2) to[out=270, in=90] (rr) -- (brr);
  \draw[functorL, cross line] (d) -- (mu) -- (mb) -- (b);
  \draw (cornerNW) rectangle (cornerSE);
 \end{tikzpicture}
\end{equation}

These conditions all amount to asking for the `crossings' to pass through the structure of the lax functors.

The collection of all distributive laws between all such families of lax functors with domains $\B$ and $\C$ and codomain $\D$ will be denoted by $\Dist(\B,\C,\D)$.
\end{definition}

\begin{theorem}[Lax bifunctor theorem]\label{thm:lax_bifunctor}
 Let $\sigma \in \Dist(\B,\C,\D)$ be a distributive law between families of lax functors $(L_C,\gamma^C,\iota^C)$ and $(M_B,\gamma^B,\iota^B)$.
 
We may construct a lax bifunctor $P\colon \B \times \C \to \D$ with $P(B,C) = L_C(B) = M_B(C)$ on objects, $P(f,g) = M_{B_2}(g)L_{C_1}(f)$ on 1-morphisms $f\colon B_1 \to B_2$ and $g\colon C_1 \to C_2$, and $P(\alpha,\beta)= M_{B_2}(\beta) \ast L_{C_1}(\alpha)$ on 2-morphisms, and with the unitor $\iota_{B,C}\colon \id_{S(B,C)} \to P(\id_B,\id_C)$ given by $\iota^B_C \ast \iota^C_B$, and the compositor $\gamma_{(f',g'),(f,g)}\colon P(f',g')P(f,g) \to P(f'f,g'g)$ for $f'\colon B_2 \to B_3$ and $g'\colon C_2 \to C_3$ given by the following string diagram.
  \\ 
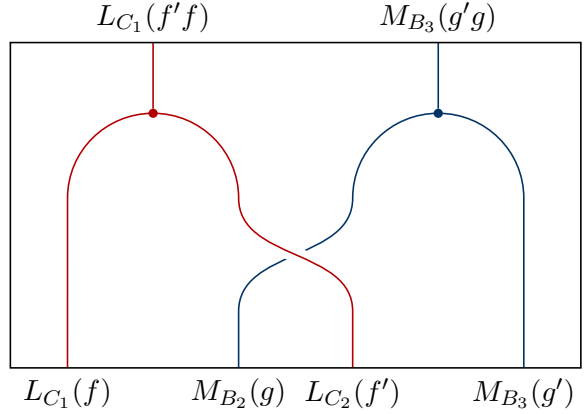
\vspace{-3pt plus 1pt}
  \begin{center}
   \begin{tikzpicture}[scale=1.5,baseline={([yshift=-0.5ex]current bounding box.center)}]
    \begin{scope}[on layer=over]
    \path coordinate[dot, functorM] (mu)
    +(0,0.625) coordinate[label={above:$M_{B_3}(g'g)$}] (d)
    +(-0.75,-0.75) coordinate (mbl)
    +(0.75,-0.75) coordinate (mbr);
    \path (mbr) ++(0,-1.5) coordinate[label={below:$M_{B_3}(g')$}] (br);
    \path (mbl) ++(-1.0,-1.0) coordinate (bl2) ++(0,-0.5) coordinate[label={below:$M_{B_2}(g)$}] (bl);
    \path (mu) ++(-2.5,0) coordinate[dot, functorL] (muL)
    +(0,0.625) coordinate[label={above:$L_{C_1}(f'f)$}] (dL)
    +(-0.75,-0.75) coordinate (mblL)
    +(0.75,-0.75) coordinate (mbrL);
    \path (mbrL) ++(1.0,-1.0) coordinate (brL2) ++(0,-0.5) coordinate[label={below:$L_{C_2}(f')$}] (brL);
    \path (mblL) ++(0,-1.5) coordinate[label={below:$L_{C_1}(f)$}] (blL);
    \end{scope}
    \draw[functorM] (bl) -- (bl2) to[out=90, in=-90] (mbl) to[out=90, in=180] (mu.center) to[out=0, in=90] (mbr) -- (br)
                    (mu) -- (d);
    \draw[functorL, cross line] (blL) -- (mblL) to[out=90, in=180] (muL.center) to[out=0, in=90] (mbrL) to[out=-90, in=90] (brL2) -- (brL)
                                (muL) -- (dL);
    \coordinate (tl) at (dL -| blL);
    \coordinate (cornerNW) at ($(tl) + (-0.5,0)$);
    \coordinate (cornerSE) at ($(br) + (0.5,0)$);
    \draw (cornerNW) rectangle (cornerSE);
   \end{tikzpicture}
  \end{center}
 
 Furthermore, $P$ is related to the $L$ and $M$ families by canonical oplax transformations $\kappa^B\colon M_B \to P(B,-)$ and $\kappa^C\colon L_C \to P(-,C)$
  whose 1-morphism components are identities (that is, these are \emph{icons} as in \cite{lack2010icons}) and where $\kappa^B_g = M_B(g) \iota^{C_1}_B$ and $\kappa^C_f = \iota^{B_2}_C L_C(f)$ for $g\colon C_1 \to C_2$ and $f\colon B_1 \to B_2$.
 \end{theorem}
 
 \begin{proof}
 We start by establishing that $P$ is `functorial on 2-morphisms'.
  From the definition of $P$ we have $P(\alpha'\alpha,\beta'\beta) = M_{B_2}(\beta'\beta)\ast L_{C_1}(\alpha'\alpha)$. Since the lax functors $M_{B_2}$ and $L_{C_1}$ are functorial on 2-morphisms, this is equal to $M_{B_2}(\beta')M_{B_2}(\beta) \ast L_{C_1}(\alpha')L_{C_1}(\alpha)$ and by the interchange law this is the same as $(M_{B_2}(\beta')\ast L_{C_1}(\alpha')) \circ (M_{B_2}(\beta) \ast L_{C_1}(\alpha)) = P(\alpha',\beta')P(\alpha,\beta)$. Moreover, $P(\id_f,\id_g)=M_{B_2}(\id_g) \ast L_{C_1}(\id_f)=\id_{M_{B_2}(g)} \ast \id_{L_{C_1}(f)} = \id_{P(f,g)}$, which proves the claim.
  
  Next we show that naturality condition for $P$. This follows from conditions \ref{eq:distnat1} and \ref{eq:distnat2} and the naturality conditions of the $L$ and $M$ lax functors as shown in the string diagrams below.
  \\ 
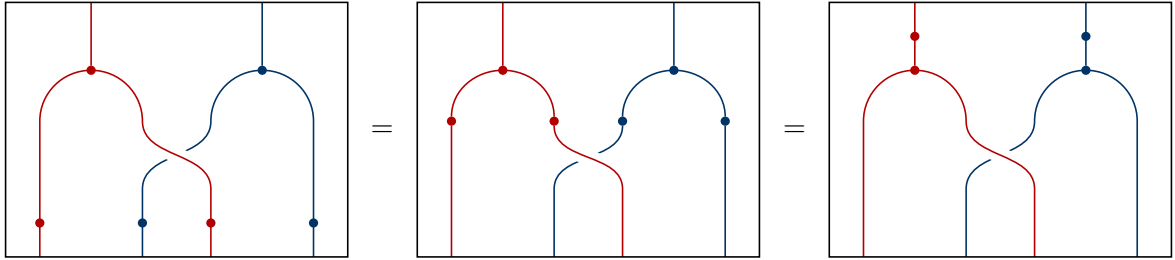
\vspace{-3pt plus 1pt}
  \begin{align*}
   \begin{tikzpicture}[scale=0.9,baseline={([yshift=-0.5ex]current bounding box.center)}]
    \begin{scope}[on layer=over]
    \path coordinate[dot, functorM] (mu)
    +(0,1.0) coordinate (d)
    +(-0.75,-0.75) coordinate (mbl)
    +(0.75,-0.75) coordinate (mbr);
    \path (mbr) ++(0,-1.5) coordinate[dot, functorM] (blLa) ++(0,-0.5) coordinate (br);
    \path (mbl) ++(-1.0,-1.0) coordinate (bl2) ++(0,-0.5) coordinate[dot, functorM] (bla) ++(0,-0.5) coordinate (bl);
    \path (mu) ++(-2.5,0) coordinate[dot, functorL] (muL)
    +(0,1.0) coordinate (dL)
    +(-0.75,-0.75) coordinate (mblL)
    +(0.75,-0.75) coordinate (mbrL);
    \path (mbrL) ++(1.0,-1.0) coordinate (brL2) ++(0,-0.5) coordinate[dot, functorL] (brLa) ++(0,-0.5) coordinate (brL);
    \path (mblL) ++(0,-1.5) coordinate[dot, functorL] (blLa) ++(0,-0.5) coordinate (blL);
    \end{scope}
    \draw[functorM] (bl) -- (bl2) to[out=90, in=-90] (mbl) to[out=90, in=180] (mu.center) to[out=0, in=90] (mbr) -- (br)
                    (mu) -- (d);
    \draw[functorL, cross line] (blL) -- (mblL) to[out=90, in=180] (muL.center) to[out=0, in=90] (mbrL) to[out=-90, in=90] (brL2) -- (brL)
                                (muL) -- (dL);
    \coordinate (tl) at (dL -| blL);
    \coordinate (cornerNW) at ($(tl) + (-0.5,0)$);
    \coordinate (cornerSE) at ($(br) + (0.5,0)$);
    \draw (cornerNW) rectangle (cornerSE);
   \end{tikzpicture}
   \enspace=\enspace
   \begin{tikzpicture}[scale=0.9,baseline={([yshift=-0.5ex]current bounding box.center)}]
    \begin{scope}[on layer=over]
    \path coordinate[dot, functorM] (mu)
    +(0,1.0) coordinate (d)
    +(-0.75,-0.75) coordinate[dot, functorM] (mbl)
    +(0.75,-0.75) coordinate[dot, functorM] (mbr);
    \path (mbr) ++(0,-1.5) coordinate (blLa) ++(0,-0.5) coordinate (br);
    \path (mbl) ++(-1.0,-1.0) coordinate (bl2) ++(0,-0.5) coordinate (bla) ++(0,-0.5) coordinate (bl);
    \path (mu) ++(-2.5,0) coordinate[dot, functorL] (muL)
    +(0,1.0) coordinate (dL)
    +(-0.75,-0.75) coordinate[dot, functorL] (mblL)
    +(0.75,-0.75) coordinate[dot, functorL] (mbrL);
    \path (mbrL) ++(1.0,-1.0) coordinate (brL2) ++(0,-0.5) coordinate (brLa) ++(0,-0.5) coordinate (brL);
    \path (mblL) ++(0,-1.5) coordinate (blLa) ++(0,-0.5) coordinate (blL);
    \end{scope}
    \draw[functorM] (bl) -- (bl2) to[out=90, in=-90] (mbl) to[out=90, in=180] (mu.center) to[out=0, in=90] (mbr) -- (br)
                    (mu) -- (d);
    \draw[functorL, cross line] (blL) -- (mblL) to[out=90, in=180] (muL.center) to[out=0, in=90] (mbrL) to[out=-90, in=90] (brL2) -- (brL)
                                (muL) -- (dL);
    \coordinate (tl) at (dL -| blL);
    \coordinate (cornerNW) at ($(tl) + (-0.5,0)$);
    \coordinate (cornerSE) at ($(br) + (0.5,0)$);
    \draw (cornerNW) rectangle (cornerSE);
   \end{tikzpicture}
   \enspace=\enspace
   \begin{tikzpicture}[scale=0.9,baseline={([yshift=-0.5ex]current bounding box.center)}]
    \begin{scope}[on layer=over]
    \path coordinate[dot, functorM] (mu)
    +(-0.75,-0.75) coordinate (mbl)
    +(0.75,-0.75) coordinate (mbr);
    \path (mu) ++(0,0.5) coordinate[dot, functorM] (d2) ++(0,0.5) coordinate (d);
    \path (mbr) ++(0,-1.5) coordinate (blLa) ++(0,-0.5) coordinate (br);
    \path (mbl) ++(-1.0,-1.0) coordinate (bl2) ++(0,-0.5) coordinate (bla) ++(0,-0.5) coordinate (bl);
    \path (mu) ++(-2.5,0) coordinate[dot, functorL] (muL)
    +(-0.75,-0.75) coordinate (mblL)
    +(0.75,-0.75) coordinate (mbrL);
    \path (muL) ++(0,0.5) coordinate[dot, functorL] (dL2) ++(0,0.5) coordinate (dL);
    \path (mbrL) ++(1.0,-1.0) coordinate (brL2) ++(0,-0.5) coordinate (brLa) ++(0,-0.5) coordinate (brL);
    \path (mblL) ++(0,-1.5) coordinate (blLa) ++(0,-0.5) coordinate (blL);
    \end{scope}
    \draw[functorM] (bl) -- (bl2) to[out=90, in=-90] (mbl) to[out=90, in=180] (mu.center) to[out=0, in=90] (mbr) -- (br)
                    (mu) -- (d);
    \draw[functorL, cross line] (blL) -- (mblL) to[out=90, in=180] (muL.center) to[out=0, in=90] (mbrL) to[out=-90, in=90] (brL2) -- (brL)
                                (muL) -- (dL);
    \coordinate (tl) at (dL -| blL);
    \coordinate (cornerNW) at ($(tl) + (-0.5,0)$);
    \coordinate (cornerSE) at ($(br) + (0.5,0)$);
    \draw (cornerNW) rectangle (cornerSE);
   \end{tikzpicture}
  \end{align*}
  
  To see that $\gamma$ satisfies the associativity laws consider the following string diagrams. 
  \\ \vspace{-3pt plus 1pt}
  \begingroup
  \allowdisplaybreaks
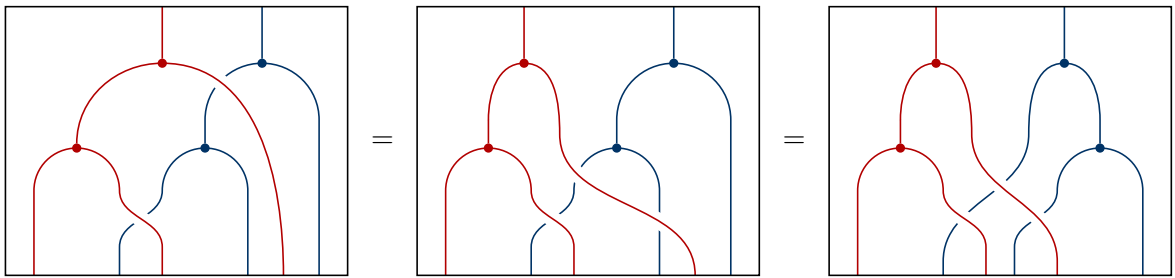
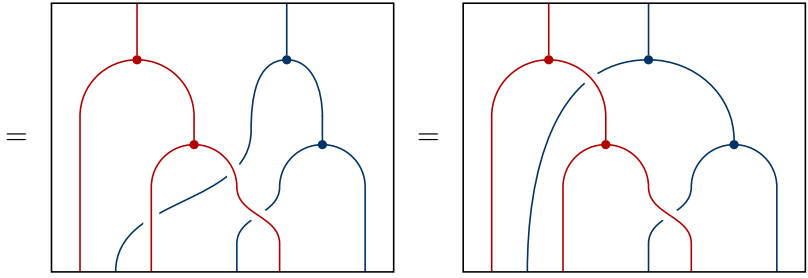
  \begin{align*} 
   \begin{tikzpicture}[scale=0.75,baseline={([yshift=-0.5ex]current bounding box.center)}]
    \begin{scope}[on layer=over]
    \path coordinate[dot, functorM] (mu)
    +(0,0.5) coordinate (d)
    +(-0.75,-0.75) coordinate (mbl)
    +(0.75,-0.75) coordinate (mbr);
    \path (mbr) ++(0,-1.5) coordinate (br);
    \path (mbl) ++(-0.75,-1.0) coordinate (bl2) ++(0,-0.5) coordinate (bl);
    \path (mu) ++(-2.25,0) coordinate[dot, functorL] (muL)
    +(0,0.5) coordinate (dL)
    +(-0.75,-0.75) coordinate (mblL)
    +(0.75,-0.75) coordinate (mbrL);
    \path (mbrL) ++(0.75,-1.0) coordinate (brL2) ++(0,-0.5) coordinate (brL);
    \path (mblL) ++(0,-1.5) coordinate (blL);
    \path (dL) ++(+1.5,1) coordinate[dot, functorL] (muUL)
    +(0,1.0) coordinate (dUL)
    +(-1.5,-1) coordinate (mblUL)
    +(1.5,-1) coordinate (mbrUL);
    \path (d) ++(+1,1) coordinate[dot, functorM] (muU)
    +(0,1.0) coordinate (dU)
    +(-1,-1) coordinate (mblU)
    +(1,-1) coordinate (mbrU);
    \coordinate (brr) at (mbrU |- br);
    \coordinate (brmid) at ($(br)!0.5!(brr)$);
    \end{scope}
    \draw[functorM] (bl) -- (bl2) to[out=90, in=-90] (mbl) to[out=90, in=180] (mu.center) to[out=0, in=90] (mbr) -- (br)
                     (mu) -- (d) -- (mblU) to[out=90, in=180] (muU.center) to[out=0, in=90] (mbrU) -- (brr)
                     (muU) -- (dU);
    \draw[functorL, cross line] (blL) -- (mblL) to[out=90, in=180] (muL.center) to[out=0, in=90] (mbrL) to[out=-90, in=90] (brL2) -- (brL)
                     (muL) to[out=90, in=180] (muUL.center)
                     (muUL) -- (dUL)
                     (muUL.center) to[out=0, in=90] (brmid);
    \coordinate (tl) at (mblL |- dUL);
    \coordinate (cornerNW) at ($(tl) + (-0.5,0)$);
    \coordinate (cornerSE) at ($(brr) + (0.5,0)$);
    \draw (cornerNW) rectangle (cornerSE);
   \end{tikzpicture}
   &\enspace=\enspace
   \begin{tikzpicture}[scale=0.75,baseline={([yshift=-0.5ex]current bounding box.center)}]
    \begin{scope}[on layer=over]
    \path coordinate[dot, functorM] (mu)
    +(0,0.5) coordinate (d)
    +(-0.75,-0.75) coordinate (mbl)
    +(0.75,-0.75) coordinate (mbr);
    \path (mbr) ++(0,-1.5) coordinate (br);
    \path (mbl) ++(-0.75,-1.0) coordinate (bl2) ++(0,-0.5) coordinate (bl);
    \path (mu) ++(-2.25,0) coordinate[dot, functorL] (muL)
    +(0,0.5) coordinate (dL)
    +(-0.75,-0.75) coordinate (mblL)
    +(0.75,-0.75) coordinate (mbrL);
    \path (mbrL) ++(0.75,-1.0) coordinate (brL2) ++(0,-0.5) coordinate (brL);
    \path (mblL) ++(0,-1.5) coordinate (blL);
    \path (dL) ++(+0.625,1) coordinate[dot, functorL] (muUL)
    +(0,1.0) coordinate (dUL)
    +(-0.625,-1) coordinate (mblUL)
    +(0.625,-1.25) coordinate (mbrUL);
    \path (d) ++(+1,1) coordinate[dot, functorM] (muU)
    +(0,1.0) coordinate (dU)
    +(-1,-1) coordinate (mblU)
    +(1,-1) coordinate (mbrU);
    \coordinate (brr) at (mbrU |- br);
    \coordinate (brmid) at ($(br)!0.5!(brr)$);
    \end{scope}
    \draw[functorM] (bl) -- (bl2) to[out=90, in=-90] (mbl) to[out=90, in=180] (mu.center) to[out=0, in=90] (mbr) -- (br)
                    (mu) -- (d) -- (mblU) to[out=90, in=180] (muU.center) to[out=0, in=90] (mbrU) -- (brr)
                    (muU) -- (dU);
    \draw[functorL, cross line] (blL) -- (mblL) to[out=90, in=180] (muL.center) to[out=0, in=90] (mbrL) to[out=-90, in=90] (brL2) -- (brL)
                    (muL) -- (dL) -- (mblUL) to[out=90, in=180] (muUL.center)
                    (muUL) -- (dUL)
                    (muUL.center) to[out=0, in=90] (mbrUL) to[out=-90, in=90] (brmid);
    \coordinate (tl) at (mblL |- dUL);
    \coordinate (cornerNW) at ($(tl) + (-0.5,0)$);
    \coordinate (cornerSE) at ($(brr) + (0.5,0)$);
    \draw (cornerNW) rectangle (cornerSE);
   \end{tikzpicture}
   \enspace=\enspace
   \begin{tikzpicture}[scale=0.75,baseline={([yshift=-0.5ex]current bounding box.center)}]
    \begin{scope}[on layer=over]
    \path coordinate (mu)
    +(0,0.5) coordinate (d);
    \path (mu) ++(-2.25,0) coordinate[dot, functorL] (muL)
    +(0,0.5) coordinate (dL)
    +(-0.75,-0.75) coordinate (mblL)
    +(0.75,-0.75) coordinate (mbrL);
    \path (mbrL) ++(0.75,-1.0) coordinate (brL2) ++(0,-0.5) coordinate (brL);
    \path (mblL) ++(0,-1.5) coordinate (blL);
    \path (dL) ++(+0.625,1) coordinate[dot, functorL] (muUL)
    +(0,1.0) coordinate (dUL)
    +(-0.625,-1) coordinate (mblUL)
    +(0.625,-1.25) coordinate (mbrUL);
    \path (d) ++(+0.625,1) coordinate[dot, functorM] (muU)
    +(0,1.0) coordinate (dU)
    +(-0.625,-1.25) coordinate (mblU)
    +(0.625,-1.0) coordinate (mbrU);
    \path (mbrU) ++(0,-0.5) coordinate[dot, functorM] (muNew)
                  +(-0.75,-0.75) coordinate (mblNew)
                  +(0.75,-0.75) coordinate (mbrNew);
    \path (mblNew) ++(-0.75,-1.0) coordinate (blNew2) ++(0,-0.5) coordinate (blNew);
    \coordinate (brmidL) at (mbrL |- blNew);
    \path (brmidL) ++(0,0.25) coordinate (brmidL2);
    \coordinate (brmid) at (mblNew |- blNew);
    \path (brmid) ++(0,0.25) coordinate (brmid2);
    \coordinate (brr) at (mbrNew |- blNew);
    \end{scope}
    \draw[functorM] (brmidL) -- (brmidL2) to[out=90, in=-90] (mblU)
                    (mblU) to[out=90, in=180] (muU.center) to[out=0, in=90] (mbrU) -- (muNew)
                    (muU) -- (dU)
                    (blNew) -- (blNew2) to[out=90, in=-90] (mblNew) to[out=90, in=180] (muNew.center) to[out=0, in=90] (mbrNew) -- (brr);
    \draw[functorL, cross line] (blL) -- (mblL) to[out=90, in=180] (muL.center) to[out=0, in=90] (mbrL) to[out=-90, in=90] (brL2) -- (brL)
                    (muL) -- (dL) -- (mblUL) to[out=90, in=180] (muUL.center)
                    (muUL) -- (dUL)
                    (muUL.center) to[out=0, in=90] (mbrUL) to[out=-90, in=90] (brmid2) -- (brmid);
    \coordinate (tl) at (mblL |- dUL);
    \coordinate (cornerNW) at ($(tl) + (-0.5,0)$);
    \coordinate (cornerSE) at ($(brr) + (0.5,0)$);
    \draw (cornerNW) rectangle (cornerSE);
   \end{tikzpicture}
   \\[5pt] =\enspace
   \begin{tikzpicture}[scale=0.75,baseline={([yshift=-0.5ex]current bounding box.center)}]
    \begin{scope}[on layer=over]
    \path coordinate[dot, functorL] (mu)
    +(0,0.5) coordinate (d)
    +(0.75,-0.75) coordinate (mbl)
    +(-0.75,-0.75) coordinate (mbr);
    \path (mbr) ++(0,-1.5) coordinate (br);
    \path (mbl) ++(0.75,-1.0) coordinate (bl2) ++(0,-0.5) coordinate (bl);
    \path (mu) ++(2.25,0) coordinate[dot, functorM] (muL)
    +(0,0.5) coordinate (dL)
    +(0.75,-0.75) coordinate (mblL)
    +(-0.75,-0.75) coordinate (mbrL);
    \path (mbrL) ++(-0.75,-1.0) coordinate (brL2) ++(0,-0.5) coordinate (brL);
    \path (mblL) ++(0,-1.5) coordinate (blL);
    \path (dL) ++(-0.625,1) coordinate[dot, functorM] (muUL)
    +(0,1.0) coordinate (dUL)
    +(0.625,-1) coordinate (mblUL)
    +(-0.625,-1.25) coordinate (mbrUL);
    \path (d) ++(-1,1) coordinate[dot, functorL] (muU)
    +(0,1.0) coordinate (dU)
    +(1,-1) coordinate (mblU)
    +(-1,-1) coordinate (mbrU);
    \coordinate (brr) at (mbrU |- br);
    \coordinate (brmid) at ($(br)!0.5!(brr)$);
    \end{scope}
    \draw[functorM] (blL) -- (mblL) to[out=90, in=0] (muL.center) to[out=180, in=90] (mbrL) to[out=-90, in=90] (brL2) -- (brL)
                    (muL) -- (dL) -- (mblUL) to[out=90, in=0] (muUL.center)
                    (muUL) -- (dUL)
                    (muUL.center) to[out=180, in=90] (mbrUL) to[out=-90, in=90] (brmid);
    \draw[functorL, cross line] (bl) -- (bl2) to[out=90, in=-90] (mbl) to[out=90, in=0] (mu.center) to[out=180, in=90] (mbr) -- (br)
                    (mu) -- (d) -- (mblU) to[out=90, in=0] (muU.center) to[out=180, in=90] (mbrU) -- (brr)
                    (muU) -- (dU);
    \coordinate (tl) at (mblL |- dUL);
    \coordinate (cornerNW) at ($(tl) + (0.5,0)$);
    \coordinate (cornerSE) at ($(brr) + (-0.5,0)$);
    \draw (cornerNW) rectangle (cornerSE);
   \end{tikzpicture}
   &\enspace=\enspace
   \begin{tikzpicture}[scale=0.75,baseline={([yshift=-0.5ex]current bounding box.center)}]
    \begin{scope}[on layer=over]
    \path coordinate[dot, functorL] (mu)
    +(0,0.5) coordinate (d)
    +(0.75,-0.75) coordinate (mbl)
    +(-0.75,-0.75) coordinate (mbr);
    \path (mbr) ++(0,-1.5) coordinate (br);
    \path (mbl) ++(0.75,-1.0) coordinate (bl2) ++(0,-0.5) coordinate (bl);
    \path (mu) ++(2.25,0) coordinate[dot, functorM] (muL)
    +(0,0.5) coordinate (dL)
    +(0.75,-0.75) coordinate (mblL)
    +(-0.75,-0.75) coordinate (mbrL);
    \path (mbrL) ++(-0.75,-1.0) coordinate (brL2) ++(0,-0.5) coordinate (brL);
    \path (mblL) ++(0,-1.5) coordinate (blL);
    \path (dL) ++(-1.5,1) coordinate[dot, functorM] (muUL)
    +(0,1.0) coordinate (dUL)
    +(1.5,-1) coordinate (mblUL)
    +(-1.5,-1) coordinate (mbrUL);
    \path (d) ++(-1,1) coordinate[dot, functorL] (muU)
    +(0,1.0) coordinate (dU)
    +(1,-1) coordinate (mblU)
    +(-1,-1) coordinate (mbrU);
    \coordinate (brr) at (mbrU |- br);
    \coordinate (brmid) at ($(br)!0.5!(brr)$);
    \end{scope}
    \draw[functorM] (blL) -- (mblL) to[out=90, in=0] (muL.center) to[out=180, in=90] (mbrL) to[out=-90, in=90] (brL2) -- (brL)
                    (muL) to[out=90, in=0] (muUL.center)
                    (muUL) -- (dUL)
                    (muUL.center) to[out=180, in=90] (brmid);
    \draw[functorL, cross line] (bl) -- (bl2) to[out=90, in=-90] (mbl) to[out=90, in=0] (mu.center) to[out=180, in=90] (mbr) -- (br)
                    (mu) -- (d) -- (mblU) to[out=90, in=0] (muU.center) to[out=180, in=90] (mbrU) -- (brr)
                    (muU) -- (dU);
    \coordinate (tl) at (mblL |- dUL);
    \coordinate (cornerNW) at ($(tl) + (0.5,0)$);
    \coordinate (cornerSE) at ($(brr) + (-0.5,0)$);
    \draw (cornerNW) rectangle (cornerSE);
   \end{tikzpicture}
  \end{align*}
  \endgroup
  Here we use condition \ref{eq:distcomp1} for the first equality, associativity of the compositors for the next two equalities and \ref{eq:distcomp2} for the final equality.
  
  For the condition on unitors, consider the following.
  \\ 
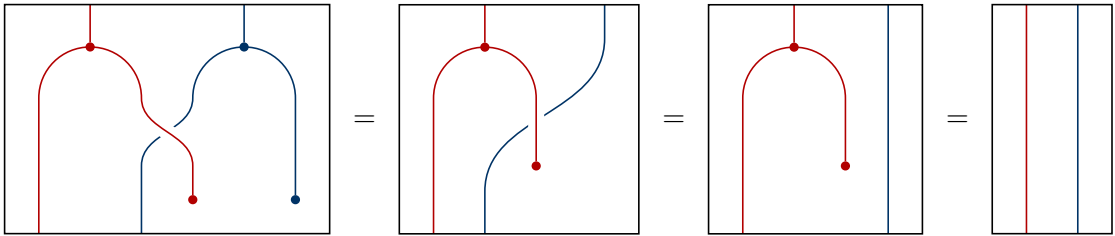
\vspace{-3pt plus 1pt}
  \begin{align*}
   \begin{tikzpicture}[scale=0.9,baseline={([yshift=-0.5ex]current bounding box.center)}]
    \begin{scope}[on layer=over]
    \path coordinate[dot, functorM] (mu)
    +(0,0.625) coordinate (d)
    +(-0.75,-0.75) coordinate (mbl)
    +(0.75,-0.75) coordinate (mbr);
    \path (mbr) ++(0,-1.5) coordinate[dot, functorM] (br);
    \path (mbl) ++(-0.75,-1.0) coordinate (bl2) ++(0,-1.0) coordinate (bl);
    \path (mu) ++(-2.25,0) coordinate[dot, functorL] (muL)
    +(0,0.625) coordinate (dL)
    +(-0.75,-0.75) coordinate (mblL)
    +(0.75,-0.75) coordinate (mbrL);
    \path (mbrL) ++(0.75,-1.0) coordinate (brL2) ++(0,-0.5) coordinate[dot, functorL] (brL);
    \path (mblL) ++(0,-2.0) coordinate (blL);
    \end{scope}
    \draw[functorM] (bl) -- (bl2) to[out=90, in=-90] (mbl) to[out=90, in=180] (mu.center) to[out=0, in=90] (mbr) -- (br)
                    (mu) -- (d);
    \draw[functorL, cross line] (blL) -- (mblL) to[out=90, in=180] (muL.center) to[out=0, in=90] (mbrL) to[out=-90, in=90] (brL2) -- (brL)
                    (muL) -- (dL);
    \coordinate (tl) at (dL -| blL);
    \coordinate (bbr) at (bl -| br);
    \coordinate (cornerNW) at ($(tl) + (-0.5,0)$);
    \coordinate (cornerSE) at ($(bbr) + (0.5,0)$);
    \draw (cornerNW) rectangle (cornerSE);
   \end{tikzpicture}
   &\enspace=\enspace
   \begin{tikzpicture}[scale=0.9,baseline={([yshift=-0.5ex]current bounding box.center)}]
    \begin{scope}[on layer=over]
    \path coordinate (mu)
    +(0,0.5) coordinate (d);
    \path (mu) ++(-1.75,-2.25) coordinate (bl2) ++(0,-0.625) coordinate (bl);
    \path (d) ++(-1.75,0) coordinate (dL) ++ (0,-0.625) coordinate[dot, functorL] (muL)
    +(-0.75,-0.75) coordinate (mblL)
    +(0.75,-0.75) coordinate (mbrL);
    \path (mbrL) ++(0,-1.0) coordinate[dot, functorL] (brL);
    \path (mblL) ++(0,-2.0) coordinate (blL);
    \end{scope}
    \draw[functorM] (bl) -- (bl2) to[out=90, in=-90]  (mu) -- (d);
    \draw[functorL, cross line] (blL) -- (mblL) to[out=90, in=180] (muL.center) to[out=0, in=90] (mbrL) -- (brL)
                    (muL) -- (dL);
    \coordinate (tl) at (dL -| blL);
    \coordinate (bbr) at (bl -| mu);
    \coordinate (cornerNW) at ($(tl) + (-0.5,0)$);
    \coordinate (cornerSE) at ($(bbr) + (0.5,0)$);
    \draw (cornerNW) rectangle (cornerSE);
   \end{tikzpicture}
   \enspace=\enspace
   \begin{tikzpicture}[scale=0.9,baseline={([yshift=-0.5ex]current bounding box.center)}]
    \begin{scope}[on layer=over]
    \path coordinate (mu)
    +(0,0.5) coordinate (d);
    \path (mu) ++(-1.75,-2.25) coordinate (bl2) ++(0,-0.625) coordinate (bl);
    \path (d) ++(-1.375,0) coordinate (dL) ++ (0,-0.625) coordinate[dot, functorL] (muL)
    +(-0.75,-0.75) coordinate (mblL)
    +(0.75,-0.75) coordinate (mbrL);
    \path (mbrL) ++(0,-1.0) coordinate[dot, functorL] (brL);
    \path (mblL) ++(0,-2.0) coordinate (blL);
    \end{scope}
    \draw[functorM] (bl -| d) -- (d);
    \draw[functorL, cross line] (blL) -- (mblL) to[out=90, in=180] (muL.center) to[out=0, in=90] (mbrL) -- (brL)
                    (muL) -- (dL);
    \coordinate (tl) at (dL -| blL);
    \coordinate (bbr) at (bl -| mu);
    \coordinate (cornerNW) at ($(tl) + (-0.5,0)$);
    \coordinate (cornerSE) at ($(bbr) + (0.5,0)$);
    \draw (cornerNW) rectangle (cornerSE);
   \end{tikzpicture}
   \enspace=\enspace
   \begin{tikzpicture}[scale=0.9,baseline={([yshift=-0.5ex]current bounding box.center)}]
    \begin{scope}[on layer=over]
    \path coordinate (mu)
    +(0,0.375) coordinate (d);
    \path (mu) (0,-3.0) coordinate (bl);
    \path (d) ++(-0.75,0) coordinate (dL) ++ (0,-1.125) coordinate (muL);
    \path (muL) ++(0,-2.25) coordinate (blL);
    \end{scope}
    \draw[functorM] (bl) -- (mu) -- (d);
    \draw[functorL, cross line] (blL) to[out=90, in=-90] (muL) -- (dL);
    \coordinate (tl) at (dL);
    \coordinate (bbr) at (bl -| mu);
    \coordinate (cornerNW) at ($(tl) + (-0.5,0)$);
    \coordinate (cornerSE) at ($(bbr) + (0.5,0)$);
    \draw (cornerNW) rectangle (cornerSE);
   \end{tikzpicture}
  \end{align*}
  The other unit condition is symmetric.
  Thus, $P\colon \B \times \C \to \D$ is a lax bifunctor.
  
  Note that $P(B,-)$ coincides with $M_B$ and $P(-,C)$ coincides with $L_C$ on objects. Thus we have canonically defined 2-morphisms for all 1-morphisms $f\colon B_1 \to B_2$ and $g\colon C_1 \to C_2$ defined by $\iota^{B_2}_C L_C(f) \colon L_C(f) \to M_{B_2}(\id_C)L_C(f) = P(f,C)$ and $M_B(g) \iota^{C_1}_B\colon M_B(g) \to M_B(g)L_{C_1}(\id_B) = P(B,g)$.
  
  We claim the family of identity morphisms on $P(B,C) = M_B(C)$ and the family of 2-morphisms $\iota_C^{B_2}L_C(f)$ define an oplax transformation $\kappa^B$ between $M_B$ and $P(B,-)$.
  
  Expanding the axiom for respecting composition we find that we require the following equality. (Note that because the 1-morphisms of the oplax transformation are the identity, they are omitted).
  \\ 
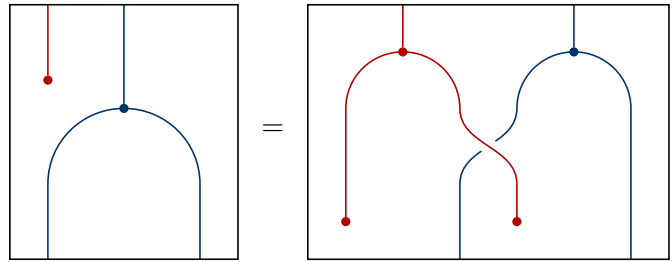
\vspace{-3pt plus 1pt}
  \begin{equation*}
   \begin{tikzpicture}[scale=1.0,baseline={([yshift=-0.5ex]current bounding box.center)}]
    \begin{scope}[on layer=over]
    \path coordinate[dot, functorM] (mu)
    +(0,1.375) coordinate (d)
    +(-1,-1) coordinate (mbl)
    +(1,-1) coordinate (mbr);
    \path (d -| mbl) ++(-0.0,0) coordinate (tl) ++(0,-1.0) coordinate[dot, functorL] (tl2);
    \path (mbr) ++(0,-1.0) coordinate (br);
    \path (mbl) ++(0,-1.0) coordinate (bl);
    \end{scope}
    \newcommand{\crossing}{(tl) -- (tl2)}
    \newcommand{\fork}{
     (bl) -- (mbl) to[out=90, in=180] (mu.center) to[out=0, in=90] (mbr) -- (br)
     (mu) -- (d)
    }
    \coordinate (cornerNW) at ($(tl) + (-0.5,0)$);
    \coordinate (cornerSE) at ($(br) + (0.5,0)$);
    \draw[functorM] \fork;
    \draw[functorL, cross line] \crossing;
    \draw (cornerNW) rectangle (cornerSE);
   \end{tikzpicture}
   \enspace=\enspace
   \begin{tikzpicture}[scale=1.0,baseline={([yshift=-0.5ex]current bounding box.center)}]
    \begin{scope}[on layer=over]
    \path coordinate[dot, functorM] (mu)
    +(0,0.625) coordinate (d)
    +(-0.75,-0.75) coordinate (mbl)
    +(0.75,-0.75) coordinate (mbr);
    \path (mbl) ++(-0.75,-1.0) coordinate (bl2) ++(0,-1.0) coordinate (bl);
    \coordinate (br) at (mbr |- bl);
    \path (mu) ++(-2.25,0) coordinate[dot, functorL] (muL)
    +(0,0.625) coordinate (dL)
    +(-0.75,-0.75) coordinate (mblL)
    +(0.75,-0.75) coordinate (mbrL);
    \path (mbrL) ++(0.75,-1.0) coordinate (brL2) ++(0,-0.5) coordinate[dot, functorL] (brL);
    \path (mblL) ++(0,-1.5) coordinate[dot, functorL] (blL);
    \end{scope}
    \draw[functorM] (bl) -- (bl2) to[out=90, in=-90] (mbl) to[out=90, in=180] (mu.center) to[out=0, in=90] (mbr) -- (br)
                    (mu) -- (d);
    \draw[functorL, cross line] (blL) -- (mblL) to[out=90, in=180] (muL.center) to[out=0, in=90] (mbrL) to[out=-90, in=90] (brL2) -- (brL)
                                (muL) -- (dL);
    \coordinate (tl) at (dL -| blL);
    \coordinate (bbr) at (bl -| br);
    \coordinate (cornerNW) at ($(tl) + (-0.5,0)$);
    \coordinate (cornerSE) at ($(bbr) + (0.5,0)$);
    \draw (cornerNW) rectangle (cornerSE);
   \end{tikzpicture}
  \end{equation*}
  But again, this is clearly true by condition \ref{eq:distunit2} and the unit law for lax functors.
  
  Expanding the unit law for respecting units gives the horizontal composition $\iota^B \ast \iota^C$ in both cases and thus this condition is also satisfied.
  Similarly, the two sides of the 2-dimensional naturality condition both become $M_B(\alpha) \ast \iota^C$ and so we have indeed defined an oplax transformation. The same argument gives that $\kappa^C\colon L_C \to P(-,C)$ is an oplax transformation.
 \end{proof}
 
 \begin{definition}
 We refer to the process of constructing a lax bifunctor from data of a distributive law as \emph{collation} and to the result as the \emph{collated bifunctor}.
 \end{definition}
 
 As indicated above, we may recover the usual distributive laws of monads.
\begin{proposition}
  Given monads $S$ and $T$ in a 2-category $\D$, we obtain a monad structure on $TS$ from a distributive law $\sigma\colon S T \to T S$.
\end{proposition}
\begin{proof}
  Simply apply \cref{thm:lax_bifunctor} with $L\colon 1 \to \D$ being the lax functor corresponding to $S$ and $M\colon 1 \to \D$ the lax functor corresponding to $T$. (Note the conditions \ref{eq:distnat1} and \ref{eq:distnat2} hold automatically, since the 2-category $1$ has no nontrivial 2-morphisms.)
\end{proof}
Of course, we also obtain distributive laws for comonads working with lax functors into $\D^\mathrm{co}$.

Unlike in the 1-dimensional setting, it is not in general possible to recover the data of the distributive law from the collated bifunctor. This is already clear in the monad case, since any composite monad $TS$ could also be obtained from the trivial distributive law between $TS$ and the identity monad.
Later in \cref{sec:converse} we will discuss conditions under which it is possible to go backwards. In particular, this is the case for invertible distributive laws between pseudofunctors.
In order to make sense of this reverse process we will need to understand when two distributive laws are `the same'. This motivates an exploration of their (2-)categorical structure.

\section{Uncurrying and the 2-category of distributive laws}\label{sec:uncurrying}

Even in the 1-dimensional case it is clear that there ought to be a version of the bifunctor theorem that applies to morphisms. Let $\B$, $\C$ and $\D$ be categories and let the functors $L^1_C$ and $M^1_B$ form one `distributive law' and functors $L^2_C$ and $M^2_B$ form another.
Suppose we have natural transformations $\theta^C\colon L^1_C \to L^2_C$ and $\theta^B \colon M_B^1 \to M_B^2$ satisfying that $\theta^B_C = \theta^C_B$. Then we may `collate' these families of natural transformations to give a natural transformation $\theta\colon P^1 \to P^2$ defined by $\theta_{B,C} = \theta^B_C$, where $P^1$ and $P^2$ denote the collations of the respective families of functors. We expect an analogue of this result in the lax functor setting, as well as a corresponding result for 2-morphisms. These families of natural transformations may be taken as morphisms of distributive laws. Indeed, in the 1-dimensional setting the resulting category $\Dist(\B,\C,\D)$ is equivalent to $\Hom(\B \times \C,\D)$.
It is natural then to ask what the 1-morphisms and 2-morphisms are in our lax setting. Here it is helpful to take inspiration from the category of distributive laws of monads.

In the monad setting, distributive laws in a 2-category $\C$ may be thought of as monads in the 2-category of monads in $\C$ (see \cite{street1972formal}). Using the correspondence with lax functors this means that these distributive laws correspond to objects of the 2-category $\Laxop(1,\Laxop(1,X))$. Then the process of constructing a monad on $X$ from a distributive law corresponds to a transformation sending objects in $\Laxop(1,\Laxop(1,X))$ to objects in $\Laxop(1,X) \cong \Laxop(1 {\times} 1,X)$.
This bears a strong resemblance to the process of \emph{uncurrying}. In this section we make this connection precise in our more general setting. In doing so, light is shed on the resemblance between the conditions of distributive laws and (op)lax transformations.
This will also provide a link to related work by Nikolić \cite{nikolic2018strictification}.

\begin{lemma}\label{prp:laxopbijectioonobjects}
  There is a bijection between the objects of $\Laxop(\B,\Laxop(\C,\D))$ and $\Dist(\B,\C,\D)$.
\end{lemma}

\begin{proof}
    Let $(Q,\gamma,\iota) \colon \B \to \Laxop(\C,\D)$ be a lax functor. We shall specify the data of a distributive law $\sigma$ between families $L$ and $M$.
    
    We define the family $M$ by $M_B = Q(B)$. Now let $f \colon B_1 \to B_2$ be a 1-morphism in $\B$ and consider the oplax transformation $Q(f) \colon M_{B_1} \to M_{B_2}$.
    Notice that if $g \colon C_1 \to C_2$ is a 1-morphism in $\C$, we have a 2-morphism $Q(f)_g\colon Q(f)_{C_2}M_{B_1}(g) \to M_{B_2}(g)Q(f)_{C_1}$ in $\D$. So if $L_C(f)$ were equal to $Q(f)_C$, we could treat $Q(f)_g$ as a candidate for $\sigma_{f,g}$. This suggests the following definition for the family $L$.
    
    Let $C$ be an object in $\C$ and define $(L_C,\gamma^C,\iota^C)$ as follows:
    \begin{itemize}
        \item $L_C(B) = M_B(C) = Q(B)(C)$ for objects $B$ in $\B$,
        \item $L_C(f) = Q(f)_C$ for 1-morphisms $f$ in $\B$,
        \item $L_C(\theta) = Q(\theta)_C$ for 2-morphisms $\theta$ in $\B$,
        \item $\iota^C_B = (\iota_B)_C\colon \id_{L_C(B)} \to L_C(\id_B)$,
        \item $\gamma^C_{f_2,f_1} = (\gamma_{f_2,f_1})_C \colon L_C(f_2)L_C(f_1) \to L_C(f_2f_1)$.
    \end{itemize}
    
    We must show that $(L_C,\gamma^C,\iota^C)$ satisfies the conditions in the definition of a lax functor. But this follows immediately from the fact that $(Q,\gamma,\iota)$ satisfies these equations and that composition is computed componentwise.
    
    We may now define $\sigma_{f,g} = Q(f)_g\colon L_{C_2}(f)M_{B_1}(g) \to M_{B_2}(g)L_{C_1}(f)$. It remains to show that $\sigma$ satisfies the six requisite conditions.
    
    Conditions \ref{eq:distcomp1}, \ref{eq:distunit1} and \ref{eq:distnat1} are automatically satisfied as $Q(f)$ is an oplax transformation. Conditions \ref{eq:distcomp2} and \ref{eq:distunit2} follow from the fact that $\gamma_{f_2,f_1}$ and $\iota_B$ are modifications and condition \ref{eq:distnat2} follows because $Q(\alpha)$ is a modification.
    
    In the other direction, suppose that $\sigma$ is a distributive law between families $L$ and $M$.
    We define a lax functor $(Q,\gamma,\iota)\colon \B \to \Laxop(\C,\D)$ as follows.
    Firstly, for objects $B$ in $\B$ we set $Q(B) = M_B$. 
    
    Next, for each 1-morphism $f$ in $\B$ we must give an oplax transformation $Q(f)$. It is enough to specify it componentwise. Let $Q(f)_C = L_C(f)$ and let $Q(f)_g = \sigma_{f,g}$. Conditions \ref{eq:distcomp1}, \ref{eq:distunit1} and \ref{eq:distnat1} together imply that this is an oplax transformation.
    
    For a 2-morphism $\theta$ in $\B$ we require a modification $Q(\theta)$, which we again define componentwise by setting $Q(\theta)_C = L_C(\theta)$. Condition \ref{eq:distnat2} gives that this is indeed a modification.
    Note that $Q$ preserves composition of 2-morphisms, since each $L_C$ does.
    
    Finally, both the compositor $\gamma$ and the unitor $\iota$ should be modifications. We may define these componentwise as $(\gamma_{f_2,f_1})_C = \gamma^C_{f_2,f_1}$ and $(\iota_{B})_C = \iota^C_B$. These are modifications by conditions \ref{eq:distcomp2} and \ref{eq:distunit2} respectively.
    
    Because each $(L_C,\gamma^C,\iota^C)$ is a lax functor, we have that $\gamma$ and $\iota$ satisfy the requisite diagrams componentwise. This is enough to deduce that $(Q,\gamma,\iota)$ is a lax functor.
    
    It is apparent that these processes are inverses.
\end{proof}

\begin{remark}
The above proof explains the resemblance between conditions \ref{eq:distcomp1}, \ref{eq:distunit1} and \ref{eq:distnat1} and the axioms for an oplax transformation.

Incidentally, there is also a `dual' bijection between $\Dist(\B,\C,\D)$ and lax functors from $\C$ to $\mathrm{Lax}(\B,\D)$, the 2-category lax functors, lax transformations and modifications. This correspondence explains the relation between conditions \ref{eq:distcomp2}, \ref{eq:distunit2} and \ref{eq:distnat2} and the axioms of a lax transformation.
\end{remark}

Using similar ideas, we may now unwind the 1-morphisms and 2-morphisms in the 2-category $\Laxop(\B,\Laxop(\C,\D))$ to produce appropriate definitions for 1-morphisms and 2-morphisms of distributive laws.

Suppose we are given an oplax transformation $\widehat{\theta}$ between lax functors $Q^1, Q^2\colon \mathcal{B} \to \Laxop(\mathcal{C}, \mathcal{D})$. From $\widehat{\theta}$ we will extract families of oplax transformations $\theta^B\colon M^1_B \to M^2_B$ and $\theta^C\colon L^1_C \to L^2_C$, where the families $M^1$ and $L^1$ and $M^2$ and $L^2$ have been extracted from $Q^1$ and $Q^2$ respectively as in \cref{prp:laxopbijectioonobjects}.

The oplax transformations $\theta^B$ are given by $\widehat{\theta}_B$ and are evidently seen to be oplax transformations between $Q^1(B) = M^1_B$ and $Q^2(B) = M^2_B$.

On the other hand the oplax transformations $\theta^C$ can be defined componentwise by $(\theta^C)_B = (\widehat{\theta}_B)_C$ and $(\theta^C)_f = (\widehat{\theta}_f)_C$. The resulting family $\theta^C$ is seen to be an oplax transformation by considering the $C$ components of the corresponding axioms for $\widehat{\theta}$.

By construction $\theta^B_C = \theta^C_B$ for all $B$ and $C$. Moreover, since $\widehat{\theta}_f$ is a modification, we find that the families satisfy the Yang--Baxter equation. This motivates the following definition.

\begin{definition}\label{def:morphismofdist}
  Let $\sigma^1, \sigma^2 \in \Dist(\B, \C, \D)$ be distributive laws between families of lax functors $L^1$ and $M^1$ and $L^2$ and $M^2$ respectively. We define a \emph{morphism of distributive laws} between them to be a pair of families of oplax transformations $\theta^C\colon L_C^1 \to L_C^2$ and $\theta^B\colon M_B^1 \to M_B^2$ such that $\theta^C_B = \theta^B_C$ and which satisfy the Yang--Baxter equation.
  
  
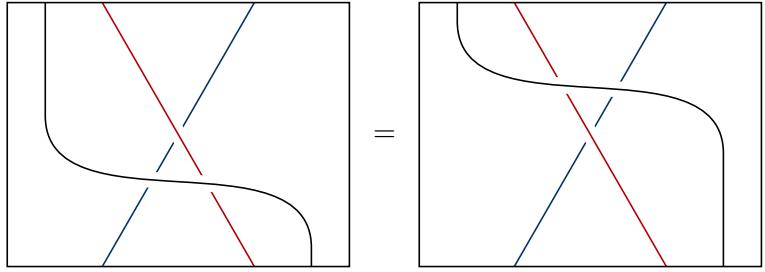
\begin{equation*} 
   \begin{tikzpicture}[scale=1.0,baseline={([yshift=-0.5ex]current bounding box.center)}]
    \begin{scope}[on layer=over]
    \path coordinate (d) ++(2,-3.5) coordinate (b);
    \path (d) ++(-0.75,0) coordinate (tl) ++(0,-1.5) coordinate (tl2);
    \path (b) ++ (0.75,0) coordinate (brr) ++(0,0.25) coordinate (rr);
    \end{scope}
    \draw[functorM] (d -| b) -- (d |- b);
    \draw[functorL, cross line] (d) -- (b);
    \draw[cross line] (tl) -- (tl2) to[out=270, in=90] (rr) -- (brr);
    \coordinate (cornerNW) at ($(tl) + (-0.5,0)$);
    \coordinate (cornerSE) at ($(brr) + (0.5,0)$);
    \draw (cornerNW) rectangle (cornerSE);
   \end{tikzpicture}
   \enspace =\enspace
   \begin{tikzpicture}[scale=1.0,baseline={([yshift=-0.5ex]current bounding box.center)}]
    \begin{scope}[on layer=over]
    \path coordinate (d) ++(2,-3.5) coordinate (b);
    \path (d) ++(-0.75,0) coordinate (tl) ++(0,-0.25) coordinate (tl2);
    \path (b) ++ (0.75,0) coordinate (brr) ++(0,1.5) coordinate (rr);
    \end{scope}
    \draw[functorM] (d -| b) -- (d |- b);
    \draw[functorL, cross line] (d) -- (b);
    \draw[cross line] (tl) -- (tl2) to[out=270, in=90] (rr) -- (brr);
    \coordinate (cornerNW) at ($(tl) + (-0.5,0)$);
    \coordinate (cornerSE) at ($(brr) + (0.5,0)$);
    \draw (cornerNW) rectangle (cornerSE);
   \end{tikzpicture}
  \end{equation*}
  
  We will write $\overline{\theta}$ for such a morphism where the oplax transformations are denoted by $\theta^C$ and $\theta^B$.
\end{definition}

Next we study 2-morphisms. Let $\widehat{\beth}\colon \widehat{\theta} \to \widehat{\zeta}$ be a modification, where $\widehat{\theta}, \widehat{\zeta}\colon Q^1 \to Q^2$ are oplax transformations in $\Laxop(\B,\Laxop(\C,\D))$. We extract families of modifications $\beth^B\colon \theta^B \to \zeta^B$ and $\beth^C\colon \theta^C \to \zeta^C$ as follows.

The modifications $\beth^B$ are simply the maps $\widehat{\beth}_B\colon \widehat{\theta}_B \to \widehat{\zeta}_B$. We define the modification $\beth^C$ componentwise by $\beth^C_B = (\widehat{\beth}_B)_C$. This definition type checks since $\theta^B_C = \theta^C_B$ and $\zeta^B_C = \zeta^C_B$ and $\beth^C$ is seen to be a modification by evaluating the modification axiom of $\widehat{\beth}$ componentwise. We also note that $\beth^B_C = \beth^C_B$.

\begin{definition}
  Let $\sigma^1, \sigma^2 \in \Dist(\B, \C, \D)$ be distributive laws and let $\overline{\theta}, \overline{\zeta}\colon \sigma^1 \to \sigma^2$ be morphisms of distributive laws. We define a \emph{2-morphism of distributive laws} to be a pair of families of modifications $\beth^B\colon \theta^B \to \zeta^B$ and $\beth^C\colon \theta^C \to \zeta^C$ such that $\beth^B_C = \beth^C_B$.
  We will write $\overline{\beth}$ for such a 2-morphism where the modifications are denoted by $\beth^B$ and $\beth^C$.
\end{definition}

These form a 2-category of distributive laws.
\begin{definition}\label{def:2catDist}
  Let $\mathcal{B}$, $\mathcal{C}$ and $\mathcal{D}$ be $2$-categories. We define a $2$-category $\Dist(\B,\C\,\D)$ whose objects are distributive laws of lax functors from $\B$ and $\C$ to $\D$, whose 1-morphisms are morphisms of distributive laws and whose 2-morphisms are 2-morphisms of distributive laws.
  
  Composition of 1-morphisms is computed by composing the relevant families of oplax transformations componentwise. Vertical and horizontal composition of 2-morphisms is computed componentwise in a similar way.
  It is easy to see that the relevant conditions on these families are stable under composition and so $\Dist(\B,\C\,\D)$ is indeed a 2-category. 
\end{definition}

\begin{theorem}\label{thm:equivalence_between_dist_and_iterated_functor_cat}
    The 2-categories $\Dist(\B,\C,\D)$ and $\Laxop(\B,\Laxop(\C,\D))$ are isomorphic.
\end{theorem}

\begin{proof}
  We have shown above how to transform objects, 1-morphisms and 2-morphisms in $\Laxop(\B,\Laxop(\C,\D))$ into objects, 1-morphisms and 2-morphisms in $\Dist(\B,\C,\D)$. It is not hard to see that this assignment is 2-functorial.
  All that remains is to show that this process is invertible. We have already shown this on objects in \Cref{prp:laxopbijectioonobjects}.
  
  Let $\overline{\theta} \colon \sigma^1 \to \sigma^2$ be a morphism of distributive laws given by families $\theta^B\colon M^1(B) \to M^2(B)$ and $\theta^C\colon L^1(C) \to L^2(C)$ of oplax transformations. Let $Q^1, Q^2 \colon \B \to \Laxop(\C,\D)$ be the lax functors corresponding to $\sigma^1$ and $\sigma^2$ respectively as in \cref{prp:laxopbijectioonobjects}.
  
  We may define an oplax transformation $\widehat{\theta}\colon Q^1 \to Q^2$ by setting $\widehat{\theta}_B = \theta^B$ and taking the $C$ component of the modification $\widehat{\theta}_f$ to be $\theta^C_f$. As before, the Yang--Baxter equations satisfied by the families $\theta^B$ and $\theta^C$ are precisely what is needed to show $\widehat{\theta}_f$ is a modification.
  To see that $\widehat{\theta}$ is an oplax transformation it suffices to check that the necessary axioms hold componentwise, but these reduce to the oplax transformation axioms satisfied by $\theta^C$. This is easily seen to be inverse to the construction of $\overline{\theta}$ from $\widehat{\theta}$.
  
  Now suppose $\overline{\zeta}\colon \sigma^1 \to \sigma^2$ is another morphism of distributive laws and let $\overline{\beth}\colon \overline{\theta} \to \overline{\zeta}$ be a 2-morphism of distributive laws. We define $\widehat{\beth}\colon \widehat{\theta} \to \widehat{\zeta}$ by $\widehat{\beth}_B = \beth^B$. It is not hard to see that is a modification by working componentwise and using that $\beth^B_C = \beth^C_B$. It is again clear this is inverse the construction of $\overline{\beth}$ from $\widehat{\beth}$.
\end{proof}

Now that we have defined the 2-category $\Dist(\B,\C,\D)$, we would like to extend collation to a 2-functor $K\colon \Dist(\B,\C,\D) \to \Laxop(B \times C,D)$. To this end we describe the collation of 1-morphisms and 2-morphisms of distributive laws.

\begin{proposition}\label{prp:collate_oplax_transformations}
  Let $\sigma^1, \sigma^2 \in \Dist(\B, \C, \D)$ be distributive laws between families of lax functors $L^1$ and $M^1$ and between families $L^2$ and $M^2$ respectively and let $P^1$ and $P^2$ be the respective collated bifunctors. Let $\overline{\theta}\colon \sigma^1 \to \sigma^2$ be a morphism of distributive laws.
  
  Then we may define an oplax transformation $\theta\colon P^1 \to P^2$ by $\theta_{B,C} = \theta^B_C$ and with $\theta_{g,f}$ given by $M^2_B(g)\theta^C_f \circ \theta^B_g L^1_C(f)$ as depicted in the following string diagram.
  \\ 
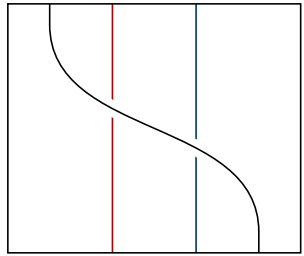
\vspace{-3pt plus 1pt}
  \begin{center}
   \begin{tikzpicture}[scale=1.1,baseline={([yshift=-0.5ex]current bounding box.center)}]
    \begin{scope}[on layer=over]
    \path coordinate (dl) ++(0,-3.0) coordinate (bl);
    \path (dl) ++(1.0,0) coordinate (dr) ++(0,-3.0) coordinate (br);
    \path (dl) ++(-0.75,0) coordinate (tl) ++(0,-0.25) coordinate (tl2);
    \path (br) ++(0.75,0) coordinate (brr) ++(0,0.25) coordinate (rr);
    \end{scope}
    \draw[functorM] (dr) -- (br);
    \draw[functorL] (dl) -- (bl);
    \draw[cross line] (tl) -- (tl2) to[out=270, in=90] (rr) -- (brr);
    \coordinate (cornerNW) at ($(tl) + (-0.5,0)$);
    \coordinate (cornerSE) at ($(brr) + (0.5,0)$);
    \draw (cornerNW) rectangle (cornerSE);
   \end{tikzpicture}
  \end{center}
  
  Moreover, if we define $\theta_{B,-}\colon P^1(B,-) \to P^2(B,-)$ and $\theta_{-,C}\colon P^1(-,C) \to P^2(-,C)$ to be the oplax transformations obtained by restricting $P^1$ and $P^2$ in each component, then we have $\theta_{B,-}\circ\kappa^{1,B} = \kappa^{2,B}\theta^B $ and $\theta_{-,C}\circ\kappa^{1,C} = \kappa^{2,C}\theta^C$, where the `$\kappa$' maps are defined as in \cref{thm:lax_bifunctor}.
\end{proposition}

\begin{proof}
  We must check that $\theta$ as defined above is an oplax transformation. The unit and naturality conditions follow easily, as $\theta^C$ and $\theta^B$ are individually are oplax transformations.
  \\ 
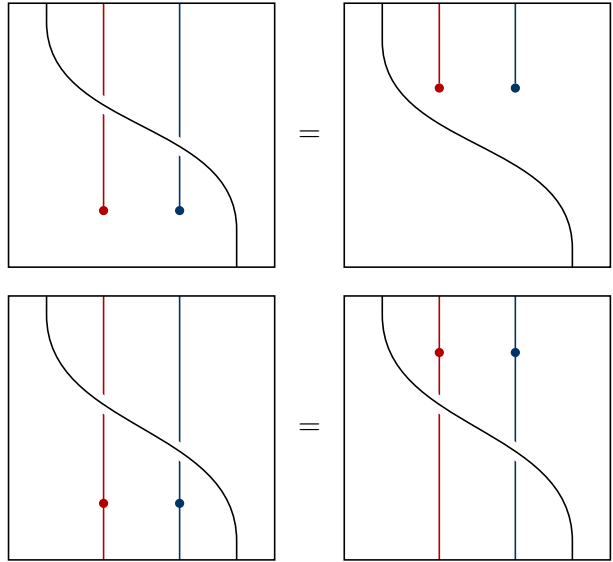
\vspace{-3pt plus 1pt}
  \begin{equation*}
   \begin{tikzpicture}[scale=1.0,baseline={([yshift=-0.5ex]current bounding box.center)}]
    \begin{scope}[on layer=over]
    \path coordinate (dl) ++(0,-2.75) coordinate[dot, functorL] (bl2) ++(0,-0.75) coordinate (bl);
    \path (dl) ++(1.0,0) coordinate (dr) ++(0,-2.75) coordinate[dot, functorM] (br2)  ++(0,-0.75) coordinate (br);
    \path (dl) ++(-0.75,0) coordinate (tl) ++(0,-0.25) coordinate (tl2);
    \path (br) ++(0.75,0) coordinate (brr) ++(0,0.5) coordinate (rr);
    \end{scope}
    \draw[functorM] (dr) -- (br2);
    \draw[functorL] (dl) -- (bl2);
    \draw[cross line] (tl) -- (tl2) to[out=270, in=90] (rr) -- (brr);
    \coordinate (cornerNW) at ($(tl) + (-0.5,0)$);
    \coordinate (cornerSE) at ($(brr) + (0.5,0)$);
    \draw (cornerNW) rectangle (cornerSE);
   \end{tikzpicture}
   \enspace=\enspace
   \begin{tikzpicture}[scale=1.0,baseline={([yshift=-0.5ex]current bounding box.center)}]
    \begin{scope}[on layer=over]
    \path coordinate (dl) ++(0,-1.125) coordinate[dot, functorL] (bl2) ++(0,-2.375) coordinate (bl);
    \path (dl) ++(1.0,0) coordinate (dr) ++(0,-1.125) coordinate[dot, functorM] (br2)  ++(0,-2.375) coordinate (br);
    \path (dl) ++(-0.75,0) coordinate (tl) ++(0,-0.5) coordinate (tl2);
    \path (br) ++(0.75,0) coordinate (brr) ++(0,0.25) coordinate (rr);
    \end{scope}
    \draw[functorM] (dr) -- (br2);
    \draw[functorL] (dl) -- (bl2);
    \draw[cross line] (tl) -- (tl2) to[out=270, in=90] (rr) -- (brr);
    \coordinate (cornerNW) at ($(tl) + (-0.5,0)$);
    \coordinate (cornerSE) at ($(brr) + (0.5,0)$);
    \draw (cornerNW) rectangle (cornerSE);
   \end{tikzpicture}
  \end{equation*}
  \begin{equation*}
   \begin{tikzpicture}[scale=1.0,baseline={([yshift=-0.5ex]current bounding box.center)}]
    \begin{scope}[on layer=over]
    \path coordinate (dl) ++(0,-2.75) coordinate[dot, functorL] (bl2) ++(0,-0.75) coordinate (bl);
    \path (dl) ++(1.0,0) coordinate (dr) ++(0,-2.75) coordinate[dot, functorM] (br2) ++(0,-0.75) coordinate (br);
    \path (dl) ++(-0.75,0) coordinate (tl) ++(0,-0.25) coordinate (tl2);
    \path (br) ++(0.75,0) coordinate (brr) ++(0,0.25) coordinate (rr);
    \end{scope}
    \draw[functorM] (dr) -- (br);
    \draw[functorL] (dl) -- (bl);
    \draw[cross line] (tl) -- (tl2) to[out=270, in=90] (rr) -- (brr);
    \coordinate (cornerNW) at ($(tl) + (-0.5,0)$);
    \coordinate (cornerSE) at ($(brr) + (0.5,0)$);
    \draw (cornerNW) rectangle (cornerSE);
   \end{tikzpicture}
   \enspace=\enspace
   \begin{tikzpicture}[scale=1.0,baseline={([yshift=-0.5ex]current bounding box.center)}]
    \begin{scope}[on layer=over]
    \path coordinate (dl) ++(0,-0.75) coordinate[dot, functorL] (bl2) ++(0,-2.75) coordinate (bl);
    \path (dl) ++(1.0,0) coordinate (dr) ++(0,-0.75) coordinate[dot, functorM] (br2) ++(0,-2.75) coordinate (br);
    \path (dl) ++(-0.75,0) coordinate (tl) ++(0,-0.25) coordinate (tl2);
    \path (br) ++(0.75,0) coordinate (brr) ++(0,0.25) coordinate (rr);
    \end{scope}
    \draw[functorM] (dr) -- (br);
    \draw[functorL] (dl) -- (bl);
    \draw[cross line] (tl) -- (tl2) to[out=270, in=90] (rr) -- (brr);
    \coordinate (cornerNW) at ($(tl) + (-0.5,0)$);
    \coordinate (cornerSE) at ($(brr) + (0.5,0)$);
    \draw (cornerNW) rectangle (cornerSE);
   \end{tikzpicture}
  \end{equation*}
  
  The compositor condition is proved using the Yang--Baxter equation.
  \\ 
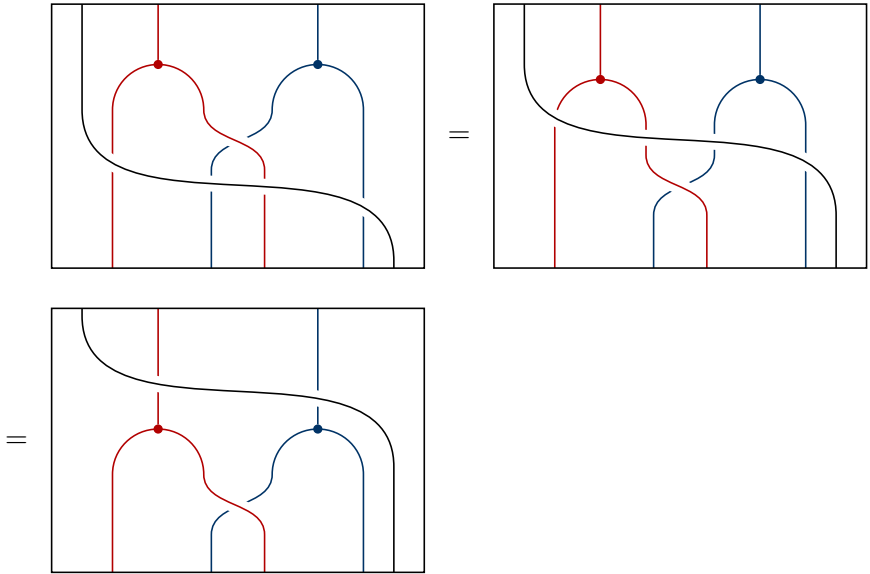
\vspace{-3pt plus 1pt}
  \begin{align*}
   \begin{tikzpicture}[scale=0.8,baseline={([yshift=-0.5ex]current bounding box.center)}]
    \begin{scope}[on layer=over]
    \path coordinate[dot, functorM] (mu)
    +(0,1.0) coordinate (d)
    +(-0.75,-0.75) coordinate (mbl)
    +(0.75,-0.75) coordinate (mbr);
    \path (mbr) ++(0,-2.625) coordinate (br);
    \path (mbl) ++(-1.0,-1.0) coordinate (bl2) ++(0,-1.625) coordinate (bl);
    \path (mu) ++(-2.625,0) coordinate[dot, functorL] (muL)
    +(0,1.0) coordinate (dL)
    +(-0.75,-0.75) coordinate (mblL)
    +(0.75,-0.75) coordinate (mbrL);
    \path (mbrL) ++(1.0,-1.0) coordinate (brL2) ++(0,-1.625) coordinate (brL);
    \path (mblL) ++(0,-2.625) coordinate (blL);
    \coordinate (tl) at (dL -| blL);
    \path (tl) ++(-0.5,0) coordinate (tll) ++(0,-1.75) coordinate (tll2);
    \path (br) ++(0.5,0) coordinate (brr) ++(0,0.125) coordinate (brr2);
    \end{scope}
    \draw[functorM] (bl) -- (bl2) to[out=90, in=-90] (mbl) to[out=90, in=180] (mu.center) to[out=0, in=90] (mbr) -- (br)
                    (mu) -- (d);
    \draw[functorL, cross line] (blL) -- (mblL) to[out=90, in=180] (muL.center) to[out=0, in=90] (mbrL) to[out=-90, in=90] (brL2) -- (brL)
                                (muL) -- (dL);
    \draw[cross line] (tll) -- (tll2) to[out=270, in=90] (brr2) -- (brr);
    \coordinate (cornerNW) at ($(tll) + (-0.5,0)$);
    \coordinate (cornerSE) at ($(brr) + (0.5,0)$);
    \draw (cornerNW) rectangle (cornerSE);
   \end{tikzpicture}
   &\enspace=\enspace 
   \begin{tikzpicture}[scale=0.8,baseline={([yshift=-0.5ex]current bounding box.center)}]
    \begin{scope}[on layer=over]
    \path coordinate[dot, functorM] (mu)
    +(0,1.25) coordinate (d)
    +(-0.75,-0.75) coordinate (mbl)
    +(0.75,-0.75) coordinate (mbr);
    \path (mbr) ++(0,-2.375) coordinate (br);
    \path (mbl) ++(0,-0.5) coordinate (mbl2) ++(-1.0,-1.0) coordinate (bl2) ++(0,-0.875) coordinate (bl);
    \path (mu) ++(-2.625,0) coordinate[dot, functorL] (muL)
    +(0,1.25) coordinate (dL)
    +(-0.75,-0.75) coordinate (mblL)
    +(0.75,-0.75) coordinate (mbrL);
    \path (mbrL) ++ (0,-0.5) coordinate (mbrL2) ++(1.0,-1.0) coordinate (brL2) ++(0,-0.875) coordinate (brL);
    \path (mblL) ++(0,-2.375) coordinate (blL);
    \coordinate (tl) at (dL -| blL);
    \path (tl) ++(-0.5,0) coordinate (tll) ++(0,-1.0) coordinate (tll2);
    \path (br) ++(0.5,0) coordinate (brr) ++(0,0.875) coordinate (brr2);
    \end{scope}
    \draw[functorM] (bl) -- (bl2) to[out=90, in=-90] (mbl2) -- (mbl) to[out=90, in=180] (mu.center) to[out=0, in=90] (mbr) -- (br)
                    (mu) -- (d);
    \draw[functorL, cross line] (blL) -- (mblL) to[out=90, in=180] (muL.center) to[out=0, in=90] (mbrL) -- (mbrL2) to[out=-90, in=90] (brL2) -- (brL)
                               (muL) -- (dL);
    \draw[cross line] (tll) -- (tll2) to[out=270, in=90] (brr2) -- (brr);
    \coordinate (cornerNW) at ($(tll) + (-0.5,0)$);
    \coordinate (cornerSE) at ($(brr) + (0.5,0)$);
    \draw (cornerNW) rectangle (cornerSE);
   \end{tikzpicture}
   \\[5pt] =\enspace 
   \begin{tikzpicture}[scale=0.8,baseline={([yshift=-0.5ex]current bounding box.center)}]
    \begin{scope}[on layer=over]
    \path coordinate[dot, functorM] (mu)
    +(0,2.0) coordinate (d)
    +(-0.75,-0.75) coordinate (mbl)
    +(0.75,-0.75) coordinate (mbr);
    \path (mbr) ++(0,-1.625) coordinate (br);
    \path (mbl) ++(-1.0,-1.0) coordinate (bl2) ++(0,-0.625) coordinate (bl);
    \path (mu) ++(-2.625,0) coordinate[dot, functorL] (muL)
    +(0,2.0) coordinate (dL)
    +(-0.75,-0.75) coordinate (mblL)
    +(0.75,-0.75) coordinate (mbrL);
    \path (mbrL) ++(1.0,-1.0) coordinate (brL2) ++(0,-0.625) coordinate (brL);
    \path (mblL) ++(0,-1.625) coordinate (blL);
    \coordinate (tl) at (dL -| blL);
    \path (tl) ++(-0.5,0) coordinate (tll) ++(0,-0.125) coordinate (tll2);
    \path (br) ++(0.5,0) coordinate (brr) ++(0,1.75) coordinate (brr2);
    \end{scope}
    \draw[functorM] (bl) -- (bl2) to[out=90, in=-90] (mbl) to[out=90, in=180] (mu.center) to[out=0, in=90] (mbr) -- (br)
                    (mu) -- (d);
    \draw[functorL, cross line] (blL) -- (mblL) to[out=90, in=180] (muL.center) to[out=0, in=90] (mbrL) to[out=-90, in=90] (brL2) -- (brL)
                                (muL) -- (dL);
    \draw[cross line] (tll) -- (tll2) to[out=270, in=90] (brr2) -- (brr);
    \coordinate (cornerNW) at ($(tll) + (-0.5,0)$);
    \coordinate (cornerSE) at ($(brr) + (0.5,0)$);
    \draw (cornerNW) rectangle (cornerSE);
   \end{tikzpicture}
  \end{align*}
  
  Finally, we must show that $\theta_{B,-} \circ \kappa^{1,B} = \kappa^{2,B} \theta^B$ and $\theta_{-,C} \circ \kappa^{1,C} = \kappa^{2,C} \theta^C$.
  We immediately get agreement on 1-morphisms, since the 1-morphisms of each $\kappa$ are identities. They are seen to be equal on 2-morphisms by the interchange law.
\end{proof}

\begin{remark}
 Since there is a isomorphism between $\Laxop(\C\op,\D\op)\op$ and the 2-category $\mathrm{Lax}(\C,\D)$ of lax functors from $\C$ to $\D$ and \emph{lax} transformations we can conclude a similar result to \cref{prp:collate_oplax_transformations} for lax transformations by duality.
\end{remark}

The following proposition shows how to obtain modifications between oplax transformations of bifunctors by collation.
\begin{proposition}\label{prp:collate_modifications}
  Let $\sigma^1, \sigma^2 \in \Dist(\B, \C, \D)$ be distributive laws between families of lax functors $L^1$ and $M^1$ and $L^2$ and $M^2$ respectively. Let $\overline{\theta}$ and $\overline{\zeta}$ be morphisms of distributive laws from $\sigma^1$ to $\sigma^2$.
  Furthermore, let $P^1$ and $P^2$ be the collated bifunctors and $\theta$ and $\zeta$ the collated oplax transformations. Finally, let $\overline{\beth}\colon \overline{\theta} \to \overline{\zeta}$ be a 2-morphism of distributive laws. We may define a collated modification $\beth \colon \theta \to \zeta$ by $\beth_{B,C} = \beth^B_C = \beth^C_B$.
\end{proposition}

\begin{proof}
  To see that $\beth$ so defined is a modification, consider the following series of string diagrams.
  \\ 
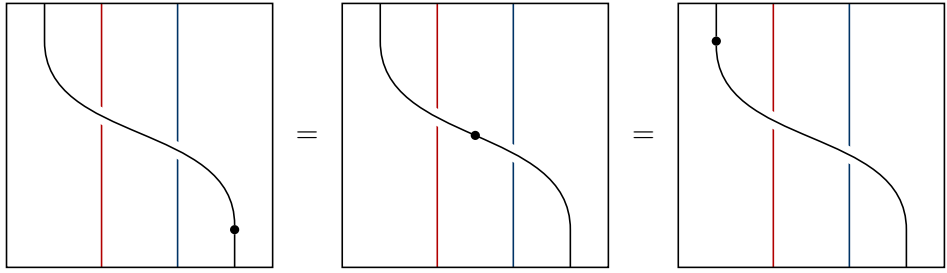
\vspace{-3pt plus 1pt}
  \begin{equation*}
   \begin{tikzpicture}[scale=1.0,baseline={([yshift=-0.5ex]current bounding box.center)}]
    \begin{scope}[on layer=over]
    \path coordinate (dl) ++(0,-3.5) coordinate (bl);
    \path (dl) ++(1.0,0) coordinate (dr) ++(0,-3.5) coordinate (br);
    \path (dl) ++(-0.75,0) coordinate (tl) ++(0,-0.5) coordinate (tl2);
    \path (br) ++(0.75,0) coordinate (brr) ++(0,0.5) coordinate[dot] (rr);
    \coordinate (mid) at ($(rr)!0.5!(tl2)$);
    \end{scope}
    \draw[functorM] (dr) -- (br);
    \draw[functorL] (dl) -- (bl);
    \draw[cross line] (tl) -- (tl2) to[out=270, in=90] (rr) -- (brr);
    \coordinate (cornerNW) at ($(tl) + (-0.5,0)$);
    \coordinate (cornerSE) at ($(brr) + (0.5,0)$);
    \draw (cornerNW) rectangle (cornerSE);
   \end{tikzpicture}
   \enspace=\enspace
   \begin{tikzpicture}[scale=1.0,baseline={([yshift=-0.5ex]current bounding box.center)}]
    \begin{scope}[on layer=over]
    \path coordinate (dl) ++(0,-3.5) coordinate (bl);
    \path (dl) ++(1.0,0) coordinate (dr) ++(0,-3.5) coordinate (br);
    \path (dl) ++(-0.75,0) coordinate (tl) ++(0,-0.5) coordinate (tl2);
    \path (br) ++(0.75,0) coordinate (brr) ++(0,0.5) coordinate (rr);
    \coordinate[dot] (mid) at ($(rr)!0.5!(tl2)$);
    \end{scope}
    \draw[functorM] (dr) -- (br);
    \draw[functorL] (dl) -- (bl);
    \draw[cross line] (tl) -- (tl2) to[out=270, in=90] (rr) -- (brr);
    \coordinate (cornerNW) at ($(tl) + (-0.5,0)$);
    \coordinate (cornerSE) at ($(brr) + (0.5,0)$);
    \draw (cornerNW) rectangle (cornerSE);
   \end{tikzpicture}
   \enspace=\enspace
   \begin{tikzpicture}[scale=1.0,baseline={([yshift=-0.5ex]current bounding box.center)}]
    \begin{scope}[on layer=over]
    \path coordinate (dl) ++(0,-3.5) coordinate (bl);
    \path (dl) ++(1.0,0) coordinate (dr) ++(0,-3.5) coordinate (br);
    \path (dl) ++(-0.75,0) coordinate (tl) ++(0,-0.5) coordinate[dot] (tl2);
    \path (br) ++(0.75,0) coordinate (brr) ++(0,0.5) coordinate (rr);
    \coordinate (mid) at ($(rr)!0.5!(tl2)$);
    \end{scope}
    \draw[functorM] (dr) -- (br);
    \draw[functorL] (dl) -- (bl);
    \draw[cross line] (tl) -- (tl2) to[out=270, in=90] (rr) -- (brr);
    \coordinate (cornerNW) at ($(tl) + (-0.5,0)$);
    \coordinate (cornerSE) at ($(brr) + (0.5,0)$);
    \draw (cornerNW) rectangle (cornerSE);
   \end{tikzpicture}
  \end{equation*}
  To move from the second diagram to the third, we make use of the fact that $\beth^B_C = \beth^C_B$.
\end{proof}

We can now upgrade the process of collation into a (strict) $2$-functor from $\Dist(\B,\C,\D)$ to $\Laxop(\B \times \C, \D)$.

\begin{theorem}\label{thm:collation_functor}
  Collation as in \cref{thm:lax_bifunctor,prp:collate_oplax_transformations,prp:collate_modifications} defines a strict $2$-functor $K\colon \Dist(\mathcal{B}, \mathcal{C}, \mathcal{D}) \to \Laxop(\mathcal{B}\times \mathcal{C}, \mathcal{D})$. 
\end{theorem}
\begin{proof}
   Inspecting the collation constructions and using that composition in $\Dist(\mathcal{B}, \mathcal{C}, \mathcal{D})$ is computed componentwise, it is easy to see that $K$ strictly preserves identities and composition of 1-morphisms, as well as vertical and horizontal composition of 2-morphisms.
\end{proof}

\begin{remark}
These ideas might also be generalised to $n$-ary functors. For instance, a `ternary functor theorem' would involve data which could be related to $\Laxop(\B,\Laxop(\C,\Laxop(\D,\mathcal{E})))$ and compared to the iterated distributive laws of \cite{cheng2011iterated}.
\end{remark}

\Cref{thm:equivalence_between_dist_and_iterated_functor_cat} allows us to compare the collation functor defined above and the `uncurrying' functor $J\colon \Laxop(\B,\Laxop(\C,\D)) \to \Laxop(\B \times \C,\D)$ defined implicitly in \cite{nikolic2018strictification}. We have the following commutative diagram.

\begin{center}
   \begin{tikzpicture}[node distance=2.5cm, auto]
    \node (A) {$\Dist(\B,\C,\D)$};
    \node (B) [above=1.5cm of A] {$\Laxop(\B,\Laxop(\C,\D))$};
    \node (C) [right=1.5cm of A] {$\Laxop(\B \times \C,\D)$};
    \draw[->] (A) to node[sloped] {$\sim$} (B);
    \draw[->] (A) to node [swap]{$K$} (C);
    \draw[->] (B) to node {$J$} (C);
   \end{tikzpicture}
\end{center}
In other words, collation of distributive laws is essentially uncurrying of lax bifunctors.

Since collation is not invertible in general, we have no general \emph{currying} functor. In the next section we consider how restricting to a certain full subcategory gives a setting in which such an inverse does exist. In particular, by restricting to pseudofunctors we obtain the exponential adjunction $\Hom(\B \times \C,\D) \cong \Hom(\B,\Hom(\C,\D))$.

\section{A partial converse to the 2-dimensional bifunctor theorem}\label{sec:converse}

Let us consider for which restricted class of lax functors we might reverse the collation procedure.

First observe that the canonical maps $\kappa^B\colon M_B \to P(B,-)$ and $\kappa^C\colon L_C \to P(-,C)$ from \cref{thm:lax_bifunctor} are invertible whenever all the supplied lax functors are unitary.
Hence, this gives a natural setting in which we can recover the input data from the collated lax bifunctor.
Also note that under these conditions the collated lax bifunctor is itself unitary. In fact, the collated lax bifunctor also satisfies the following further condition.

\begin{definition}
 We say a unitary lax bifunctor $(P, \gamma, \iota)$ is \emph{decomposable} if the compositors of the form $\gamma_{(\id,g),(f,\id)}\colon P(\id,g) P(f,\id) \to P(f,g)$ are invertible for all 1-morphisms $f$ and $g$.
 
 Note that in particular all pseudofunctors are decomposable.
\end{definition}

\begin{lemma}\label{prp:unitary_collation_decomposable}
 The collation of a distributive law between unitary lax bifunctors is a decomposable unitary lax bifunctor. Explicitly, the inverse to $\gamma_{(\id,g),(f,\id)}$ is given by $(M_{B_2}(g) \iota^{C_1}_{B_2}) \ast (\iota^{B_2}_{C_1} L_{C_1}(f))$.
\end{lemma}
\begin{proof}
  The collated lax bifunctor is certainly unitary.
  Now consider the following equality.
  \\ 
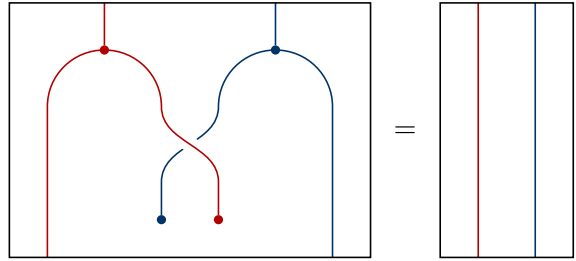
\vspace{-3pt plus 1pt}
  \begin{align*}
   \begin{tikzpicture}[scale=1.0,baseline={([yshift=-0.5ex]current bounding box.center)}]
    \begin{scope}[on layer=over]
    \path coordinate[dot, functorM] (mu)
    +(0,0.625) coordinate (d)
    +(-0.75,-0.75) coordinate (mbl)
    +(0.75,-0.75) coordinate (mbr);
    \path (mbr) ++(0,-2.0) coordinate (br);
    \path (mbl) ++(-0.75,-1.0) coordinate (bl2) ++(0,-0.5) coordinate[dot, functorM] (bl1) ++(0,-0.5) coordinate (bl);
    \path (mu) ++(-2.25,0) coordinate[dot, functorL] (muL)
    +(0,0.625) coordinate (dL)
    +(-0.75,-0.75) coordinate (mblL)
    +(0.75,-0.75) coordinate (mbrL);
    \path (mbrL) ++(0.75,-1.0) coordinate (brL2) ++(0,-0.5) coordinate[dot, functorL] (brL);
    \path (mblL) ++(0,-2.0) coordinate (blL);
    \end{scope}
    \draw[functorM] (bl1) -- (bl2) to[out=90, in=-90] (mbl) to[out=90, in=180] (mu.center) to[out=0, in=90] (mbr) -- (br)
                    (mu) -- (d);
    \draw[functorL, cross line] (blL) -- (mblL) to[out=90, in=180] (muL.center) to[out=0, in=90] (mbrL) to[out=-90, in=90] (brL2) -- (brL)
                    (muL) -- (dL);
    \coordinate (tl) at (dL -| blL);
    \coordinate (cornerNW) at ($(tl) + (-0.5,0)$);
    \coordinate (cornerSE) at ($(br) + (0.5,0)$);
    \draw (cornerNW) rectangle (cornerSE);
   \end{tikzpicture}
   &\enspace=\enspace
   \begin{tikzpicture}[scale=1.0,baseline={([yshift=-0.5ex]current bounding box.center)}]
    \begin{scope}[on layer=over]
    \path coordinate (mu)
    +(0,0.375) coordinate (d);
    \path (mu) (0,-3.0) coordinate (bl);
    \path (d) ++(-0.75,0) coordinate (dL) ++ (0,-1.125) coordinate (muL);
    \path (muL) ++(0,-2.25) coordinate (blL);
    \end{scope}
    \draw[functorM] (bl) -- (mu) -- (d);
    \draw[functorL, cross line] (blL) to[out=90, in=-90] (muL) -- (dL);
    \coordinate (tl) at (dL);
    \coordinate (bbr) at (bl -| mu);
    \coordinate (cornerNW) at ($(tl) + (-0.5,0)$);
    \coordinate (cornerSE) at ($(bbr) + (0.5,0)$);
    \draw (cornerNW) rectangle (cornerSE);
   \end{tikzpicture}
  \end{align*}
  
  Thus, for each $f\colon B_1 \to B_2$ and $g\colon C_1 \to C_2$ the compositor $\gamma_{(\id,g),(f,\id)}$ has a right inverse given by $(M_{B_2}(g) \iota^{C_1}_{B_2}) \ast (\iota^{B_2}_{C_1} L_{C_1}(f))$, or equivalently, $\kappa^{B_2}_g \ast \kappa^{C_1}_f$ recalling the `$\kappa$' maps from \cref{thm:lax_bifunctor}.
  But the unitors $\iota^{C_1}_{B_2}$ and $\iota^{B_2}_{C_1}$ are invertible by assumption and hence so is $\gamma_{(\id,g),(f,\id)}$.
\end{proof}

We are now able to show the following equivalence.

\begin{theorem}[Unitary lax bifunctor theorem]\label{thm:normal_lax_bifunctor}
  The collation 2-functor $K\colon \Dist(\B,\C,\D) \to \Laxop(\B \times \C,\D)$ restricts to an equivalence of $2$-categories $\overline{K}$ between the full sub-2-category of distributive laws between families of unitary lax functors and the full sub-2-category of decomposable unitary lax bifunctors.
  
  The `inverse' to $\overline{K}$ can be given explicitly by a 2-functor $T$ defined as follows.
  \begin{itemize}
      \item If $(P', \gamma', \iota')$ be a unitary decomposable lax bifunctor, then $T(P')$ is the distributive law given by $\sigma_{f,g} = \gamma^{\prime\, -1}_{(\id,g),(f,\id)} \gamma'_{(f,\id),(\id,g)}$ between families of lax functors with $L_C = P'(-,C)$ and $M_B = P'(B,-)$.
      \item If $\theta'\colon P'^1 \to P'^2$ is an oplax transformation between unitary decomposable lax bifunctors, then $T(\theta')$ is the morphism of distributive laws given by $(\theta'_{-,C})_C$ and $(\theta'_{B,-})_B$.
      \item If $\beth' \colon \theta' \to \zeta'$ is a modification between oplax transformations of unitary decomposable lax bifunctors, then $T(\beth')$ is given by the families $(\beth'_{-,C})_C$ and $(\beth'_{B,-})_B$.
  \end{itemize}
\end{theorem}

\begin{proof}
  The codomain of $\overline{K}$ restricts as appropriate by \cref{prp:unitary_collation_decomposable}.
  
  Let us demonstrate that the putative inverse $T$ is well-defined. We begin by showing that $T(P')$ satisfies the axioms of a distributive law.
  
  The form of condition \ref{eq:distcomp1} that we need to prove can be written as follows, where we use red and blue for $L$ and $M$ as before, but purple for $P'$ more generally.
  \\ 
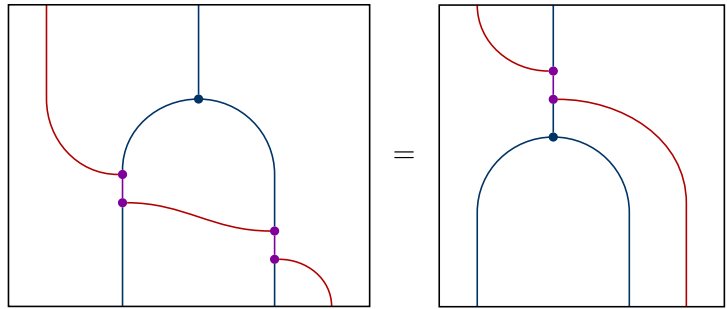
\vspace{-3pt plus 1pt}
  \begin{equation*}
   \begin{tikzpicture}[scale=1.0,baseline={([yshift=-0.5ex]current bounding box.center)}]
    \begin{scope}[on layer=over]
    \path coordinate[dot, functorM] (mu)
    +(0,1.25) coordinate (d)
    +(-1,-1) coordinate[dot, functorF2] (mbl)
    +(1,-1) coordinate (mbr);
    \path (d -| mbl) ++(-1.0,0) coordinate (tl) ++(0,-1.25) coordinate (tl2);
    \path (mbr) ++(0,-0.75) coordinate[dot, functorF2] (r) ++(0,-0.375) coordinate[dot, functorF2] (r2) ++(0,-0.625) coordinate (br) ++(0.75,0) coordinate (brr) ++(0,0.25) coordinate (brr2);
    \path (mbl) ++(0,-0.375) coordinate[dot, functorF2] (l) ++(0,-1.375) coordinate (bl);
    \end{scope}
    \draw[functorM] (bl) -- (l)
                    (mbl) to[out=90, in=180] (mu.center) to[out=0, in=90] (mbr) -- (r)
                    (r2) -- (br)
                    (mu) -- (d);
    \draw[functorL] (tl) -- (tl2) to[out=-90, in=180] (mbl)
                    (l) to[out=0, in=180] (r)
                    (r2) to[out=0, in=90] (brr);
    \draw[functorF2] (mbl) -- (l)
                     (r) -- (r2);
    \coordinate (cornerNW) at ($(tl) + (-0.5,0)$);
    \coordinate (cornerSE) at ($(brr) + (0.5,0)$);
    \draw (cornerNW) rectangle (cornerSE);
   \end{tikzpicture}
   \enspace =\enspace
   \begin{tikzpicture}[scale=1.0,baseline={([yshift=-0.5ex]current bounding box.center)}]
    \begin{scope}[on layer=over]
    \path coordinate[dot, functorM] (mu)
    +(-1,-1) coordinate (mbl)
    +(1,-1) coordinate (mbr)
    ++(0,0.5) coordinate[dot, functorF2] (t) ++(0,0.375) coordinate[dot, functorF2] (t2) ++(0,0.875) coordinate (d);
    \path (d -| mbl) ++(-0.0,0) coordinate (tl) ++(0,-0.25) coordinate (tl2);
    \path (mbr) ++(0,-1.25) coordinate (br) ++(0.75,0) coordinate (brr) ++(0,1.375) coordinate (brr2);
    \path (mbl) ++(0,-1.25) coordinate (bl);
    \end{scope}
    \draw[functorM] (bl) -- (mbl) to[out=90, in=180] (mu.center) to[out=0, in=90] (mbr) -- (br)
                    (mu) -- (t)
                    (t2) -- (d);
    \draw[functorL] (tl) to[out=270, in=180] (t2)
                    (t) to[out=0, in=90] (brr2) -- (brr);
    \draw[functorF2] (t) -- (t2);
    \coordinate (cornerNW) at ($(tl) + (-0.5,0)$);
    \coordinate (cornerSE) at ($(brr) + (0.5,0)$);
    \draw (cornerNW) rectangle (cornerSE);
   \end{tikzpicture}
  \end{equation*}
  
  We will show this holds by post-composing with the invertible 2-morphism $\gamma'_{(\id,g),(f,\id)}$.
  \\ \vspace{-3pt plus 1pt}
  \begingroup
  \allowdisplaybreaks
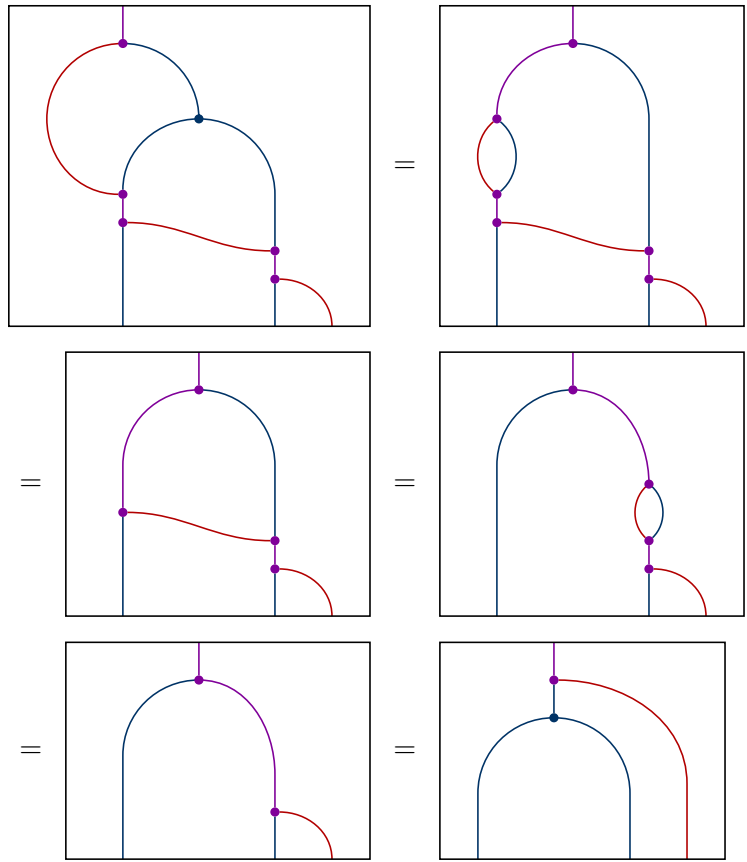
  \begin{align*}
   \begin{tikzpicture}[scale=1.0,baseline={([yshift=-0.5ex]current bounding box.center)}]
    \begin{scope}[on layer=over]
    \path coordinate[dot, functorM] (mu)
    +(0,0.0) coordinate (d)
    +(-1,-1) coordinate[dot, functorF2] (mbl)
    +(1,-1) coordinate (mbr);
    \path (d) ++(-2.0,0) coordinate (tl);
    \path (mbr) ++(0,-0.75) coordinate[dot, functorF2] (r) ++(0,-0.375) coordinate[dot, functorF2] (r2) ++(0,-0.625) coordinate (br) ++(0.75,0) coordinate (brr) ++(0,0.25) coordinate (brr2);
    \path (mbl) ++(0,-0.375) coordinate[dot, functorF2] (l) ++(0,-1.375) coordinate (bl);
    \path (d) ++(-1.0,1.0) coordinate[dot, functorF2] (muU) ++(0,0.5) coordinate (dU);
    \end{scope}
    \draw[functorM] (bl) -- (l)
                    (mbl) to[out=90, in=180] (mu.center) to[out=0, in=90] (mbr) -- (r)
                    (r2) -- (br)
                    (mu) -- (d) to[out=90, in=0] (muU.center);
    \draw[functorL] (tl) to[out=-90, in=180] (mbl)
                    (l) to[out=0, in=180] (r)
                    (r2) to[out=0, in=90] (brr)
                    (tl) to[out=90, in=180] (muU.center);
    \draw[functorF2] (mbl) -- (l)
                     (r) -- (r2)
                     (muU) -- (dU);
    \coordinate (ttl) at (tl |- dU);
    \coordinate (cornerNW) at ($(ttl) + (-0.5,0)$);
    \coordinate (cornerSE) at ($(brr) + (0.5,0)$);
    \draw (cornerNW) rectangle (cornerSE);
   \end{tikzpicture}
   &\enspace=\enspace
   \begin{tikzpicture}[scale=1.0,baseline={([yshift=-0.5ex]current bounding box.center)}]
    \begin{scope}[on layer=over]
    \path coordinate (mu)
    +(0,0.0) coordinate (d)
    +(-1,-1) coordinate[dot, functorF2] (mbl)
    ++(1,-1) coordinate (mbr) ++(0,1.0) coordinate (mbrU);
    \path (mu) ++(-1.0,0) coordinate[dot, functorF2] (muNew);
    \path (mbr) ++(0,-0.75) coordinate[dot, functorF2] (r) ++(0,-0.375) coordinate[dot, functorF2] (r2) ++(0,-0.625) coordinate (br) ++(0.75,0) coordinate (brr) ++(0,0.25) coordinate (brr2);
    \path (mbl) ++(0,-0.375) coordinate[dot, functorF2] (l) ++(0,-1.375) coordinate (bl);
    \path (d) ++(0,1.0) coordinate[dot, functorF2] (muU) ++(0,0.5) coordinate (dU);
    \end{scope}
    \draw[functorM] (bl) -- (l)
                    (r2) -- (br)
                    (r) -- (mbr) -- (mbrU) to[out=90, in=0] (muU.center)
                    (muNew) to[bend left=50] (mbl);
    \draw[functorL] (l) to[out=0, in=180] (r)
                    (r2) to[out=0, in=90] (brr)
                    (muNew) to[bend right=50] (mbl);
    \draw[functorF2] (mbl) -- (l)
                     (r) -- (r2)
                     (muU) -- (dU)
                     (muU.center) to[out=180, in=90] (muNew);
    \coordinate (ttl) at (mbl |- dU);
    \coordinate (cornerNW) at ($(ttl) + (-0.75,0)$);
    \coordinate (cornerSE) at ($(brr) + (0.5,0)$);
    \draw (cornerNW) rectangle (cornerSE);
   \end{tikzpicture}
   \\[5pt] =\enspace 
   \begin{tikzpicture}[scale=1.0,baseline={([yshift=-0.5ex]current bounding box.center)}]
    \begin{scope}[on layer=over]
    \path coordinate (mu)
    +(0,0.0) coordinate (d)
    +(-1,-0.25) coordinate (mbl)
    +(1,0) coordinate (mbrU)
    +(1,-0.25) coordinate (mbr);
    \path (mu) ++(-1.0,0) coordinate (muNew);
    \path (mbr) ++(0,-0.75) coordinate[dot, functorF2] (r) ++(0,-0.375) coordinate[dot, functorF2] (r2) ++(0,-0.625) coordinate (br) ++(0.75,0) coordinate (brr) ++(0,0.25) coordinate (brr2);
    \path (mbl) ++(0,-0.375) coordinate[dot, functorF2] (l) ++(0,-1.375) coordinate (bl);
    \path (d) ++(0,1.0) coordinate[dot, functorF2] (muU) ++(0,0.5) coordinate (dU);
    \end{scope}
    \draw[functorM] (bl) -- (l)
                    (r2) -- (br)
                    (r) -- (mbr) -- (mbrU) to[out=90, in=0] (muU.center);
    \draw[functorL] (l) to[out=0, in=180] (r)
                    (r2) to[out=0, in=90] (brr);
    \draw[functorF2] (mbl) -- (l)
                     (r) -- (r2)
                     (muU) -- (dU)
                     (muU.center) to[out=180, in=90] (muNew) -- (mbl);
    \coordinate (ttl) at (mbl |- dU);
    \coordinate (cornerNW) at ($(ttl) + (-0.75,0)$);
    \coordinate (cornerSE) at ($(brr) + (0.5,0)$);
    \draw (cornerNW) rectangle (cornerSE);
   \end{tikzpicture}
   &\enspace=\enspace
   \begin{tikzpicture}[scale=1.0,baseline={([yshift=-0.5ex]current bounding box.center)}]
    \begin{scope}[on layer=over]
    \path coordinate (mu)
    +(0,0.0) coordinate (d)
    +(-1,-0.25) coordinate (mbl)
    +(1,-0.25) coordinate[dot, functorF2] (mbr);
    \path (mu) ++(-1.0,0) coordinate (muNew);
    \path (mbr) ++(0,-0.75) coordinate[dot, functorF2] (r) ++(0,-0.375) coordinate[dot, functorF2] (r2) ++(0,-0.625) coordinate (br) ++(0.75,0) coordinate (brr) ++(0,0.25) coordinate (brr2);
    \path (mbl) ++(0,-0.375) coordinate (l) ++(0,-1.375) coordinate (bl);
    \path (d) ++(0,1.0) coordinate[dot, functorF2] (muU) ++(0,0.5) coordinate (dU);
    \end{scope}
    \draw[functorM] (muU.center) to[out=180, in=90] (muNew) -- (l) -- (bl)
                    (r2) -- (br)
                    (r) to[bend right=50] (mbr.center);
    \draw[functorL] (mbr.center) to[bend right=50] (r)
                    (r2) to[out=0, in=90] (brr);
    \draw[functorF2] (r) -- (r2)
                     (muU) -- (dU)
                     (mbr.center) to[out=90, in=0] (muU.center);
    \coordinate (ttl) at (mbl |- dU);
    \coordinate (cornerNW) at ($(ttl) + (-0.75,0)$);
    \coordinate (cornerSE) at ($(brr) + (0.5,0)$);
    \draw (cornerNW) rectangle (cornerSE);
   \end{tikzpicture}
   \\[5pt] =\enspace
   \begin{tikzpicture}[scale=1.0,baseline={([yshift=-0.5ex]current bounding box.center)}]
    \begin{scope}[on layer=over]
    \path coordinate (mu)
    +(0,0.0) coordinate (d)
    +(-1,-0.25) coordinate (mbl)
    +(1,-0.25) coordinate (mbr);
    \path (mu) ++(-1.0,0) coordinate (muNew);
    \path (mbr) ++(0,-0.5) coordinate[dot, functorF2] (r2) ++(0,-0.625) coordinate (br) ++(0.75,0) coordinate (brr) ++(0,0.25) coordinate (brr2);
    \path (mbl) ++(0,-1.125) coordinate (bl);
    \path (d) ++(0,1.0) coordinate[dot, functorF2] (muU) ++(0,0.5) coordinate (dU);
    \end{scope}
    \draw[functorM] (muU.center) to[out=180, in=90] (muNew) -- (bl)
                    (r2) -- (br);
    \draw[functorL] (r2) to[out=0, in=90] (brr);
    \draw[functorF2] (r2) -- (mbr)
                     (muU) -- (dU)
                     (mbr.center) to[out=90, in=0] (muU.center);
    \coordinate (ttl) at (mbl |- dU);
    \coordinate (cornerNW) at ($(ttl) + (-0.75,0)$);
    \coordinate (cornerSE) at ($(brr) + (0.5,0)$);
    \draw (cornerNW) rectangle (cornerSE);
   \end{tikzpicture}
   &\enspace=\enspace
   \begin{tikzpicture}[scale=1.0,baseline={([yshift=-0.5ex]current bounding box.center)}]
    \begin{scope}[on layer=over]
    \path coordinate[dot, functorM] (mu)
    +(-1,-1) coordinate (mbl)
    +(1,-1) coordinate (mbr)
    ++(0,0.5) coordinate[dot, functorF2] (t) ++(0,0.5) coordinate (d);
    \path (d -| mbl) ++(-0.0,0) coordinate (tl);
    \path (mbr) ++(0,-0.875) coordinate (br) ++(0.75,0) coordinate (brr) ++(0,1.0) coordinate (brr2);
    \path (mbl) ++(0,-0.875) coordinate (bl);
    \end{scope}
    \draw[functorM] (bl) -- (mbl) to[out=90, in=180] (mu.center) to[out=0, in=90] (mbr) -- (br)
                    (mu) -- (t);
    \draw[functorL] (t) to[out=0, in=90] (brr2) -- (brr);
    \draw[functorF2] (t) -- (d);
    \coordinate (cornerNW) at ($(tl) + (-0.5,0)$);
    \coordinate (cornerSE) at ($(brr) + (0.5,0)$);
    \draw (cornerNW) rectangle (cornerSE);
   \end{tikzpicture}
  \end{align*}
  \endgroup
  Here we make repeated use of associativity of the compositors and the invertibility of $\gamma'_{(\id,g),(f,\id)}$. It is helpful to remember that the invertible compositors are those with a red wire as their left input and a blue wire as their right input.
  This last string diagram is then the right-hand diagram above composed with $\gamma'_{(\id,g),(f,\id)}$ (before cancelling with its inverse) as required.
  
  To prove condition \ref{eq:distunit1} we again post-compose with $\gamma'_{(\id,\id),(f,\id)}$. Then the left-hand side gives the following.
  \\ 
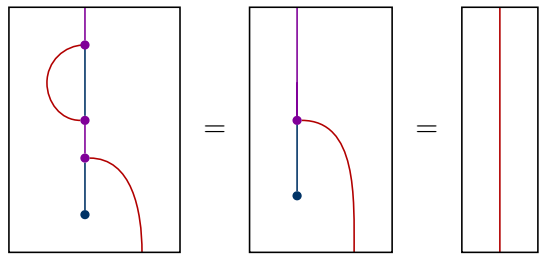
\vspace{-3pt plus 1pt}
  \begin{equation*}
   \begin{tikzpicture}[scale=1.0,baseline={([yshift=-0.5ex]current bounding box.center)}]
    \begin{scope}[on layer=over]
    \path coordinate[dot, functorF2] (mu) ++(0,0.5) coordinate (d);
    \path (mu) ++(0,-1.0) coordinate[dot, functorF2] (r) ++(0,-0.5) coordinate[dot, functorF2] (r2) ++(0,-0.75) coordinate[dot, functorM] (br) ++(0.75,-0.5) coordinate (brr);
    \path (mu) ++(-0.5,-0.5) coordinate (l);
    \end{scope}
    \draw[functorM] (br) -- (r2)
                    (r) -- (mu.center);
    \draw[functorL] (mu.center) to[out=180, in=90] (l) to[out=-90, in=180] (r)
                    (r2) to[out=0, in=90] (brr);
    \draw[functorF2] (r) -- (r2)
                     (mu) -- (d);
    \coordinate (ttl) at (l |- d);
    \coordinate (cornerNW) at ($(ttl) + (-0.5,0)$);
    \coordinate (cornerSE) at ($(brr) + (0.5,0)$);
    \draw (cornerNW) rectangle (cornerSE);
   \end{tikzpicture}
   \enspace=\enspace
   \begin{tikzpicture}[scale=1.0,baseline={([yshift=-0.5ex]current bounding box.center)}]
    \begin{scope}[on layer=over]
    \path coordinate (mu) ++(0,0.5) coordinate (d);
    \path (mu) ++(0,-0.5) coordinate (r) ++(0,-0.5) coordinate[dot, functorF2] (r2) ++(0,-1.0) coordinate[dot, functorM] (br) ++(0.75,-0.75) coordinate (brr);
    \end{scope}
    \draw[functorM] (br) -- (r2);
    \draw[functorL] (r2) to[out=0, in=90] (brr);
    \draw[functorF2] (r) -- (r2) -- (d);
    \coordinate (ttl) at (d);
    \coordinate (cornerNW) at ($(ttl) + (-0.625,0)$);
    \coordinate (cornerSE) at ($(brr) + (0.5,0)$);
    \draw (cornerNW) rectangle (cornerSE);
   \end{tikzpicture}
   \enspace=\enspace
   \begin{tikzpicture}[scale=1.0,baseline={([yshift=-0.5ex]current bounding box.center)}]
    \begin{scope}[on layer=over]
    \path coordinate (mu) ++(0,0.5) coordinate (d);
    \path (mu) ++(0,-0.5) coordinate (r) ++(0,-0.5) coordinate (r2) ++(0,-1.0) coordinate (br) ++(0,-0.75) coordinate (brr);
    \end{scope}
    \draw[functorL] (brr) -- (d);
    \coordinate (ttl) at (d);
    \coordinate (cornerNW) at ($(ttl) + (-0.5,0)$);
    \coordinate (cornerSE) at ($(brr) + (0.5,0)$);
    \draw (cornerNW) rectangle (cornerSE);
   \end{tikzpicture}
  \end{equation*}
  On the other hand, the right-hand side yields the same result.
  \\ 
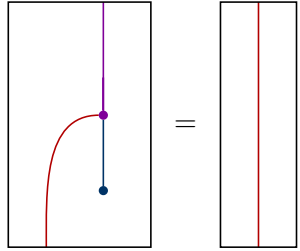
\vspace{-3pt plus 1pt}
  \begin{equation*}
   \begin{tikzpicture}[scale=1.0,baseline={([yshift=-0.5ex]current bounding box.center)}]
    \begin{scope}[on layer=over]
    \path coordinate (mu) ++(0,0.5) coordinate (d);
    \path (mu) ++(0,-0.5) coordinate (r) ++(0,-0.5) coordinate[dot, functorF2] (r2) ++(0,-1.0) coordinate[dot, functorM] (br) ++(-0.75,-0.75) coordinate (brr);
    \end{scope}
    \draw[functorM] (br) -- (r2);
    \draw[functorL] (r2) to[out=180, in=90] (brr);
    \draw[functorF2] (r) -- (r2) -- (d);
    \coordinate (ttl) at (d);
    \coordinate (cornerNW) at ($(ttl) + (0.625,0)$);
    \coordinate (cornerSE) at ($(brr) + (-0.5,0)$);
    \draw (cornerNW) rectangle (cornerSE);
   \end{tikzpicture}
   \enspace=\enspace
   \begin{tikzpicture}[scale=1.0,baseline={([yshift=-0.5ex]current bounding box.center)}]
    \begin{scope}[on layer=over]
    \path coordinate (mu) ++(0,0.5) coordinate (d);
    \path (mu) ++(0,-0.5) coordinate (r) ++(0,-0.5) coordinate (r2) ++(0,-1.0) coordinate (br) ++(0,-0.75) coordinate (brr);;
    \end{scope}
    \draw[functorL] (brr) -- (d);
    \coordinate (ttl) at (d);
    \coordinate (cornerNW) at ($(ttl) + (-0.5,0)$);
    \coordinate (cornerSE) at ($(brr) + (0.5,0)$);
    \draw (cornerNW) rectangle (cornerSE);
   \end{tikzpicture}
  \end{equation*}
  
  The condition \ref{eq:distnat1} is shown in a similar way.
  After composing with $\gamma'_{(\id,g),(f,\id)}$ we obtain the following sequence of string diagrams.
  
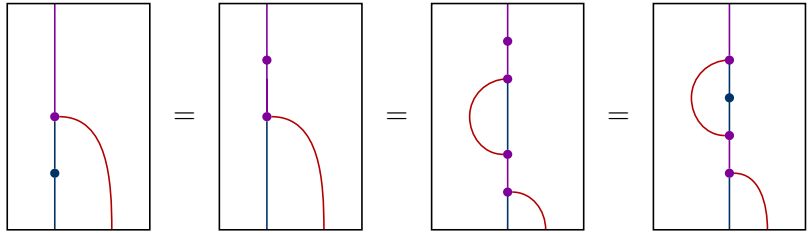
\begin{equation*}
   \begin{tikzpicture}[scale=1.0,baseline={([yshift=-0.5ex]current bounding box.center)}]
    \begin{scope}[on layer=over]
    \path coordinate (mu) ++(0,0.5) coordinate (d);
    \path (mu) ++(0,-0.5) coordinate (r) ++(0,-0.5) coordinate[dot, functorF2] (r2) ++(0,-0.75) coordinate[dot, functorM] (br) ++(0.75,-0.75) coordinate (brr);
    \end{scope}
    \draw[functorM] (br |- brr) -- (br) -- (r2);
    \draw[functorL] (r2) to[out=0, in=90] (brr);
    \draw[functorF2] (r2) -- (d);
    \coordinate (ttl) at (d);
    \coordinate (cornerNW) at ($(ttl) + (-0.625,0)$);
    \coordinate (cornerSE) at ($(brr) + (0.5,0)$);
    \draw (cornerNW) rectangle (cornerSE);
   \end{tikzpicture}
   \enspace=\enspace
   \begin{tikzpicture}[scale=1.0,baseline={([yshift=-0.5ex]current bounding box.center)}]
    \begin{scope}[on layer=over]
    \path coordinate[dot, functorF2] (mu) ++(0,0.75) coordinate (d);
    \path (mu) ++(0,-0.25) coordinate (r) ++(0,-0.5) coordinate[dot, functorF2] (r2) ++(0,-0.75) coordinate (br) ++(0.75,-0.75) coordinate (brr);
    \end{scope}
    \draw[functorM] (br |- brr) -- (br) -- (r2);
    \draw[functorL] (r2) to[out=0, in=90] (brr);
    \draw[functorF2] (r) -- (r2) -- (d);
    \coordinate (ttl) at (d);
    \coordinate (cornerNW) at ($(ttl) + (-0.625,0)$);
    \coordinate (cornerSE) at ($(brr) + (0.5,0)$);
    \draw (cornerNW) rectangle (cornerSE);
   \end{tikzpicture}
   \enspace=\enspace
   \begin{tikzpicture}[scale=1.0,baseline={([yshift=-0.5ex]current bounding box.center)}]
    \begin{scope}[on layer=over]
    \path coordinate[dot, functorF2] (mu) ++(0,0.5) coordinate[dot, functorF2] (d2) ++(0,0.5) coordinate (d);
    \path (mu) ++(0,-1.0) coordinate[dot, functorF2] (r) ++(0,-0.5) coordinate[dot, functorF2] (r2) ++(0,-0.0) coordinate (br) ++(0.5,-0.5) coordinate (brr);
    \path (mu) ++(-0.5,-0.5) coordinate (l);
    \end{scope}
    \draw[functorM] (br |- brr) -- (br) -- (r2)
                    (r) -- (mu.center);
    \draw[functorL] (mu.center) to[out=180, in=90] (l) to[out=-90, in=180] (r)
                    (r2) to[out=0, in=90] (brr);
    \draw[functorF2] (r) -- (r2)
                     (mu) -- (d);
    \coordinate (ttl) at (l |- d);
    \coordinate (cornerNW) at ($(ttl) + (-0.5,0)$);
    \coordinate (cornerSE) at ($(brr) + (0.5,0)$);
    \draw (cornerNW) rectangle (cornerSE);
   \end{tikzpicture}
   \enspace=\enspace
   \begin{tikzpicture}[scale=1.0,baseline={([yshift=-0.5ex]current bounding box.center)}]
    \begin{scope}[on layer=over]
    \path coordinate[dot, functorF2] (mu) ++(0,0.75) coordinate (d);
    \path (mu) ++(0,-0.5) coordinate[dot, functorM] (mid) ++(0,-0.5) coordinate[dot, functorF2] (r) ++(0,-0.5) coordinate[dot, functorF2] (r2) ++(0,-0.25) coordinate (br) ++(0.5,-0.5) coordinate (brr);
    \path (mu) ++(-0.5,-0.5) coordinate (l);
    \end{scope}
    \draw[functorM] (br |- brr) -- (br) -- (r2)
                    (r) -- (mu.center);
    \draw[functorL] (mu.center) to[out=180, in=90] (l) to[out=-90, in=180] (r)
                    (r2) to[out=0, in=90] (brr);
    \draw[functorF2] (r) -- (r2)
                     (mu) -- (d);
    \coordinate (ttl) at (l |- d);
    \coordinate (cornerNW) at ($(ttl) + (-0.5,0)$);
    \coordinate (cornerSE) at ($(brr) + (0.5,0)$);
    \draw (cornerNW) rectangle (cornerSE);
   \end{tikzpicture}
  \end{equation*}
  
  The other conditions are dual. 
  
  Next we show that $T(\theta')$ is a 1-morphism of distributive laws.
  Clearly, $\theta'^B_C = \theta'^C_B$. To see that they satisfy the Yang--Baxter equation, we post-compose the necessary diagrams with $\gamma'^{-1}_{(\id,g),(f,\id)} \theta'_{B,C}$ and observe the following equalities.
  \\ 
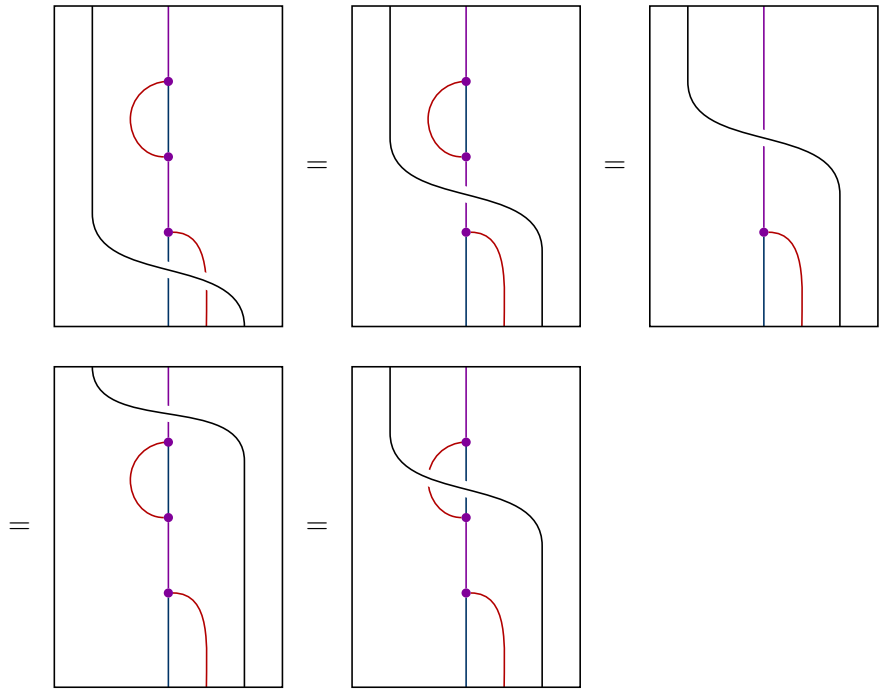
\vspace{-3pt plus 1pt}
  \begin{align*}
  \begin{tikzpicture}[scale=1.0,baseline={([yshift=-0.5ex]current bounding box.center)}]
    \begin{scope}[on layer=over]
    \path coordinate[dot, functorF2] (mu) ++(0,1.0) coordinate (d);
    \path (mu) ++(0,-0.5) coordinate (mid) ++(0,-0.5) coordinate[dot, functorF2] (m) ++(0,-1.0) coordinate[dot, functorF2] (m2) ++(0,-0.75) coordinate (b) ++(0.5,-0.5) coordinate (br) ++(0.5,0) coordinate (brr) ++(0,0.0) coordinate (brr2);
    \path (mu) ++(-0.5,-0.5) coordinate (l);
    \path (d) ++(-1.0,0) coordinate (tl) ++(0,-2.75) coordinate (tl2);
    \end{scope}
    \draw[functorM] (b |- br) -- (b) -- (m2)
                    (m) -- (mu.center);
    \draw[functorL] (mu.center) to[out=180, in=90] (l) to[out=-90, in=180] (m)
                    (m2) to[out=0, in=90] (br);
    \draw[functorF2] (m) -- (m2)
                     (mu) -- (d);
    \draw[cross line] (tl) -- (tl2) to[out=270, in=90] (brr2) -- (brr);
    \coordinate (cornerNW) at ($(tl) + (-0.5,0)$);
    \coordinate (cornerSE) at ($(brr) + (0.5,0)$);
    \draw (cornerNW) rectangle (cornerSE);
  \end{tikzpicture}
  &\enspace=\enspace
  \begin{tikzpicture}[scale=1.0,baseline={([yshift=-0.5ex]current bounding box.center)}]
    \begin{scope}[on layer=over]
    \path coordinate[dot, functorF2] (mu) ++(0,1.0) coordinate (d);
    \path (mu) ++(0,-0.5) coordinate (mid) ++(0,-0.5) coordinate[dot, functorF2] (m) ++(0,-1.0) coordinate[dot, functorF2] (m2) ++(0,-0.75) coordinate (b) ++(0.5,-0.5) coordinate (br) ++(0.5,0) coordinate (brr) ++(0,1.0) coordinate (brr2);
    \path (mu) ++(-0.5,-0.5) coordinate (l);
    \path (d) ++(-1.0,0) coordinate (tl) ++(0,-1.75) coordinate (tl2);
    \end{scope}
    \draw[functorM] (b |- br) -- (b) -- (m2)
                    (m) -- (mu.center);
    \draw[functorL] (mu.center) to[out=180, in=90] (l) to[out=-90, in=180] (m)
                    (m2) to[out=0, in=90] (br);
    \draw[functorF2] (m) -- (m2)
                     (mu) -- (d);
    \draw[cross line] (tl) -- (tl2) to[out=270, in=90] (brr2) -- (brr);
    \coordinate (cornerNW) at ($(tl) + (-0.5,0)$);
    \coordinate (cornerSE) at ($(brr) + (0.5,0)$);
    \draw (cornerNW) rectangle (cornerSE);
  \end{tikzpicture}
  \enspace=\enspace
  \begin{tikzpicture}[scale=1.0,baseline={([yshift=-0.5ex]current bounding box.center)}]
    \begin{scope}[on layer=over]
    \path coordinate (mu) ++(0,1.0) coordinate (d);
    \path (mu) ++(0,-0.5) coordinate (mid) ++(0,-0.5) coordinate (m) ++(0,-1.0) coordinate[dot, functorF2] (m2) ++(0,-0.75) coordinate (b) ++(0.5,-0.5) coordinate (br) ++(0.5,0) coordinate (brr) ++(0,1.75) coordinate (brr2);
    \path (mu) ++(-0.5,-0.5) coordinate (l);
    \path (d) ++(-1.0,0) coordinate (tl) ++(0,-1.0) coordinate (tl2);
    \end{scope}
    \draw[functorM] (b |- br) -- (b) -- (m2);
    \draw[functorL] (m2) to[out=0, in=90] (br);
    \draw[functorF2] (m2) -- (m) -- (mu) -- (d);
    \draw[cross line] (tl) -- (tl2) to[out=270, in=90] (brr2) -- (brr);
    \coordinate (cornerNW) at ($(tl) + (-0.5,0)$);
    \coordinate (cornerSE) at ($(brr) + (0.5,0)$);
    \draw (cornerNW) rectangle (cornerSE);
  \end{tikzpicture}
  \\[5pt]=\enspace
  \begin{tikzpicture}[scale=1.0,baseline={([yshift=-0.5ex]current bounding box.center)}]
    \begin{scope}[on layer=over]
    \path coordinate[dot, functorF2] (mu) ++(0,1.0) coordinate (d);
    \path (mu) ++(0,-0.5) coordinate (mid) ++(0,-0.5) coordinate[dot, functorF2] (m) ++(0,-1.0) coordinate[dot, functorF2] (m2) ++(0,-0.75) coordinate (b) ++(0.5,-0.5) coordinate (br) ++(0.5,0) coordinate (brr) ++(0,3.0) coordinate (brr2);
    \path (mu) ++(-0.5,-0.5) coordinate (l);
    \path (d) ++(-1.0,0) coordinate (tl) ++(0,-0.0) coordinate (tl2);
    \end{scope}
    \draw[functorM] (b |- br) -- (b) -- (m2)
                    (m) -- (mu.center);
    \draw[functorL] (mu.center) to[out=180, in=90] (l) to[out=-90, in=180] (m)
                    (m2) to[out=0, in=90] (br);
    \draw[functorF2] (m) -- (m2)
                     (mu) -- (d);
    \draw[cross line] (tl) -- (tl2) to[out=270, in=90] (brr2) -- (brr);
    \coordinate (cornerNW) at ($(tl) + (-0.5,0)$);
    \coordinate (cornerSE) at ($(brr) + (0.5,0)$);
    \draw (cornerNW) rectangle (cornerSE);
  \end{tikzpicture}
  &\enspace=\enspace
  \begin{tikzpicture}[scale=1.0,baseline={([yshift=-0.5ex]current bounding box.center)}]
    \begin{scope}[on layer=over]
    \path coordinate[dot, functorF2] (mu) ++(0,1.0) coordinate (d);
    \path (mu) ++(0,-0.5) coordinate (mid) ++(0,-0.5) coordinate[dot, functorF2] (m) ++(0,-1.0) coordinate[dot, functorF2] (m2) ++(0,-0.75) coordinate (b) ++(0.5,-0.5) coordinate (br) ++(0.5,0) coordinate (brr) ++(0,1.875) coordinate (brr2);
    \path (mu) ++(-0.5,-0.5) coordinate (l);
    \path (d) ++(-1.0,0) coordinate (tl) ++(0,-0.875) coordinate (tl2);
    \end{scope}
    \draw[functorM] (b |- br) -- (b) -- (m2)
                    (m) -- (mu.center);
    \draw[functorL] (mu.center) to[out=180, in=90] (l) to[out=-90, in=180] (m)
                    (m2) to[out=0, in=90] (br);
    \draw[functorF2] (m) -- (m2)
                     (mu) -- (d);
    \draw[cross line] (tl) -- (tl2) to[out=270, in=90] (brr2) -- (brr);
    \coordinate (cornerNW) at ($(tl) + (-0.5,0)$);
    \coordinate (cornerSE) at ($(brr) + (0.5,0)$);
    \draw (cornerNW) rectangle (cornerSE);
  \end{tikzpicture}
  \end{align*}
  This yields the Yang--Baxter equation since $\gamma\inv_{(\id,g),(f,\id)}$ is invertible.
  
  It is immediate that $T(\beth')$ defines families of modifications that agree in the sense necessary to give a 2-morphism of distributive laws.
  Moreover, it is clear that $T$ is 2-functorial since composition in $\Dist(\B,\C,\D)$ is computed componentwise.
  
  Now all that remains is to exhibit 2-natural isomorphisms $\lambda\colon \overline{K}T \cong \Id$ and $\kappa\colon \Id \cong T\overline{K}$.
  Let us start by defining the components of $\lambda$.
  For a unitary decomposable lax bifunctor $P'\colon \B \times \C \to \D$ we require an oplax transformation $\lambda_{P'}\colon \overline{K}T(P') \to P'$.
  Explicitly, we have $\overline{K}T(P')(B,C) = P'(B,C)$ and $\overline{K}T(P')(f,g) = P'(\id,g) P'(f,\id)$. Define $\lambda_{P'}$ to have identities as its 1-morphisms and $\gamma'_{(\id,g),(f,\id)}$ as its 2-morphisms.
  
  We prove $\lambda_{P'}$ is indeed an oplax transformation.
  The unit condition for an oplax transformation holds by the unit law for the lax functor $P'$ as in the diagram below. (Here the invisible identity wire intersects the left-hand diagram at $\gamma'$.)
  \\ 
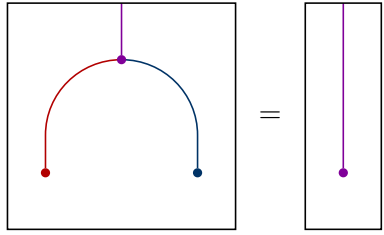
\vspace{-3pt plus 1pt}
  \begin{equation*} 
   \begin{tikzpicture}[scale=1.0,baseline={([yshift=-0.5ex]current bounding box.center)}]
    \begin{scope}[on layer=over]
    \path coordinate[dot, functorF2] (muU)
    +(0,0.75) coordinate (dU)
    +(-1,-1) coordinate (mblU)
    +(1,-1) coordinate (mbrU);
    \path (mblU) ++(0,-0.5) coordinate[dot, functorL] (mu);
    \path (mbrU) ++(0,-0.5) coordinate[dot, functorM] (mur) ++ (0,-0.75) coordinate (brr);
    \end{scope}
    \draw[functorL] (mu) -- (mblU) to[out=90, in=180] (muU.center);
    \draw[functorM] (muU.center) to[out=0, in=90] (mbrU) -- (mur);
    \draw[functorF2] (muU) -- (dU);
    \coordinate (tl) at (mblU |- dU);
    \coordinate (cornerNW) at ($(tl) + (-0.5,0)$);
    \coordinate (cornerSE) at ($(brr) + (0.5,0)$);
    \draw (cornerNW) rectangle (cornerSE);
   \end{tikzpicture}
   \enspace=\enspace
   \begin{tikzpicture}[scale=1.0,baseline={([yshift=-0.5ex]current bounding box.center)}]
    \begin{scope}[on layer=over]
    \path coordinate (muU)
    +(0,1.0) coordinate (dU)
    ++(0,-1.25) coordinate[dot, functorF2] (b) ++(0,-0.75) coordinate (bb);
    \end{scope}
    \draw[functorF2] (dU) -- (b);
    \coordinate (cornerNW) at ($(dU) + (-0.5,0)$);
    \coordinate (cornerSE) at ($(bb) + (0.5,0)$);
    \draw (cornerNW) rectangle (cornerSE);
   \end{tikzpicture}
  \end{equation*}
  
  The compositor condition is given by the following sequence of string diagrams. (Again the topmost purple dot in the first diagram is the `intersection with the identity wire', while the rest of the diagram is the compositor for $\overline{K}T(P')$.)
  \\ \vspace{-3pt plus 1pt}
  \begingroup
  \allowdisplaybreaks
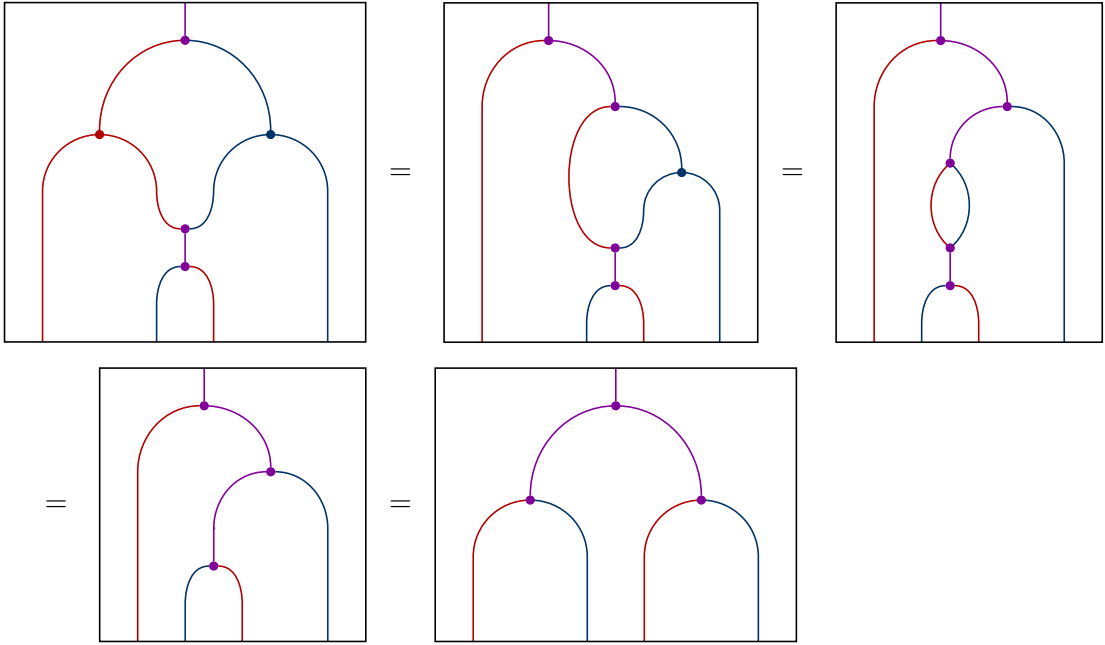
  \begin{align*}
   \begin{tikzpicture}[scale=1.0,baseline={([yshift=-0.5ex]current bounding box.center)}]
    \begin{scope}[on layer=over]
    \path coordinate[dot, functorM] (mu)
    +(-0.75,-0.75) coordinate (mbl)
    +(0.75,-0.75) coordinate (mbr);
    \path (mbr) ++(0,-2.0) coordinate (br);
    \path (mbl) ++(-0.75,-1.5) coordinate (bl2) ++(0,-0.5) coordinate (bl);
    \path (mu) ++(-2.25,0) coordinate[dot, functorL] (muL)
    +(-0.75,-0.75) coordinate (mblL)
    +(0.75,-0.75) coordinate (mbrL);
    \path (mbrL) ++(0.75,-1.5) coordinate (brL2) ++(0,-0.5) coordinate (brL);
    \path (mblL) ++(0,-2.0) coordinate (blL);
    \path ($(mu)!0.5!(muL)$) ++(0,1.25) coordinate[dot, functorF2] (muUL)
    +(0,0.5) coordinate (dUL)
    +(-1.5,-1) coordinate (mblUL)
    +(1.5,-1) coordinate (mbrUL);
    \path ($(mbr)!0.5!(mblL)$) ++(0,-0.5) coordinate[dot, functorF2] (mbMid) ++ (0,-0.5) coordinate[dot, functorF2] (bMid);
    \end{scope}
    \draw[functorM] (bl) -- (bl2) to[out=90, in=180] (bMid)
                    (mbMid) to[out=0, in=-90] (mbl) to[out=90, in=180] (mu.center) to[out=0, in=90] (mbr) -- (br)
                    (muUL.center) to[out=0, in=90] (mu);
    \draw[functorL] (blL) -- (mblL) to[out=90, in=180] (muL.center) to[out=0, in=90] (mbrL) to[out=-90, in=180] (mbMid)
                    (bMid) to[out=0, in=90] (brL2) -- (brL)
                    (muL) to[out=90, in=180] (muUL.center);
    \draw[functorF2] (muUL) -- (dUL)
                     (mbMid) -- (bMid);
    \coordinate (tl) at (mblL |- dUL);
    \coordinate (cornerNW) at ($(tl) + (-0.5,0)$);
    \coordinate (cornerSE) at ($(br) + (0.5,0)$);
    \draw (cornerNW) rectangle (cornerSE);
   \end{tikzpicture}
   &\enspace=\enspace
   \begin{tikzpicture}[scale=1.0,baseline={([yshift=-0.5ex]current bounding box.center)}]
    \begin{scope}[on layer=over] 
    \path coordinate[dot, functorM] (mu)
    +(-0.5,-0.5) coordinate (mbl)
    +(0.5,-0.5) coordinate (mbr);
    \path (mbr) ++(0,-1.75) coordinate (br);
    \path (mbl) ++(-0.75,-1.5) coordinate (bl2) ++(0,-0.25) coordinate (bl);
    \path (bl2) ++(0.75,0) coordinate (brl2) ++(0,-0.25) coordinate (brl);
    \path (mu) ++(-0.875,0.875) coordinate[dot, functorF2] (muNew);
    \path (muNew) ++(-0.875,0.875) coordinate[dot, functorF2] (muUL)
    +(0,0.5) coordinate (dUL)
    +(-1.5,-1) coordinate (mblUL)
    +(1.5,-1) coordinate (mbrUL);
    \path (mbl) ++(-0.375,-0.5) coordinate[dot, functorF2] (mbMid) ++ (0,-0.5) coordinate[dot, functorF2] (bMid);
    \path (muNew) ++(-1.75,0) coordinate (bll2) ++(0,-3.125) coordinate (bll);
    \end{scope}
    \draw[functorM] (bl) -- (bl2) to[out=90, in=180] (bMid)
                    (mbMid) to[out=0, in=-90] (mbl) to[out=90, in=180] (mu.center) to[out=0, in=90] (mbr) -- (br)
                    (muNew) to[out=0, in=90] (mu);
    \draw[functorL] (bll) -- (bll2) to[out=90, in=180] (muUL)
                    (bMid) to[out=0, in=90] (brl2) -- (brl)
                    (muNew) to[out=180, in=180] (mbMid);
    \draw[functorF2] (muUL.center) to[out=0, in=90] (muNew)
                     (muUL) -- (dUL)
                     (mbMid) -- (bMid);
    \coordinate (tl) at (bll |- dUL);
    \coordinate (cornerNW) at ($(tl) + (-0.5,0)$);
    \coordinate (cornerSE) at ($(br) + (0.5,0)$);
    \draw (cornerNW) rectangle (cornerSE);
   \end{tikzpicture}
   \enspace=\enspace
   \begin{tikzpicture}[scale=1.0,baseline={([yshift=-0.5ex]current bounding box.center)}]
    \begin{scope}[on layer=over]
    \path coordinate (mu)
    +(-0.5,-0.5) coordinate (mbl)
    +(0,-0.5) coordinate (mbr);
    \path (mbl) ++(-1.5,-1.5) coordinate (bl2) ++(0,-0.25) coordinate (bl);
    \path (bl2) ++(0.75,0) coordinate (brl2) ++(0,-0.25) coordinate (brl);
    \path (mu) ++(-0.875,0.875) coordinate[dot, functorF2] (muNew);
    \path (muNew) ++(-0.875,0.875) coordinate[dot, functorF2] (muUL)
    +(0,0.5) coordinate (dUL)
    +(-1.5,-1) coordinate (mblUL)
    +(1.5,-1) coordinate (mbrUL);
    \path (muNew)
    +(-0.75,-0.75) coordinate[dot, functorF2] (mblNew)
    +(0.75,-0.75) coordinate (mbrNew);
    \path (mblNew) ++(0,-1.125) coordinate[dot, functorF2] (mbMid) ++ (0,-0.5) coordinate[dot, functorF2] (bMid);
    \path (mbrNew) ++(0,-2.375) coordinate (br);
    \path (muNew) ++(-1.75,0) coordinate (bll2) ++(0,-3.125) coordinate (bll);
    \end{scope}
    \draw[functorM] (bl) -- (bl2) to[out=90, in=180] (bMid)
                    (muNew) to[out=0, in=90] (mbrNew) -- (br)
                    (mbMid.center) to[bend right=50] (mblNew.center);
    \draw[functorL] (bll) -- (bll2) to[out=90, in=180] (muUL)
                    (bMid) to[out=0, in=90] (brl2) -- (brl)
                    (mbMid.center) to[bend left=50] (mblNew.center);
    \draw[functorF2] (muUL.center) to[out=0, in=90] (muNew)
                     (muUL) -- (dUL)
                     (mbMid) -- (bMid)
                     (muNew) to[out=180, in=90] (mblNew);
    \coordinate (tl) at (bll |- dUL);
    \coordinate (cornerNW) at ($(tl) + (-0.5,0)$);
    \coordinate (cornerSE) at ($(br) + (0.5,0)$);
    \draw (cornerNW) rectangle (cornerSE);
   \end{tikzpicture}
   \\[5pt] =\enspace
   \begin{tikzpicture}[scale=1.0,baseline={([yshift=-0.5ex]current bounding box.center)}]
    \begin{scope}[on layer=over]
    \path coordinate (mu)
    +(-0.5,-0.5) coordinate (mbl)
    +(0,-0.5) coordinate (mbr);
    \path (mbl) ++(-1.5,-0.375) coordinate (bl2) ++(0,-0.5) coordinate (bl);
    \path (bl2) ++(0.75,0) coordinate (brl2) ++(0,-0.5) coordinate (brl);
    \path (mu) ++(-0.875,0.875) coordinate[dot, functorF2] (muNew);
    \path (muNew) ++(-0.875,0.875) coordinate[dot, functorF2] (muUL)
    +(0,0.5) coordinate (dUL)
    +(-1.5,-1) coordinate (mblUL)
    +(1.5,-1) coordinate (mbrUL);
    \path (muNew)
    +(-0.75,-0.75) coordinate (mblNew)
    +(0.75,-0.75) coordinate (mbrNew);
    \path (mblNew) ++(0,-0.0) coordinate (mbMid) ++ (0,-0.5) coordinate[dot, functorF2] (bMid);
    \path (mbrNew) ++(0,-1.5) coordinate (br);
    \path (muNew) ++(-1.75,0) coordinate (bll2) ++(0,-2.25) coordinate (bll);
    \end{scope}
    \draw[functorM] (bl) -- (bl2) to[out=90, in=180] (bMid)
                    (muNew) to[out=0, in=90] (mbrNew) -- (br);
    \draw[functorL] (bll) -- (bll2) to[out=90, in=180] (muUL)
                    (bMid) to[out=0, in=90] (brl2) -- (brl);
    \draw[functorF2] (muUL.center) to[out=0, in=90] (muNew)
                     (muUL) -- (dUL)
                     (mbMid) -- (bMid)
                     (muNew) to[out=180, in=90] (mblNew)
                     (mbMid) -- (mblNew);
    \coordinate (tl) at (bll |- dUL);
    \coordinate (cornerNW) at ($(tl) + (-0.5,0)$);
    \coordinate (cornerSE) at ($(br) + (0.5,0)$);
    \draw (cornerNW) rectangle (cornerSE);
   \end{tikzpicture}
   &\enspace=\enspace 
   \begin{tikzpicture}[scale=1.0,baseline={([yshift=-0.5ex]current bounding box.center)}]
    \begin{scope}[on layer=over]
    \path coordinate[dot, functorF2] (mu)
    +(-0.75,-0.75) coordinate (mbl)
    +(0.75,-0.75) coordinate (mbr);
    \path (mbr) ++(0,-1.125) coordinate (br);
    \path (mbl) ++(-0.75,-1.125) coordinate (bl);
    \path (mu) ++(-2.25,0) coordinate[dot, functorF2] (muL)
    +(-0.75,-0.75) coordinate (mblL)
    +(0.75,-0.75) coordinate (mbrL);
    \path (mbrL) ++(0.75,-1.125) coordinate (brL);
    \path (mblL) ++(0,-1.125) coordinate (blL);
    \path ($(mu)!0.5!(muL)$) ++(0,1.25) coordinate[dot, functorF2] (muUL)
    +(0,0.5) coordinate (dUL)
    +(-1.5,-1) coordinate (mblUL)
    +(1.5,-1) coordinate (mbrUL);
    \end{scope}
    \draw[functorM] (mu.center) to[out=0, in=90] (mbr) -- (br)
                    (muL.center) to[out=0, in=90] (mbrL) -- (bl);
    \draw[functorL] (blL) -- (mblL) to[out=90, in=180] (muL.center)
                    (brL) -- (mbl) to[out=90, in=180] (mu.center);
    \draw[functorF2] (muL) to[out=90, in=180] (muUL.center) to[out=0, in=90] (mu)
                     (muUL) -- (dUL);
    \coordinate (tl) at (mblL |- dUL);
    \coordinate (cornerNW) at ($(tl) + (-0.5,0)$);
    \coordinate (cornerSE) at ($(br) + (0.5,0)$);
    \draw (cornerNW) rectangle (cornerSE);
   \end{tikzpicture}
  \end{align*}
  \endgroup
  
  Finally, the 2-dimensional naturality condition follows immediately from the naturality condition for the lax functor $P'$.
  
  Thus we have defined the components of $\lambda$. We claim that this is indeed a strict 2-natural transformation (i.e.\ it is a pseudonatural transformation where the 2-morphism components are identities).
  
  For each oplax transformation $\theta'\colon P' \to Q'$ between decomposable unitary lax bifunctors we require that $\lambda_{Q'} \circ \overline{K}T(\theta') = \theta' \circ \lambda_{P'}$. This is clear on 1-morphism components, since the 1-morphisms components of $\lambda_{P'},\lambda_{Q'}$ are identities.
  For the 2-morphism components, the desired equality follows immediately from the coherence of the oplax transformation $\theta'$ with respect to compositors.
  
  The 2-dimensional naturality condition for $\lambda$ becomes that for each $\theta', \zeta'\colon P' \to Q'$ and modification $\beth'\colon \theta' \to \zeta'$ we have $\lambda_{Q'} \overline{K}T(\beth') = \beth' \lambda_{P'}$. But this is clear since the 1-morphism components of $\lambda_{P'}$ and $\lambda_{Q'}$ are identities.
  Moreover, $\lambda$ automatically satisfies the unitor and compositor conditions, since it is a strict 2-natural transformation between strict 2-functors.
  Finally, the decomposability assumption gives that the (1-morphism) components of $\lambda$ are pseudonatural transformations and consequently isomorphisms.
  Hence, $\lambda$ itself is a 2-natural isomorphism, as required.
  
  Now we define the 2-natural isomorphism $\kappa\colon \Id \cong T\overline{K}$. For each distributive law $\sigma$ between unitary lax functors we must define a morphism of distributive laws $\kappa_\sigma\colon \sigma \to T\overline{K}(\sigma)$.
  Let $(P, \gamma, \iota)$ be the collated functor $\overline{K}(\sigma)$.
  Then $T\overline{K}(\sigma)$ is the distributive law between the families $(P(-,C))_C$ and $(P(B,-))_B$ given by $T\overline{K}(\sigma)_{f,g} = \gamma\inv_{(\id,g),(f,\id)} \gamma_{(f,\id),(\id,g)}$.
  We claim the component families of $\kappa_\sigma$ can be given by the canonical families $(\kappa^C)_C$ and $(\kappa^B)_B$ of oplax transformations defined in \cref{thm:lax_bifunctor}.
  
  To show this defines a valid morphism of distributive laws we first note that $\kappa^C_B = \kappa^B_C$ since both are the identity morphism on $P(B,C)$.
  We must now check that these families satisfy the Yang--Baxter equation from \cref{def:morphismofdist}.
  In our situation the desired equality becomes that depicted by the following string diagrams. Here we have used \cref{prp:unitary_collation_decomposable} to express the inverse of $\gamma'_{(\id,g),(f,\id)}$ in terms of the unitors of the original lax functors.
  \\ 
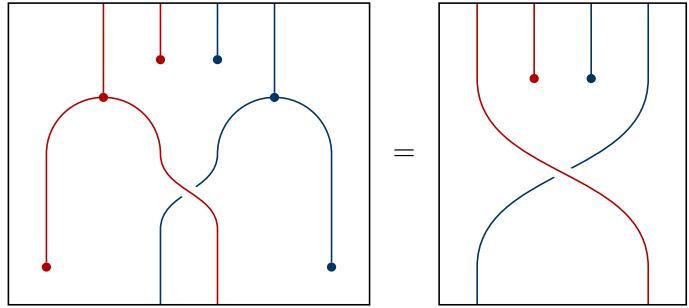
\vspace{-3pt plus 1pt}
  \begin{equation*}
   \begin{tikzpicture}[scale=1.0,baseline={([yshift=-0.5ex]current bounding box.center)}]
    \begin{scope}[on layer=over]
    \path coordinate[dot, functorM] (mu)
    +(-0.75,-0.75) coordinate (mbl)
    +(0.75,-0.75) coordinate (mbr);
    \path (mbl) ++(-0.75,-1.0) coordinate (bl2) ++(0,-1.0) coordinate (bl);
    \coordinate (br) at (mbr |- bl);
    \path (br) ++(0,0.5) coordinate[dot, functorM] (br1);
    \path (mu) ++(-2.25,0) coordinate[dot, functorL] (muL)
    +(-0.75,-0.75) coordinate (mblL)
    +(0.75,-0.75) coordinate (mbrL);
    \path (mbrL) ++(0.75,-1.0) coordinate (brL2) ++(0,-1.0) coordinate (brL);
    \path (mblL) ++(0,-1.5) coordinate[dot, functorL] (blL1) ++(0,-0.5) coordinate (blL);
    \path (mu) ++(0,0.5) coordinate (d)
    +(0,0.75) coordinate (d2)
    +(-0.75,0) coordinate[dot, functorM] (mtl);
    \path (muL) ++(0,0.5) coordinate (dL)
    +(0,0.75) coordinate (dL2)
    +(0.75,0) coordinate[dot, functorL] (mtlL);
    \end{scope}
    \draw[functorM] (bl) -- (bl2) to[out=90, in=-90] (mbl) to[out=90, in=180] (mu.center) to[out=0, in=90] (mbr) -- (br1)
                    (mu) -- (d2);
    \draw[functorL, cross line] (blL1) -- (mblL) to[out=90, in=180] (muL.center) to[out=0, in=90] (mbrL) to[out=-90, in=90] (brL2) -- (brL)
                                (muL) -- (dL2);
    \draw[functorM] (mtl) -- (d2 -| mtl);
    \draw[functorL] (mtlL) -- (d2 -| mtlL);
    \coordinate (tl) at (d2 -| blL);
    \coordinate (cornerNW) at ($(tl) + (-0.5,0)$);
    \coordinate (cornerSE) at ($(br) + (0.5,0)$);
    \draw (cornerNW) rectangle (cornerSE);
   \end{tikzpicture}
   \enspace=\enspace
   \begin{tikzpicture}[scale=1.0,baseline={([yshift=-0.5ex]current bounding box.center)}]
    \begin{scope}[on layer=over]
    \path coordinate (mu);
    \path (mu) ++(-2.25,-2.5) coordinate (bl2) ++(0,-0.5) coordinate (bl);
    \path (mu) ++(-2.25,0) coordinate (muL);
    \path (muL) ++(2.25,-2.5) coordinate (brL2) ++(0,-0.5) coordinate (brL);
    \path (mu) ++(0,1.0) coordinate (d2);
    \path (mu) ++(-0.75,0.0) coordinate[dot, functorM] (mtl);
    \path (muL) ++(0,1.0) coordinate (dL2);
    \path (muL) ++(0.75,0.0) coordinate[dot, functorL] (mtlL);
    \end{scope}
    \draw[functorM] (bl) -- (bl2) to[out=90, in=-90] (mu.center)
                    (mu) -- (d2);
    \draw[functorL, cross line] (muL.center) to[out=-90, in=90] (brL2) -- (brL)
                                (muL) -- (dL2);
    \draw[functorM] (mtl) -- (d2 -| mtl);
    \draw[functorL] (mtlL) -- (d2 -| mtlL);
    \coordinate (cornerNW) at ($(dL2) + (-0.5,0)$);
    \coordinate (cornerSE) at ($(brL) + (0.5,0)$);
    \draw (cornerNW) rectangle (cornerSE);
   \end{tikzpicture}
  \end{equation*}
  This equality is immediately seen to hold by the unit laws.
  
  We now check that $\kappa$ is 2-natural transformation. The 1-dimensional naturality condition for $\kappa$ requires that for each 1-morphism of distributive laws $\overline{\theta}\colon \sigma^1 \to \sigma^2$ we have $\kappa_{\sigma^2} \circ \overline{\theta} = T\overline{K}(\overline{\theta}) \circ \kappa_{\sigma^1}$, which is an immediate consequence of \cref{prp:collate_oplax_transformations}.
  Similarly, the 2-dimensional naturality condition asks that $\kappa_{\sigma^2}\,\overline{\beth} = T\overline{K}(\overline{\beth})\,\kappa_{\sigma^1}$ for each 2-morphism of distributive laws $\overline{\beth}\colon \overline{\theta} \to \overline{\zeta}$, but this is immediate by inspection since the 1-morphism components of $\kappa_{\sigma^1}$ and $\kappa_{\sigma^2}$ are identity morphisms.
  Moreover, the compositor and unitor conditions again hold automatically as above.
  
  Finally, note that each component of $\kappa$ is an isomorphism since the unitors are invertible by assumption. Hence we have proved the result.
\end{proof}

\begin{remark}
At the end of \cref{sec:uncurrying} we discussed how collation corresponds to the uncurrying 2-functor $J$ composed with the equivalence between $\Dist(\B,\C\,\D)$ and $\Laxop(\B,\Laxop(\C,\D))$ defined in \cref{thm:equivalence_between_dist_and_iterated_functor_cat}.
If we denote the category of unitary lax bifunctors from $\C$ to $\D$ by $\mathrm{ULax}_\mathrm{op}(\C, \D)$, then as a consequence of the above results we find that $J$ restricts to an equivalence $\overline{J}$ between $\mathrm{ULax}_\mathrm{op}(\B,\mathrm{ULax}_\mathrm{op}(\C,\D))$ and the 2-category of decomposable unitary lax bifunctors. The `inverse' functor then allows us to \emph{curry} this class of lax bifunctors.
\end{remark}

\section{Special cases}

The collation of lax functors restricts to give a number of important constructions in category theory.
We have discussed the example of distributive laws of monads throughout this paper.

Also notice that because pseudofunctors are unitary lax functors we have the pseudo-bifunctor theorem mentioned in the introduction as a corollary of \cref{thm:normal_lax_bifunctor}.

\begin{theorem}\label{thm:pseudo_bifunctor}
  Let $\sigma \in \Dist(\B,\C,\D)$ be an \emph{invertible} distributive law between families of pseudofunctors $L$ and $M$.
  The collation of these yields a pseudo-bifunctor $(P,\gamma,\iota)\colon \B \times \C \to \D$ such that $P(B,-)$ is canonically isomorphic to $M_B$ for each $B \in \B$ and $P(-,C)$ is canonically isomorphic to $L_C$ for each $C \in \C$.
  
  Conversely, let $P' = (P', \gamma', \iota') \colon \B \times \C \to \D$ be a pseudo-bifunctor. Then setting $L_C = P'(-,C)$ and $M_B = P'(B,-)$, we have a distributive law $\sigma_{f,g} = \gamma^{\prime\, -1}_{(\id,g),(f,\id)} \gamma'_{(f,\id),(\id,g)}$ and the resulting collated pseudo-bifunctor $P$ is pseudonaturally isomorphic to $P'$.
  
  Moreover, the equivalence of \cref{thm:normal_lax_bifunctor} restricts to an equivalence between the sub-2-category of invertible distributive laws of pseudofunctors and that of pseudo-bifunctors.
\end{theorem}

In fact, this pseudofunctor result can be slightly improved, since only axioms \ref{eq:distcomp1}, \ref{eq:distcomp2}, \ref{eq:distnat1} and \ref{eq:distnat2} need to be checked.
\begin{proposition}\label{prop:sigma_invertible_implies_unit}
  If the morphism $\sigma$ in a distributive law is invertible, then the unit axioms \ref{eq:distunit1} and \ref{eq:distunit2} follow from the other axioms.
\end{proposition}
\begin{proof} 
  Observe the following sequence of string diagrams.
  \\ \vspace{-3pt plus 1pt}
  \begingroup
  \allowdisplaybreaks
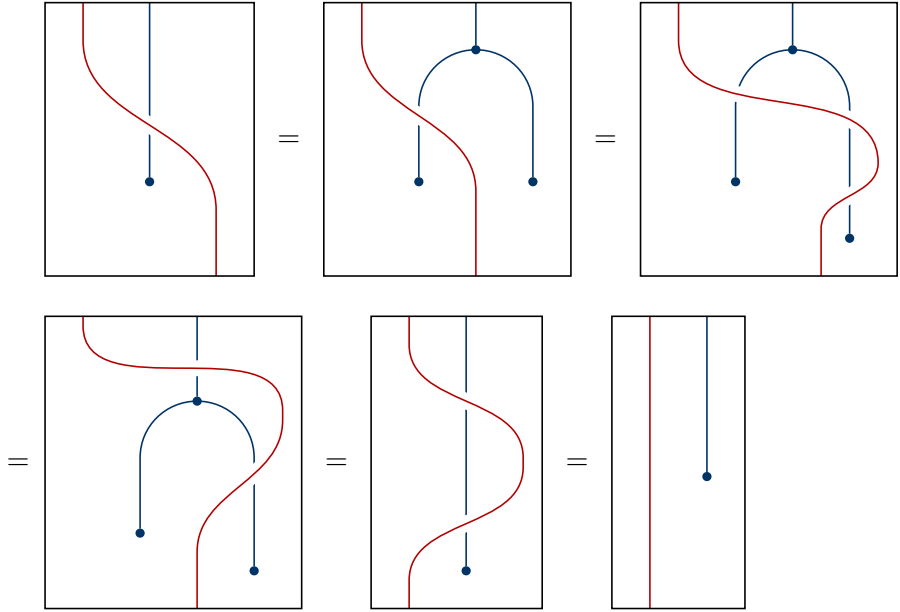
  \begin{align*}
   &\begin{tikzpicture}[scale=1.0,baseline={([yshift=-0.5ex]current bounding box.center)}]
    \begin{scope}[on layer=over]
    \path coordinate (mu) ++(0,0.5) coordinate (d);
    \path (d) ++(1.75,-2.75) coordinate (bl2) ++(0,-0.875) coordinate (bl);
    \path (d) ++(0.875,-1.375) coordinate (mbrL);
    \path (mbrL) ++(0,-1.0) coordinate[dot, functorM] (brL);
    \end{scope}
    \draw[functorM] (mbrL |- d) -- (brL);
    \draw[functorL, cross line] (bl) -- (bl2) to[out=90, in=-90]  (mu) -- (d);
    \coordinate (tl) at (d -| bl);
    \coordinate (bbr) at (bl -| mu);
    \coordinate (cornerNW) at ($(tl) + (0.5,0)$);
    \coordinate (cornerSE) at ($(bbr) + (-0.5,0)$);
    \draw (cornerNW) rectangle (cornerSE);
   \end{tikzpicture}
   \enspace=\enspace
   \begin{tikzpicture}[scale=1.0,baseline={([yshift=-0.5ex]current bounding box.center)}]
    \begin{scope}[on layer=over]
    \path coordinate (mu)
    +(0,0.5) coordinate (d);
    \path (mu) ++(1.5,-2.0) coordinate (bl2) ++(0,-1.125) coordinate (bl);
    \path (d) ++(1.5,0) coordinate (dL) ++ (0,-0.625) coordinate[dot, functorM] (muL)
    +(0.75,-0.75) coordinate (mblL)
    +(-0.75,-0.75) coordinate (mbrL);
    \path (mbrL) ++(0,-1.0) coordinate[dot, functorM] (brL);
    \path (mblL) ++(0,-1.0) coordinate[dot, functorM] (blL2) ++(0,-1.0) coordinate (blL);
    \end{scope}
    \draw[functorM] (blL2) -- (mblL) to[out=90, in=0] (muL.center) to[out=180, in=90] (mbrL) -- (brL)
                    (muL) -- (dL);
    \draw[functorL, cross line] (bl) -- (bl2) to[out=90, in=-90] (mu) -- (d);
    \coordinate (tl) at (dL -| blL);
    \coordinate (bbr) at (bl -| mu);
    \coordinate (cornerNW) at ($(tl) + (0.5,0)$);
    \coordinate (cornerSE) at ($(bbr) + (-0.5,0)$);
    \draw (cornerNW) rectangle (cornerSE);
   \end{tikzpicture}
   \enspace=\enspace
   \begin{tikzpicture}[scale=1.0,baseline={([yshift=-0.5ex]current bounding box.center)}]
    \begin{scope}[on layer=over]
    \path coordinate (mu)
    +(0,0.5) coordinate (d);
    \path (mu) ++(2.625,-1.625) coordinate (bl3) ++(-0.75,-0.875) coordinate (bl2) ++(0,-0.625) coordinate (bl);
    \path (d) ++(1.5,0) coordinate (dL) ++ (0,-0.625) coordinate[dot, functorM] (muL)
    +(0.75,-0.75) coordinate (mblL)
    +(-0.75,-0.75) coordinate (mbrL);
    \path (mbrL) ++(0,-1.0) coordinate[dot, functorM] (brL);
    \path (mblL) ++(0,-1.75) coordinate[dot, functorM] (blL);
    \end{scope}
    \draw[functorM] (blL) -- (mblL) to[out=90, in=0] (muL.center) to[out=180, in=90] (mbrL) -- (brL)
                    (muL) -- (dL);
    \draw[functorL, cross line] (bl) -- (bl2) to[out=90, in=-90] (bl3) to[out=90, in=-90]  (mu) -- (d);
    \coordinate (tl) at (dL -| bl3);
    \coordinate (bbr) at (bl -| mu);
    \coordinate (cornerNW) at ($(tl) + (0.25,0)$);
    \coordinate (cornerSE) at ($(bbr) + (-0.5,0)$);
    \draw (cornerNW) rectangle (cornerSE);
   \end{tikzpicture}
   \\[5pt] =\enspace&
   \begin{tikzpicture}[scale=1.0,baseline={([yshift=-0.5ex]current bounding box.center)}]
    \begin{scope}[on layer=over]
    \path coordinate (mu)
    +(0,0.125) coordinate (d);
    \path (mu) ++(2.625,-1.125) coordinate (bl3) ++ (0,-0.125) coordinate (bl32) ++(-1.125,-1.75) coordinate (bl2) ++(0,-0.75) coordinate (bl);
    \path (d) ++(1.5,0) coordinate (dL) ++ (0,-1.125) coordinate[dot, functorM] (muL)
    +(0.75,-0.75) coordinate (mblL)
    +(-0.75,-0.75) coordinate (mbrL);
    \path (mbrL) ++(0,-1.0) coordinate[dot, functorM] (brL);
    \path (mblL) ++(0,-1.5) coordinate[dot, functorM] (blL2) ++(0,-0.5) coordinate (blL);
    \end{scope}
    \draw[functorM] (blL2) -- (mblL) to[out=90, in=0] (muL.center) to[out=180, in=90] (mbrL) -- (brL)
                    (muL) -- (dL);
    \draw[functorL, cross line] (bl) -- (bl2) to[out=90, in=-90] (bl32) -- (bl3) to[out=90, in=-90]  (mu) -- (d);
    \coordinate (tl) at (dL -| bl3);
    \coordinate (bbr) at (bl -| mu);
    \coordinate (cornerNW) at ($(tl) + (0.25,0)$);
    \coordinate (cornerSE) at ($(bbr) + (-0.5,0)$);
    \draw (cornerNW) rectangle (cornerSE);
   \end{tikzpicture}
   \enspace=\enspace
   \begin{tikzpicture}[scale=1.0,baseline={([yshift=-0.5ex]current bounding box.center)}]
    \begin{scope}[on layer=over]
    \path coordinate (mu)
    +(0,0.375) coordinate (d);
    \path (mu) ++(1.5,-1.5) coordinate (bl3) ++ (0,-0.125) coordinate (bl32) ++(-1.5,-1.5) coordinate (bl2) ++(0,-0.375) coordinate (bl);
    \path (d) ++(0.75,0) coordinate (dL) ++ (0,-1.125) coordinate (muL);
    \path (muL) ++(0,-2.25) coordinate[dot, functorM] (blL2) ++(0,-0.5) coordinate (blL);
    \end{scope}
    \draw[functorM] (blL2) to[out=90, in=-90] (muL) -- (dL);
    \draw[functorL, cross line] (bl) -- (bl2) to[out=90, in=-90] (bl32) -- (bl3) to[out=90, in=-90]  (mu) -- (d);
    \coordinate (tl) at (dL -| bl3);
    \coordinate (bbr) at (bl -| mu);
    \coordinate (cornerNW) at ($(tl) + (0.25,0)$);
    \coordinate (cornerSE) at ($(bbr) + (-0.5,0)$);
    \draw (cornerNW) rectangle (cornerSE);
   \end{tikzpicture}
   \enspace=\enspace
   \begin{tikzpicture}[scale=1.0,baseline={([yshift=-0.5ex]current bounding box.center)}]
    \begin{scope}[on layer=over]
    \path coordinate (mu)
    +(0,0.375) coordinate (d);
    \path (mu) (0,-3.5) coordinate (bl);
    \path (d) ++(0.75,0) coordinate (dL) ++ (0,-1.125) coordinate (muL);
    \path (muL) ++(0,-1.0) coordinate[dot, functorM] (blL2) ++(0,-1.75) coordinate (blL);
    \end{scope}
    \draw[functorM] (blL2) to[out=90, in=-90] (muL) -- (dL);
    \draw[functorL, cross line] (bl) -- (mu) -- (d);
    \coordinate (tl) at (dL);
    \coordinate (bbr) at (bl -| mu);
    \coordinate (cornerNW) at ($(tl) + (0.5,0)$);
    \coordinate (cornerSE) at ($(bbr) + (-0.5,0)$);
    \draw (cornerNW) rectangle (cornerSE);
   \end{tikzpicture}
  \end{align*}
  \endgroup
  The other unit axiom is dual.
\end{proof}

\begin{remark}
 We have stated the results of this paper in terms of 2-categories for simplicity, but we believe they will still hold for bicategories.
 In the case of pseudofunctors it is easy to deduce this more general result from the result for 2-categories and the fact that every bicategory is biequivalent to a strict 2-category, since pseudofunctors respect biequivalence.
\end{remark}

Distributive laws of pseudofunctors generalise braidings of monoidal categories.
We can now recover the following result from \cite{joyal1986braided} as a special case of \cref{thm:pseudo_bifunctor}.
\begin{proposition}
  Let $\C$ be a monoidal category. Braidings on $\C$ are in bijection with strong monoidal structures on the tensor product ${\otimes}\colon \C \times \C \to \C$. In particular, $\C$ admits a braiding if and only if $\otimes$ admits the structure of a strong monoidal functor.
\end{proposition}
\begin{proof}
  We simply view $\C$ as a one-object bicategory and apply \cref{thm:pseudo_bifunctor} with $L = M = \Id_\C$ and $\sigma$ given by the braiding. We need only check conditions \ref{eq:distcomp1},  \ref{eq:distcomp2}, \ref{eq:distnat1} and \ref{eq:distnat2} by \cref{prop:sigma_invertible_implies_unit}.
  Naturality of the braiding gives conditions \ref{eq:distnat1} and \ref{eq:distnat2}, while \ref{eq:distcomp1} and \ref{eq:distcomp2} follow from the braiding axioms.
\end{proof}

There are also some other possible applications which we now briefly mention, but do not explore in detail.

Recall that a $V$-enriched category can be described as a lax functor from its class of objects, viewed as an \emph{indiscrete} category, to the one-object bicategory corresponding to $V$. Here it is interesting to restrict to distributive laws between constant families of enriched categories. Then collation gives a product-like operation on $V$-enriched categories generalising the tensor products of $V$-enriched categories defined in \cite{joyal1986braided} when $V$ is equipped with a braiding.

It is also possible to apply the construction to \emph{graded} monads (lax functors from one-object bicategories which have found application in computer science \cite{fujii2016towards}). This provides a way to combine an $M$-graded monad and an $N$-graded monad to give an $(M \times N)$-graded monad. We are as yet unsure of the significance of this construction.

It is shown in \cite{rosebrugh2002distributive} that distributive laws of monads in the bicategory of spans correspond to \emph{strict factorisation systems}. On the other hand, general lax functors from $\C$ into $\mathrm{Span}$ or strictly unitary lax functors into $\mathrm{Prof}$ correspond to functors into $\C$ by a version of the Grothendieck construction due to Bénabou.
It could be interesting to explore what distributive laws of lax functors give in this setting and what our construction means in terms of these associated functors.

\bibliographystyle{abbrv}
\bibliography{bibliography}
\end{document}